\documentclass{amsart}
\usepackage{amssymb,amscd}
\usepackage{graphicx}
\usepackage[all]{xy}
\usepackage{amsmath}
\usepackage{amsfonts}
\usepackage{amsthm}
\usepackage{euscript}
\usepackage{tikz,subcaption} 
\usepackage{color}
\usepackage{epstopdf}

\newtheorem{theorem}{Theorem}[section]
\newtheorem{problem}[theorem]{Problem}
\newtheorem{proposition}[theorem]{Proposition}
\newtheorem{lemma}[theorem]{Lemma}
\newtheorem{corollary}[theorem]{Corollary}

\theoremstyle{definition}
\newtheorem{definition}[theorem]{Definition}
\newtheorem{example}[theorem]{Example}
\newtheorem{construction}[theorem]{Construction}

\theoremstyle{remark}
\newtheorem*{remark}{Remark}

\numberwithin{equation}{section}

\renewcommand{\phi}{\varphi}

\def\C{\mathbb C}

\def\Q{\mathbb Q}
\def\R{\mathbb R}

\def\Z{\mathbb Z}

\def\sK{\mathcal K}

\def\k{\mathbf k}

\newcommand{\mb}[1]{{\textbf {\textit#1}}}

\renewcommand{\ge}{\geqslant}
\renewcommand{\le}{\leqslant}

\def\Ann{\mathop{\mathrm{Ann}}\nolimits}

\newcommand{\Ker}{\mathop{\rm Ker}}
\def\mdeg{\mathop{\mathrm{mdeg}}}

\newcommand{\Tor}{\mathop{\rm Tor}\nolimits}

\def\lim{\mathop\mathrm{lim}\nolimits}

\newcommand{\lk}{\mathop{\rm link}\nolimits}
\newcommand{\st}{\mathop{\rm star}\nolimits}

\newcommand{\zk}{\mathcal Z_{\mathcal K}}
\newcommand{\zp}{\mathcal Z_P}

\begin{document}

\title[Cohomological rigidity of
manifolds defined by 3-polytopes]{Cohomological rigidity of
manifolds defined by right-angled 3-dimensional polytopes}

\author[V. Buchstaber]{Victor Buchstaber}
\address{Steklov Mathematical Institute, Russian Academy of
Sciences, Gubkina str.~8, 119991 Moscow, RUSSIA,
\newline\indent
Department of Mathematics and Mechanics, Moscow State University,
\quad \emph{and}
\newline\indent Institute for Information Transmission Problems,
Russian Academy of Sciences} \email{buchstab@mi.ras.ru}

\author[N. Erokhovets]{Nikolay Erokhovets}
\address{Department of Mathematics and Mechanics, Moscow
State University, Leninskie Gory, 119991 Moscow, RUSSIA}
\email{erochovetsn@hotmail.com}

\author[M. Masuda]{Mikiya Masuda}
\address{Department of Mathematics, Osaka City University,
3-3-138 Sugimoto, Sumiyoshi-ku, Osaka-shi 558-8585, JAPAN}
\email{masuda@sci.osaka-cu.ac.jp}

\author[T. Panov]{Taras Panov}
\address{Department of Mathematics and Mechanics, Moscow
State University, Leninskie Gory, 119991 Moscow, RUSSIA,
\newline\indent Institute for Theoretical and Experimental Physics,
Moscow, \quad \emph{and}
\newline\indent Institute for Information Transmission Problems,
Russian Academy of Sciences} \email{tpanov@mech.math.msu.su}
\urladdr{http://higeom.math.msu.su/people/taras/}

\author[S. Park]{Seonjeong Park}
\address{Osaka City University Advanced Mathematical Institute (OCAMI), 3-3-138 Sugimoto, Sumiyoshi-ku
Osaka-shi 558-8585, JAPAN} \email{seonjeong1124@gmail.com}

\thanks{The research of the first, second and fourth author was
supported by the Russian Foundation for Basic Research (grants
RFBR 17-01-00671 and 16-51-55017-GFEN). The second author was
supported by the Young Russian Mathematics Award. The third author
was supported by JSPS Grant-in-Aid for Scientific Research (C)
16K05152.}

\subjclass[2010]{Primary: 57R91, 57M50; secondary: 05C15, 14M25,
52A55, 52B10}

\begin{abstract}
A family of closed manifolds is called cohomologically rigid if a cohomology ring isomorphism implies a diffeomorphism for any two manifolds in the family. We establish cohomological rigidity for large families of 3-dimensional and 6-dimensional manifolds defined by 3-dimensional polytopes.

We consider the class $\mathcal P$ of 3-dimensional combinatorial simple polytopes~$P$, different from a tetrahedron, whose facets do not form 3- and 4-belts. This class includes mathematical fullerenes, i.\,e. simple 3-polytopes with only 5-gonal and
6-gonal facets. By a theorem of Pogorelov, any polytope from $\mathcal P$ admits
a right-angled realisation in Lobachevsky 3-space, which is unique up to isometry.

Our families of smooth manifolds are associated with polytopes from the class~$\mathcal P$. The first family consists of 3-dimensional small covers of polytopes from $\mathcal P$, or hyperbolic 3-manifolds of L\"obell type. The second family consists of 6-dimensional quasitoric manifolds over polytopes from~$\mathcal P$.  Our main result is that both families are cohomologically rigid,
i.\,e. two manifolds $M$ and $M'$ from either of the families are diffeomorphic if and
only if their cohomology rings are isomorphic. We also prove that if $M$ and $M'$
are diffeomorphic, then their corresponding polytopes $P$ and $P'$
are combinatorially equivalent. These results are intertwined with the classical
subjects of geometry and topology, such as combinatorics of 3-polytopes, the Four Colour Theorem, aspherical manifolds, diffeomorphism classification of 6-manifolds and invariance of Pontryagin classes.  The proofs use techniques of toric topology.
\end{abstract}

\maketitle

\tableofcontents

\section{Introduction}
The following naive question goes back to the early days of
differential topology: given two closed smooth manifolds $M$ and
$M'$, when does an isomorphism $H^*(M)\cong H^*(M')$ of integral
cohomology rings imply that $M$ and $M'$ are diffeomorphic? This
is generally regarded as an unlikely case, as in the 20th century
topologists discovered many important series of manifolds for
which the cohomology ring, or even the homotopy type, does not
determine the diffeomorphism class. Three-dimensional lens spaces,
Milnor's exotic spheres and Donaldson's four-dimensional manifolds
are prominent examples of different level of complexity. Many
interesting examples appear in dimension~$6$, which is given a
special attention in our work. There is a family of ``fake''
complex projective $3$-spaces, i.\,e. simply connected smooth
$6$-manifolds whose cohomology rings are isomorphic to that of~$\C
P^3$. Such manifolds are homotopy equivalent to~$\C P^3$, but not
pairwise diffeomorphic in general.

We say that a family of closed smooth manifolds is \emph{cohomologically rigid} if a cohomology ring isomorphism $H^*(M)\cong H^*(M')$
implies a diffeomorphism $M\cong M'$ for any two manifolds in the family.

In this paper we establish cohomological rigidity for two
particular families of manifolds of dimension $3$ and $6$,
respectively. Each of these families arises from an important
class of combinatorial polytopes, which we refer to as the
\emph{Pogorelov class}~$\mathcal P$. It consists of simple
$3$-dimensional polytopes which are flag and do not have $4$-belts
of facets. In particular, polytopes in $\mathcal P$ do not have
triangular and quadrangular facets. The class $\mathcal P$
includes all mathematical fullerenes, i.\,e. simple $3$-polytopes
with only pentagonal and hexagonal facets. Mathematical fullerenes
are particularly interesting as they provide models for physical
fullerenes, i.\,e. molecules of carbon, whose discovery was
awarded with the Nobel Prize in Chemistry in~1996~\cite{bu-er}.

By the results of Pogorelov~\cite{pogo67} and Andreev~\cite{andr70}, the class $\mathcal P$ coincides with the class of combinatorial $3$-polytopes which can be realised in Lobachevsky (hyperbolic) space $\mathbb L^3$ with right angles between adjacent facets (\emph{right-angled $3$-polytopes} for short).

The conditions specifying the Pogorelov class $\mathcal P$ also feature as the ``no-$\triangle$'' and ``no-$\square$'' conditions in Gromov's construction~\cite{grom87} of piecewise Euclidean cubical spaces of non-positive curvature. The latter is defined via the comparison inequality of Alexandrov--Toponogov (the so-called $\mathrm{CAT}(0)$-\emph{inequality}). Gromov proved that non-positivity of the curvature (in the $\mathrm{CAT}(0)$ sense) is equivalent to the no-$\triangle$-condition (the absence of $3$-belts in the dual polytope), while the no-$\square$-condition (the absence of $4$-belts) implies that the curvature is strictly negative. As pointed out in~\cite[\S4.6]{grom87}, the barycentric subdivision of every polytope satisfies the no-$\triangle$-condition, but the no-$\square$ is harder to get. Thanks to fullerenes, we now have a large class of polytopes satisfying both conditions. It follows from the results of Thurston~\cite{thur98} that the number of combinatorially different fullerenes with $p_6$ hexagonal facets grows as~$p_6^9$. Furthermore, we show in Corollary~\ref{Pogpk2} that for any finite sequence of nonnegative integers $p_k$, $k\geqslant 7$, there exists a Pogorelov polytope whose number of $k$-gonal facets is~$p_k$.

Our first family consists of \emph{hyperbolic $3$-manifolds of L\"obell type}, studied by Vesnin in~\cite{vesn87}. They arise from right-angled realisations of polytopes from the Pogorelov class~$\mathcal P$ (see the details in Subsection~\ref{sechyp}). Each hyperbolic $3$-manifold $N$ of L\"obell type is composed of $8$ copies of a polytope $P\in\mathcal P$. Furthermore, $N$ is a branched covering of~$P$, a \emph{small cover} in the sense of Davis and Januszkiewicz~\cite{da-ja91}. We prove in Theorem~\ref{3cohrigreal} that two such manifolds $N$ and $N'$ are diffeomorphic (or isometric) if and only if their $\Z_2$-cohomology rings are isomorphic. Hyperbolic $3$-manifolds of L\"obell type are aspherical, and their fundamental groups are certain finite extensions of the commutator subgroups of hyperbolic right-angled reflection groups. Our cohomological rigidity result has a pure algebraic interpretation: the fundamental groups of $N$ are distinguished by their $\Z_2$-cohomology rings. Another example of this situation was studied in~\cite{ka-ma09}: it was proved there that the fundamental groups of small covers which admit a Riemannian flat metric (that is, small cover over $n$-cubes) are distinguished by their $\Z_2$-cohomology rings (see also~\cite{c-m-o16}).

In this regard, we note the following well-known problem: describe the class of groups realisable as fundamental groups of finite cell complexes. According to the conjecture of Arnold, Pham and Thom, this class contains all Artin groups (including those whose corresponding Coxeter group is infinite). In~\cite{ch-da95} this conjecture was proved for almost all Artin groups, including right-angled ones.

The second family arises from toric topology: it consists of quasitoric (or topological toric) manifolds
whose quotient polytopes are in the class $\mathcal P$. These are $6$-dimensional smooth manifolds acted on by a $3$-torus $T^3$ with quotient $P\in\mathcal P$. We show (in Theorem~\ref{3cohrig} and Corollary~\ref{coro:5-1}) that this family is cohomologically rigid, i.\,e. two manifolds $M$ and $M'$ in the family are diffeomorphic if and only if their cohomology rings are isomorphic. In general a non-equivariant diffeomorphism between quasitoric manifolds $M$ and $M'$ does not imply that the corresponding polytopes $P$ and $P'$ are combinatorially equivalent, but this is the case when the quotient polytopes are in the class~$\mathcal P$ (see Theorem~\ref{3cohrig}).

Our proofs use both combinatorial and cohomological techniques of toric topology. Namely, we reduce the $3$-dimensional statement (Theorem~\ref{3cohrigreal}) to the $6$-dimensional one (Theorem~\ref{3cohrig}) using the fact that the cohomology ring of a small cover and the cohomology ring of a quasitoric manifold (with coefficients in~$\Z_2$) have the same structure and differ only in grading. Then we raise the dimension even higher, by reducing the $6$-dimensional statement to certain cohomological properties of moment-angle manifolds of dimension $m+3$, where $m$ is the number of facets in the Pogorelov polytope. After reducing the statement to analysing the cohomology of moment-angle manifolds, we apply several nontrivial combinatorial and algebraic lemmata of Fan, Ma and Wang~\cite{f-m-w,fa-wa}, used in their proof of cohomological rigidity for moment-angle compexes of flag $2$-spheres without chordless cycles of length~$4$. Families of polytopes from the class $\mathcal P$ also feature in the works~\cite{bu-er,bu-erS} on combinatorial constructions of fullerenes.

The following question is still open: is the whole family of toric or topological toric manifolds cohomologically rigid? Surprisingly, no counterexamples to this ``toric cohomological rigidity problem'' have been found up to date. This question is linked to classical problems of classification of simply connected manifolds and cohomological invariance of Pontryagin characteristic classes.

In real dimension~$6$ the families of quasitoric and topological toric manifolds coincide and contain strictly the family of toric manifolds (smooth complete toric varieties).
The family of quasitoric (or topological toric) manifolds whose quotient polytopes are in the class $\mathcal P$ is large enough, as there is at least one quasitoric manifold over any simple $3$-polytope.
Indeed, the Four Colour Theorem implies that any simple $3$-polytope admits a ``characteristic function'' (see Proposition~\ref{4cprop}); this remarkable observation was made by Davis and Januszkiewicz in~\cite{da-ja91}. Algebraic toric manifolds whose associated polytopes are in $\mathcal P$ are fewer, but still abundant; many concrete examples were produced recently by Suyama~\cite{suya15}. However, there are no \emph{projective} toric manifolds among them. This follows from a result of Delaunay~\cite{dela05} that a Delzant $3$-polytope must have at least one triangular or quadrangular facet.

Our results on cohomological rigidity of toric manifolds chime
with the problem of diffeomorphism classification for simply
connected manifolds, which is a classical subject of algebraic and
differential topology. The foundations of this classification in
dimensions $\ge5$ were laid in the works of Browder and Novikov
(see~\cite{brow72}, \cite{novi96}). Novikov~\cite{novi64} showed
that for a given simply connected manifold~$M$ of dimension $\ge5$
there are only finitely many manifolds $M'$ for which there exists
a homotopy equivalence $M\stackrel\simeq\longrightarrow M'$
preserving the Pontryagin classes. The case of low dimensions
$5,6,7$ was also considered in~\cite{novi64}. In dimension $6$
appear first examples of manifolds whose rational Pontryagin
classes are not homotopy invariant. The following setting of the
classification problem is related to the question of cohomological
rigidity: under which additional assumptions an integer cohomology
ring isomorphism implies a diffeomorphism of manifolds? In this
setting, complete classification results in dimension $6$ were
obtained in the works of Wall~\cite{wall66}, Jupp~\cite{jupp73}
and Zhubr~\cite{zhub00}.

Toric, quasitoric or topological toric manifolds $M$ are simply connected, and their cohomology rings $H^*(M)$ are generated by $2$-dimensional classes. Two such manifolds of dimension $6$ are diffeomorphic if there is an isomorphism of their cohomology rings preserving the first Pontryagin class~$p_1$; this can be deduced from the classification result of Wall and Jupp using classical homotopy-theoretical techniques, see Section~\ref{class6}. Therefore, the toric cohomological rigidity problem in dimension $6$ reduces to establishing the invariance of $p_1$ under integer cohomology ring isomorphisms. This turns out to be a purely combinatorial and linear algebra problem. However, we were not able to prove directly the invariance of $p_1$ under cohomology isomorphisms for toric manifolds over simple $3$-polytopes from the class~$\mathcal P$.
One of our main results (Theorem~\ref{3cohrig}) can be interpreted as a classification result for a particular large family of simply connected $6$-dimensional manifolds. We note that our proof is independent of the general classification results of~\cite{wall66} and~\cite{jupp73}.

\medskip

\noindent\textbf{Acknowledgements}
The authors are grateful to
\begin{itemize}
\item[--]Feifei Fan for introducing us to the results of the works~\cite{f-m-w} and~\cite{fa-wa} on the
cohomological rigidity of moment-angle complexes;

\smallskip

\item[--]Alexandr Gaifullin for drawing our attention to Siebennman's  no-$\square$-condition in Gromov's theory of hyperbolic groups, which is equivalent to the absence of $4$-belts in the case of polytopes;

\smallskip

\item[--] Sergei Novikov for stimulating discussions on the diffeomorphism classification of simply connected manifolds and invariance of Pontryagin classes;

\smallskip

\item[--]Andrei Vesnin for his help in linking our results to hyperbolic geometry; in particular, to hyperbolic $3$-manifolds of L\"obell type and Pogorelov's theorem on right-angled $3$-polytopes.
\end{itemize}

We thank the Institute for Mathematical Science, National University of Singapore, and the organisers of the Program on Combinatorial and Toric Homotopy, where we were able to exchange ideas which laid the foundation of this work. We also thank the JSPS program for Advancing Strategic International Networks to Accelerate the Circulation of Talented Researchers based on Osaka City University Advanced Mathematical Institute (OCAMI),
which gave us an opportunity to meet and work together on this project.

\section{Preliminaries}
Here we collect the necessary information about toric varieties,
quasitoric manifolds and moment-angle manifolds; the details can
be found in~\cite{bu-pa15}. We also review small covers and hyperbolic manifolds here.

\subsection{Simple polytopes}
Let $\R^n$ be an $n$-dimensional Euclidean space with the scalar
product $\langle\;,\,\rangle$. A \emph{convex polytope} $P$ is a
nonempty bounded intersection of finitely many half-spaces in
some~$\R^n$:
\begin{equation}
\label{ptope}
  P=\bigl\{\mb x\in\R^n\colon\langle\mb a_i,\mb x\rangle+b_i\ge0\quad\text{for }
  i=1,\ldots,m\bigr\},
\end{equation}
where $\mb a_i\in\R^n$ and $b_i\in\R$. We often fix a presentation
by inequalities~\eqref{ptope} alongside with the polytope~$P$. We
assume that $P$ is $n$-dimensional, that is, the dimension of the
affine hull of $P$ is~$n$. We also assume that each inequality
$\langle\mb a_i,\mb x\rangle+b_i\ge0$ in~\eqref{ptope} is not
redundant, that is, cannot be removed without changing~$P$. Then
$P$ has $m$ \emph{facets} $F_1,\ldots,F_m$, where
\[
  F_i=\{\mb x\in P\colon\langle\mb a_i,\mb x\rangle+b_i=0\}.
\]
Each facet is a polytope of dimension $n-1$. A \emph{face} of $P$
is a nonempty intersection of facets. Zero-dimensional faces are
\emph{vertices}, and one-dimensional faces are \emph{edges}.

We refer to $n$-dimensional polytopes simply as
\emph{$n$-polytopes.}

Two polytopes $P$ and $Q$ are \emph{combinatorially equivalent} ($P\simeq Q$) if there is a
bijection between their faces preserving the inclusion relation. A
\emph{combinatorial polytope} is a class of combinatorially
equivalent polytopes.

We denote by $G_P$ the vertex-edge graph of a polytope~$P$, and
refer to it simply as the \emph{graph of~$P$}.
A graph is \emph{simple} if it has no loops and multiple edges. A
connected graph $G$ is \emph{$3$-connected} if it has at least $6$
edges and deletion of any one or two vertices with all incident
edges leaves $G$ connected. The following classical result
describes the graphs of $3$-polytopes.

\begin{theorem}[{Steinitz, see~\cite[Theorem~4.1]{zieg07}}]\label{S-Theorem}
A graph $G$ is the graph of a $3$-polytope if and only it is simple,
planar and $3$-connected.
\end{theorem}


An  $n$-polytope $P$ is \emph{simple} if exactly $n$ facets meet
at each vertex of~$P$. A simple polytope $P$ is called a
\emph{flag polytope} if every collection of its pairwise
intersecting facets has a nonempty intersection. An
\emph{$n$-simplex} $\varDelta^n$ is not flag for $n\ge2$. An
\emph{$n$-cube} $I^n$ is flag for any~$n$.

A \emph{$k$-belt} (or a \emph{prismatic $k$-circuit}) in a simple
$3$-polytope is a cyclic sequence $\mathcal
B_k=(F_{i_1},\dots,F_{i_k})$ of $k\ge 3$ facets in which pairs of
consecutive facets (including $F_{i_k},F_{i_1}$) are adjacent,
other pairs of facets do not intersect, and no three facets have a
common vertex.

A $3$-polytope $P$ with a triangular facet has a 3-belt around it,
unless $P=\varDelta^3$. A simple 3-polytope $P\ne\varDelta^3$ is
flag if and only if it does not contain 3-belts.

A \emph{fullerene} is a simple 3-polytope with only pentagonal and
hexagonal facets. A simple calculation with Euler characteristic
shows that the number of pentagonal facets in a fullerene is~$12$.
The number of hexagonal facets can be arbitrary except for $1$
(see \cite[Proposition~2]{d-s-s13}). Also, any fullerene is a flag
polytope without $4$-belts (see~\cite{dosl03}
and~\cite[Corollary~3.16]{bu-er}).

\subsection{Toric varieties and manifolds}
A \emph{toric variety} is a normal complex algebraic variety~$V$
containing an algebraic torus $(\C^\times)^n$ as a Zariski open
subset in such a way that the natural action of $(\C^\times)^n$ on
itself extends to an action on~$V$. We only consider nonsingular
complete (compact in the usual topology) toric varieties, also
known as \emph{toric manifolds}.

There is a bijective correspondence between the isomorphism
classes of complex $n$-dimensional toric manifolds and complete
nonsingular fans in~$\R^n$. A \emph{fan} is a finite collection
$\Sigma=\{\sigma_1,\ldots,\sigma_s\}$ of strongly convex polyhedral cones
$\sigma_i$ in $\R^n$ such that every face of a cone in $\Sigma$
belongs to $\Sigma$ and the intersection of any two cones in
$\Sigma$ is a face of each. A fan $\Sigma$ is \emph{nonsingular} (or \emph{regular}) if
each of its cones $\sigma_j$ is generated by part of a basis of
the lattice $\Z^n\subset\R^n$. Each one-dimensional cone of such $\Sigma$
is generated by a primitive vector $\mb a_i\in \Z^n$. A fan
$\Sigma$ is \emph{complete} if the union of its cones is the
whole~$\R^n$.

Projective toric varieties are particularly important. A
projective toric manifold $V$ is defined by a lattice Delzant
polytope~$P$. Given a simple $n$-polytope $P$ with the vertices in
the lattice~$\Z^n$, the \emph{normal fan} $\Sigma_P$ has one
$n$-dimensional cone $\sigma_v$ for each vertex $v$ of $P$, where
$\sigma_v$ is generated by the primitive inside-pointing normals
to the facets of $P$ meeting at~$v$. The polytope $P$ is
\emph{Delzant} whenever its normal fan $\Sigma_P$ is nonsingular.
The fan $\Sigma_P$ defines a projective toric manifold~$V_P$.
Different lattice Delzant polytopes with the same normal fan
produce different projective embeddings of the same toric
manifold.

Irreducible torus-invariant subvarieties of complex codimension
one in~$V$ correspond to one-dimensional cones of~$\Sigma$. When
$V$ is projective, they also correspond to the facets of~$P$. We
assume that there are $m$ one-dimensional cones (or facets),
denote the corresponding primitive vectors by $\mb a_1,\ldots,\mb
a_m$, and denote the corresponding codimension-one subvarieties
by~$V_1,\ldots,V_m$.

\begin{theorem}[{Danilov--Jurkiewicz, see~\cite[Theorem~5.3.1]{bu-pa15}}]\label{cohomtoric}
Let $V$ be a toric manifold of complex dimension~$n$ with the
corresponding complete nonsingular fan~$\Sigma$. The cohomology
ring $H^*(V;\Z)$ is generated by the degree-two classes $[v_i]$
dual to the invariant submanifolds $V_i$, and is given by
\[
  H^*(V;\Z)\cong \Z[v_1,\ldots,v_m]/\mathcal I,\qquad\deg v_i=2,
\]
where $\mathcal I$ is the ideal generated by elements of the
following two types:
\begin{itemize}
\item[(a)] $v_{i_1}\cdots v_{i_k}$ such that $\mb a_{i_1},\ldots,\mb
a_{i_k}$ do not span a cone of~$\Sigma$;
\item[(b)] $\displaystyle\sum_{i=1}^m\langle\mb a_i,\mb x\rangle v_i$, for
any vector $\mb x\in\Z^n$.
\end{itemize}
\end{theorem}

It is convenient to consider the integer $n\times m$-matrix
\begin{equation}\label{Lambdatoric}
  \varLambda=\begin{pmatrix}
  a_{11}&\cdots& a_{1m}\\
  \vdots&\ddots&\vdots\\
  a_{n1}&\cdots& a_{nm}
  \end{pmatrix}
\end{equation}
whose columns are the vectors $\mb a_i$ written in the standard
basis of~$\Z^n$. Then the ideal~(b) of Theorem~\ref{cohomtoric} is
generated by the $n$ linear forms $a_{j1}v_1+\cdots+a_{jm}v_m$
corresponding to the rows of~$\varLambda$.

The quotient of a projective toric manifold $V_P$ by the action of
the compact torus $T^n\subset(\C^\times)^n$ is the polytope~$P$.
When a toric manifold $V$ is not projective, the quotient $V/T^n$
has a face structure of a \emph{manifold with corners}. This face structure locally
looks like that of a simple convex polytope, but globally may fail to be so even combinatorially. In the case $n=3$, however, the quotient $V/T^3$ is combinatorially
equivalent to a simple $3$-polytope, by Steinitz's theorem (Theorem~\ref{S-Theorem}).

\subsection{Quasitoric manifolds}\label{qtmani}
In their 1991 work~\cite{da-ja91} Davis and Januszkiewicz
suggested a topological generalisation of projective toric
manifolds, which became known as quasitoric manifolds.

A \emph{quasitoric manifold} over a combinatorial simple
$n$-polytope $P$ is a topological manifold $M$ of dimension $2n$ with a
locally standard action of $T^n$ and a projection $\pi\colon
M\to P$ whose fibres are the orbits of the $T^n$-action. (An action of
$T^n$ on $M$ is \emph{locally standard} if every point $x\in
M$ is contained in a $T^n$-invariant neighbourhood
equivariantly homeomorphic to an open subset in $\C^n$ with a linear effective action of~$T^n$.
The orbit space of a locally standard torus action is a
manifold with corners. For a quasitoric manifold $M$, the orbit space $M/T^n$ is
homeomorphic to~$P$.)

Not every simple polytope can be the quotient of a quasitoric
manifold. Nevertheless, quasitoric manifolds constitute a much
larger family than projective toric manifolds, and enjoy more
flexibility for topological applications.

Let $\mathcal F=\{F_1,\ldots,F_m\}$ be the set of facets of $P$. Each
$M_i=\pi^{-1}(F_i)$ is a quasitoric submanifold of $M$ of
codimension~2, called a \emph{characteristic submanifold}. The
characteristic submanifolds $M_i\subset M$ are analogues of the
invariant divisors $V_i$ on a toric manifold~$V$. Each $M_i$ is
fixed pointwise by a closed one-dimensional subgroup (a subcircle)
$T_i\subset T^n$ and therefore corresponds to a primitive vector
$\lambda_i\in\Z^n$ defined up to a sign. Choosing a direction of
$\lambda_i$ is equivalent to choosing an orientation for the
normal bundle $\nu(M_i\subset M)$ or to choosing an
orientation for $M_i$, provided that $M$ itself is oriented. An
\emph{omniorientation} of a quasitoric manifold $M$ consists of a
choice of orientation for $M$ and each characteristic submanifold
$M_{i}$.

The vectors $\lambda_i$ are analogues of the generators $\mb a_i$
of the one-dimensional cones in the fan corresponding to a toric
manifold~$V$, or analogues of the normal vectors to the facets of
$P$ when $V$ is projective. However, the vectors $\lambda_i$ need not be
the normal vectors to the facets of $P$ in general.

There is an analogue of Theorem~\ref{cohomtoric} for quasitoric
manifolds:

\begin{theorem}[\cite{da-ja91}]\label{cohomqtoric}
Let $M$ be an omnioriented quasitoric manifold of dimension~$2n$
over a simple $n$-polytope~$P$. The cohomology ring $H^*(M;\Z)$ is generated
by the degree-two classes $[v_i]$ dual to the oriented
characteristic submanifolds $M_i$, and is given by
\[
  H^*(M;\Z)\cong \Z[v_1,\ldots,v_m]/\mathcal I,\qquad\deg v_i=2,
\]
where $\mathcal I$ is the ideal generated by elements of the
following two types:
\begin{itemize}
\item[(a)] $v_{i_1}\cdots v_{i_k}$ such that $F_{i_1}\cap\cdots\cap F_{i_k}=\varnothing$ in~$P$;
\item[(b)] $\displaystyle\sum_{i=1}^m\langle\lambda_i,\mb x\rangle v_i$, for
any vector $\mb x\in\Z^n$.
\end{itemize}
\end{theorem}

We record a simple corollary for the latter use.

\begin{corollary}\label{viprod}
In the notation of Theorem~\ref{cohomqtoric},
\begin{itemize}
\item[(a)] the product $[v_{i_1}]\cdots[v_{i_n}]$ of $n$ different
classes is a generator of $H^{2n}(M)\cong\Z$ if
$F_{i_1}\cap\cdots\cap F_{i_n}\ne\varnothing$ and is zero
otherwise;

\item[(b)] for $i\ne j$, we have $[v_i][v_j]=0$ if and only if $F_i\cap
F_j=\varnothing$.
\end{itemize}
\end{corollary}

By analogy with \eqref{Lambdatoric}, we consider the integer
\emph{characteristic matrix}
\begin{equation}\label{Lambdaqtoric}
  \varLambda=\begin{pmatrix}
  \lambda_{11}&\cdots& \lambda_{1m}\\
  \vdots&\ddots&\vdots\\
  \lambda_{n1}&\cdots& \lambda_{nm}
  \end{pmatrix}
\end{equation}
whose columns are the vectors $\lambda_i$ written in the standard
basis of~$\Z^n$. The matrix $\varLambda$ has the following property:
\begin{equation}\label{starprop}
  \det(\lambda_{i_1},\ldots,\lambda_{i_n})=\pm1\quad
  \text{whenever $F_{i_1}\cap\cdots\cap F_{i_n}\ne\varnothing$ in~$P$.}
\end{equation}

Note that the ideal~(b) of Theorem~\ref{cohomqtoric} is generated by the $n$
linear forms $\lambda_{j1}v_1+\cdots+\lambda_{jm}v_m$
corresponding to the rows of~$\varLambda$.

\medskip

A map $\lambda\colon \mathcal F\to\Z^n$, $F_i\mapsto\lambda_i$, satisfying~\eqref{starprop} is called a \emph{characteristic function} for a simple $n$-polytope~$P$. One can produce a characteristic matrix $\varLambda$ from a characteristic function $\lambda$ by fixing an ordering of facets.
A \emph{characteristic pair} $(P,\varLambda)$ consists of a simple polytope $P$ with a fixed ordering of facets and its characteristic matrix~$\varLambda$.

A quasitoric manifold $M$ defines a characteristic pair $(P,\varLambda)$. On the other hand, each characteristic pair gives rise to a quasitoric manifold as follows.

\begin{construction}[\cite{da-ja91}]\label{canqt}
Let $(P,\varLambda)$ be a characteristic pair.
For each facet $F_i$ of $P$ we denote by $T_i$ the circle subgroup of $T^n=\R^n/\Z^n$ corresponding to the $i$th column $\lambda_i\in\Z^n$ of the characteristic matrix~$\varLambda$.
For each point $x\in P$, define a torus
\[
  T(x)=\prod_{i\colon x\in F_i}T_i,
\]
assuming that $T(x)=\{1\}$ if there are no facets containing~$x$. Property~\eqref{starprop} implies that $T(x)$ embeds as a subgroup in~$T^n$. Then define
\[
  M(P,\varLambda)=P\times T^n/\!\sim,
\]
where the equivalence relation $\sim$ is given by $(x,t)\sim(x',t')$ whenever $x=x'$ and $t'-t\in T(x)$. One can see that $M(P,\varLambda)$ is a quasitoric manifold over~$P$.
\end{construction}

Changing the basis in the lattice results in
multiplying $\varLambda$ from the left by a matrix from
$\mbox{\textrm{GL}}(n,\Z)$. Changing the orientation of the
$i$th characteristic submanifold $M_i$ in the omniorientation data
results in changing the sign of the $i$th column of~$\varLambda$.
A combinatorial equivalence between polytopes $P$ and $P'$ allows us to identify their sets of facets $\mathcal F$ and $\mathcal F'$ and therefore identify their characteristic functions.
These observations lead us to the following definition.

{\samepage\begin{definition}\label{defcharpair}
Two characteristic pairs $(P,\varLambda)$ and $(P',\varLambda')$ are
\emph{equivalent} if
\begin{itemize}
\item[(a)] there is a combinatorial equivalence $P\simeq P'$, and
\item[(b)] $\varLambda'=A\varLambda B$, where $A\in \mbox{\textrm{GL}}(n,\Z)$ and
$B$ is a diagonal $(m\times m)$-matrix with $\pm1$ on the diagonal.
\end{itemize}
\end{definition}
}

Quasitoric manifolds
$M(P,\varLambda)$ and $M(P',\varLambda')$ corresponding to
equivalent pairs are equivariantly homeomorphic (in the weak sense). The
latter means that there is a homeomorphism $f\colon
M(P,\varLambda)\stackrel\cong\longrightarrow M(P',\varLambda')$
satisfying $f(t\cdot x)=\psi(t)\cdot f(x)$ for any $t\in T^n$ and
$x\in M(P,\varLambda)$, where $\psi\colon T^n\to T^n$ is the
automorphism of the torus given by the matrix~$A$. Furthermore, we have

\begin{proposition}[{\cite[Proposition~1.8]{da-ja91} and \cite[Proposition~7.3.8]{bu-pa15}}]
There is a one-to-one correspondence between equivariant homeomorphism classes of quasitoric manifolds and equivalence classes of characteristic pairs. In particular, for any quasitoric manifold $M$ over $P$ with characteristic matrix $\varLambda$, there is an equivariant homeomorphism $M\cong M(P,\varLambda)$.
\end{proposition}

\begin{remark}
Both $M$ and $M(P,\varLambda)$ were defined as topological manifolds in~\cite{da-ja91}. The manifold $M(P,\varLambda)$ can be endowed with a canonical smooth structure by defining it as the quotient of the moment-angle manifold $\zp$ by a smooth free torus action, see~\cite{b-p-r07} and Subsection~\ref{zpm}. Nevertheless, for a smooth quasitoric manifold $M$, the existence of a \emph{diffeomorphism} $M\cong M(P,\varLambda)$ is a delicate issue, see the discussion in~\cite[\S7.3]{bu-pa15}. On the other hand, in the case of $6$-dimensional quasitoric manifolds (which is our main concern in this paper), such a diffeomorphism follows from the classification results of Wall and Jupp discussed in Section~\ref{class6}.
\end{remark}

In dimensions $n\ge4$, there are simple $n$-polytopes $P$ which do
not admit any characteristic matrix $\varLambda$, see~\cite[1.22]{da-ja91}.
Such a polytope cannot be the quotient of a quasitoric manifold. On the other hand, we have the following observation by Davis and Januszkiewicz, whose
proof remarkably uses the Four Colour Theorem:

\begin{proposition}[\cite{da-ja91}]\label{4cprop}
Any simple $3$-polytope admits a characteristic matrix~$\varLambda$.
\end{proposition}
\begin{proof}
By the Four Colour Theorem, there is a \emph{regular $4$-colouring} of the facets of $P$, i.\,e. a map $\chi\colon\mathcal F\to\{1,2,3,4\}$ such that $\chi(F_i)\ne\chi(F_j)$ whenever $F_i\cap F_j\ne\varnothing$.
Given such a regular $4$-colouring, we assign to a
facet of $i$th colour the $i$th basis vector $\mb e_i\in\Z^3$ for $i=1,2,3$ and the vector
$\mb e_1+\mb e_2+\mb e_3$ for $i=4$. The resulting $3\times m$-matrix~$\varLambda$ satisfies~\eqref{starprop},
as any three of the four vectors $\mb e_1,\mb e_2,
\mb e_3,\mb e_1+\mb e_2+\mb e_3$ form a basis of $\Z^3$.
\end{proof}

A projective toric manifold is a quasitoric manifold. A
non-projective toric manifold $V$ may fail to be quasitoric, as the
quotient manifold with corners $V/T^n$ is not necessarily a simple
polytope, even combinatorially. First examples of this sort appear
in dimension $n=4$, see~\cite{suya14}. All complex 3-dimensional
toric manifolds, even non-projective ones, are quasitoric by the Steinitz theorem (Theorem~\ref{S-Theorem}).

\subsection{Small covers} Replacing the torus $T^n$ in the definition of a quasitoric manifold by the group $\Z_2^n\subset T^n$ (generated by $n$ commuting involutions), one obtains the definition of a small cover~\cite{da-ja91}. A \emph{small cover} of a simple
$n$-polytope $P$ is a manifold $N$ of dimension $n$ with a
locally standard action of $\Z_2^n$ and a projection $\pi\colon
N\to P$ whose fibres are the orbits of the $\Z_2^n$-action.

The set of real points of a projective toric manifold $V_P$
(i.\,e. the set of points fixed under the complex conjugation) is
a small cover of~$P$; it is sometimes called a \emph{real toric
manifold}.

The theory of small covers parallels that of quasitoric manifolds, and we just outline the most crucial points.

\begin{theorem}[\cite{da-ja91}]\label{cohomsc}
Let $N$ be a small cover of a simple $n$-polytope~$P$. The cohomology ring $H^*(N;\Z_2)$ is generated by the degree-one classes $[v_i]$ dual to the
characteristic submanifolds $N_i$, and is given by
\[
  H^*(N;\Z_2)\cong \Z_2[v_1,\ldots,v_m]/\mathcal I,\qquad\deg v_i=1,
\]
where $\mathcal I$ is the ideal generated by elements of the
following two types:
\begin{itemize}
\item[(a)] $v_{i_1}\cdots v_{i_k}$ such that $F_{i_1}\cap\cdots\cap F_{i_k}=\varnothing$ in~$P$;
\item[(b)] $\displaystyle\sum_{i=1}^m\langle\lambda_i,\mb x\rangle v_i$, for
any vector $\mb x\in\Z_2^n$.
\end{itemize}
\end{theorem}

The characteristic matrix $\varLambda$ corresponding to a small cover $N$ has entries in~$\Z_2$ and satisfies the same condition~\eqref{starprop}. The equivalence of $\Z_2$-characteristic pairs is defined
in the same way as in the quasitoric case, with $\mathrm{GL}(n,\Z)$ replaced by $\mathrm{GL}(n,\Z_2)$. A small cover $N$ of $P$ is equivariantly homeomorphic to the ``canonical model''
\[
  N(P,\varLambda)=P\times \Z_2^n/\!\sim
\]
with the equivalence relation $\sim$ defined as in the quasitoric case. Note that $N(P,\varLambda)$ is composed of $2^n$ copies of the polytope $P$, glued together along their facets.

Reducing a $\Z$-characteristic matrix $\mod 2$ we obtain a $\Z_2$-characteristic matrix. The following question is open:

\begin{problem}\label{mod2prob}
Assume given a $\Z_2$-characteristic pair $(P,\varLambda)$ consisting of a  simple $n$-polytope $P$ and an $(n\times m)$-matrix $\varLambda$ with entries in~$\Z_2$ satisfying~\eqref{starprop}. Can $\varLambda$ be obtained by reduction $\mod 2$ from an integer matrix satisfying the same condition~\eqref{starprop}?
\end{problem}

The answer to the above problem is positive for $3$-polytopes:

\begin{proposition}\label{lambda2}
For a simple $3$-polytope $P$, every $\Z_2$-characteristic pair
$(P,\varLambda)$ is the $\mod2$ reduction of a $\Z$-characteristic
pair.
\end{proposition}
\begin{proof}
It is enough to check that any $(3\times3)$-matrix with entries $0$ or $1$ and determinant $1\mod 2$  has determinant $\pm1$ when viewed as an integer matrix. Indeed, such a matrix either has a column with two zeros, or is
$\begin{pmatrix}1&0&1\\1&1&1\\0&1&1\end{pmatrix}$ up to permutation of rows and columns. The required property is then verified directly.
\end{proof}

\subsection{Right-angled polytopes and hyperbolic manifolds}\label{sechyp}
A particularly important class of $3$-dimensional small covers are \emph{hyperbolic $3$-manifolds of L{\"o}bell type}, introduced and studied by Vesnin in~\cite{vesn87}.

\begin{construction}
Let $P$ be a (compact) polytope in the $3$-dimensional Lobachevsky
space~$\mathbb L^3$ with right angles between adjacent facets (a
\emph{right-angled $3$-polytope} for short). It is easy to see
that a right-angled $3$-polytope is simple. Denote by $G(P)$ the
group generated by the reflections in the facets $F_1,\ldots,F_m$
of~$P$. It is a \emph{right-angled Coxeter group} given by the
presentation
\begin{equation}\label{racg}
  G(P)=\langle g_1,\ldots,g_m\;\vert\; g_i^2=1,\; g_ig_j=g_jg_i\;
  \text{if }F_i\cap F_j\ne\varnothing\rangle,
\end{equation}
where $g_i$ denotes the reflection in the facet~$F_i$. The reflections in adjacent facets commute because of the right-angledness. There are no relations
between the reflections in non-adjacent faces, as the corresponding reflection hyperplanes do not intersect in~$\mathbb L^3$.

The group $G(P)$ acts on $\mathbb L^3$ discretely with finite
isotropy subgroups and with the fundamental domain~$P$. Vertices
$v$ of reflection copies of $P$ have maximal isotropy subgroups,
isomorphic to~$\Z_2^3$ and generated by the reflections in the
three facets meeting at~$v$. This implies the following result.

\begin{lemma}[{\cite{vesn87}}]\label{vesninlemma}
Consider an epimorphism $\varphi^{(k)}\colon G(P)\to\Z_2^k$ for some~$k$. Its kernel $\Ker \varphi^{(k)}\subset G(P)$ does not contain elements of finite
order if and only if the images of the reflections in any three facets of $P$ that have a
common vertex are linearly independent in~$\Z_2^k$. In this case, the group $\Ker\varphi^{(k)}$ acts freely on~$\mathbb L^3$.
\end{lemma}

If $\varphi^{(k)}\colon G(P)\to\Z_2^k$ satisfies the condition of Lemma~\ref{vesninlemma}, then the quotient $N=\mathbb L^3/\Ker\varphi^{(k)}$ is a closed \emph{hyperbolic $3$-manifold}. It is composed of $|\Z_2^k|=2^k$ copies of $P$ and has a Riemannian metric of constant negative curvature. Furthermore, such a manifold $N$ is aspherical (has the homotopy type of Eilenberg--Mac Lane space $K(\Ker\varphi^{(k)},1)$), as its universal cover $\mathbb L^3$ is contractible.
\end{construction}

The abelianisation homomorphism
$G(P)\stackrel{\textrm{ab}}\longrightarrow \Z_2^m$ satisfies the condition of Lemma~\ref{vesninlemma}. Its kernel is the commutator subgroup $G(P)'$ of the right-angled Coxeter group~$G(P)$. The hyperbolic $3$-manifold $\mathcal R_P=\mathbb L^3/G(P)'$ can be identified with the real version of the \emph{moment-angle manifold}~$\zp$, see Subsections~\ref{mamanifolds} and~\ref{zpm}. It is also known as the \emph{universal abelian cover} of~$P$, see~\cite{grom87} and~\cite{da-ja91}.

The smallest possible value of $k$ for which $\varphi^{(k)}\colon G(P)\to\Z_2^k$ can satisfy the condition of Lemma~\ref{vesninlemma} is $k=3$. The corresponding quotient manifold  $N=\mathbb L^3/\Ker\varphi^{(3)}$, composed of $8$ copies of~$P$, was called a \emph{hyperbolic $3$-manifold of L\"obell type} in~\cite{vesn87}. L\"obell constructed first examples of these manifolds in 1931. The epimorphism $\varphi^{(3)}$ factors as $G(P)\stackrel{\textrm{ab}}\longrightarrow \Z_2^m\stackrel\varLambda\longrightarrow\Z_2^3$, where $\varLambda$ is a linear map. The condition of Lemma~\ref{vesninlemma} is equivalent to that $\varLambda$ satisfies~\eqref{starprop},  i.\,e.  $\varLambda$ is given by a $\Z_2$-characteristic matrix. We therefore can identify the hyperbolic manifold $N=\mathbb L^3/\Ker\varphi^{(3)}$ with the small cover $N(P,\varLambda)$.

\medskip

Pogorelov~\cite{pogo67} asked the following question in 1967: which combinatorial $3$-polytopes have right-angled realisations in~$\mathbb L^3$? Results of Pogorelov~\cite{pogo67} and Andreev~\cite{andr70} give a complete answer, which can be formulated in our terms as follows:

\begin{theorem}[\cite{pogo67,andr70}]\label{Pogth}
A combinatorial $3$-polytope can be realised as a right-angled polytope in Lobachevsky space~$\mathbb L^3$ if and only if it is simple, flag and does not have $4$-belts. Furthermore, such a realisation is unique up to isometry.
\end{theorem}

\begin{remark}
More specifically, Pogorelov's theorem stated that a combinatorial $3$-polytope has a right-angled realisation in~$\mathbb L^3$ if and only if it is simple, flag, does not have $4$-belts, and \emph{has a realisation in~$\mathbb L^3$ with all dihedral angles~$<\frac\pi2$}. Pogorelov also proved the uniqueness of a right-angled realisation.

Andreev considered the problem of description of discrete
reflection groups in Lobachevsky spaces, posed by Vinberg in
1967~\cite{vinb67}. This problem reduces to describing polytopes
with dihedral angles $\frac\pi n$, $n\ge 2$. Andreev's famous
theorem~\cite[Theorem~2]{andr70} gives necessary and sufficient
conditions for a combinatorial simple $3$-polytope $P$ with
prescribed values of dihedral angles $\le\frac\pi2$ to be
realisable in~$\mathbb L^3$. In particular, this theorem implies
Pogorelov's result. When $P$ is not a simplex or triangular prism,
Andreev's conditions are as follows:
\begin{itemize}
\item[(a)] the sum of dihedral angles between the facets meeting at a vertex is $>\pi$;
\item[(b)] the sum of dihedral angles between the facets forming a $3$-belt is $<\pi$;
\item[(c)] the sum of dihedral angles between the facets forming a $4$-belt is $<2\pi$.
\end{itemize}
In the absence of $3$- and $4$-belts the conditions (b) and (c)
are empty, so the result of Andreev implies that Pogorelov's last
condition follows from the other three conditions (simpleness,
flagness and the absence of $4$-belts).
%
\end{remark}

We refer to the class of simple flag $3$-polytopes without $4$-belts as the \emph{Pogorelov class}~$\mathcal P$. It will feature prominently throughout the rest of our paper.

A polytope from the class $\mathcal P$ has neither triangular nor
quadrangular facets. The Pogorelov class contains all fullerenes;
this follows from the results of Do\v{s}li\'{c}~\cite{dosl03} (see
also~\cite[Corollary~3.16]{bu-er15} and~\cite{bu-er,bu-erS}). As
we mentioned in the Introduction, the results of~\cite{thur98}
imply that the number of combinatorially different fullerenes with
$p_6$ hexagonal facets grows as~$p_6^9$. We also note that the
class $\mathcal P$ contains simple $3$-polytopes with pentagonal,
hexagonal and one heptagonal facet, which are used in the
construction of fullerenes by means of truncations
(see~\cite{bu-er, bu-erS, bu-erC}). Finally, we show in
Corollary~\ref{Pogpk} that for any finite sequence of nonnegative
integers $p_k$, $k\geqslant 7$, there exists a Pogorelov polytope
whose number of $k$-gonal facets is~$p_k$. All these facts imply
that the Pogorelov class of polytopes is large enough.

We summarise the constructions and results above as follows.

\begin{theorem}
A small cover $N(P,\varLambda)$ of a $3$-polytope $P$ from the Pogorelov class~$\mathcal P$ has the structure of a hyperbolic $3$-manifold $\mathbb L^3/\Ker\varphi^{(3)}$ of L\"obell type, with the epimorphsim $\varphi^{(3)}$ given by the composition $G(P)\stackrel{\mathrm{ab}}\longrightarrow \Z_2^m\stackrel\varLambda\longrightarrow\Z_2^3$. Furthermore, such a $3$-manifold $N(P,\varLambda)$ is aspherical.
\end{theorem}

The conditions specifying the Pogorelov class $\mathcal P$ also feature in Gromov's theory of hyperbolic groups. Namely, the ``no $\triangle$-condition'' from~\cite[\S4.2.E]{grom87} for a simplicial complex $\sK$ is the absence of missing $2$-faces, while the ``no $\square$-condition'' is the absence of chordless $4$-cycles. When $\sK$ is the dual complex of a simple polytope (see Subsection~\ref{scfr} below), these two conditions translate to the absence of $3$- and $4$-belts, respectively.

The relationship between small covers and hyperbolic manifolds was
also mentioned in the work of Davis and
Januszkiewicz~\cite[p.~428]{da-ja91}, although the criterion for
right-angledness was stated there incorrectly (as not every
$3$-polytope without triangular and quadrangular facets has a
right-angled realisation, see Example~\ref{csexa}).

Compact right-angled $n$-polytopes exist in Lobachevsky space
$\mathbb L^n$ of dimension $n=2,3,4$ only. On Lobachevsky
plane~$\mathbb L^2$, there are right-angled $m$-gons for any
$m\ge5$. The three-dimensional case has been described above.
There exist compact right-angled $4$-polytopes in $\mathbb L^4$,
but no classification is known up to date. The most well-known
example is the \emph{regular $120$-cell}. Given two right-angled
polytopes $P_1$ and $P_2$ with isometric facets $F_1\subset P_1$,
$F_2\subset P_2$ one can obtain a new right-angled polytope by
gluing $P_1$ and $P_2$ along~$F_1\cong F_2$. In this way, one can
produce infinitely many different right-angled polytopes in
$\mathbb L^4$ starting from the right-angled regular $120$-cell.
All known examples of right-angled $4$-polytopes are obtained in
this way. Note that for any convex polytope in Lobachevsky space
there is a combinatorial equivalent convex polytope in Euclidean
space; this follows easily by considering the Beltrami--Klein
model of~$\mathbb L^n$. The absence of right-angled polytopes in
$\mathbb L^n$ for $n\ge5$ was proved by Vinberg in~\cite{vinb84}
using Nikulin's inequalities~\cite{niku81} on the average number
of faces in a simple polytope. These inequalities imply that a
simple polytope of dimension $n\ge5$ has a triangular or
quadrangular $2$-face, which is impossible for a right-angled
polytope. See~\cite{po-vi05} for a survey of results on
right-angled polytopes.

\subsection{Topological toric manifolds}
A toric manifold is not necessarily a quasitoric manifold and
a quasitoric manifold is also not necessarily a toric manifold.
However, both toric and quasitoric manifolds are examples of topological toric manifolds introduced in~\cite{i-f-m13}.
Recall that a toric manifold admits an algebraic action of $(\C^\times)^n$ with an open dense orbit.
It has local charts equivariantly isomorphic to a sum of complex one-dimensional \emph{algebraic} representations of $(\C^\times)^n$.
A \emph{topological toric manifold} is a compact smooth $2n$-dimensional manifold with an effective smooth action of
$(\C^\times)^n$ having an open dense orbit and covered by finitely many invariant open subsets each equivariantly diffeomorphic to a sum of complex one-dimensional \emph{smooth} representation spaces of~$(\C^\times)^n$.
(The latter condition automatically follows from the existence of a dense orbit in the algebraic category, but not in the smooth category.)

The cohomology ring of a topological toric manifold is described similarly to the toric or quasitoric case; there is an analogue of
Theorems~\ref{cohomtoric} or~\ref{cohomqtoric}, see~\cite[Proposition~8.3]{i-f-m13}.

\subsection{Simplicial complexes and face rings}\label{scfr}
Let $\sK$ be an (abstract) \emph{simplicial complex} on the set $[m]=\{1,\ldots,m\}$,
i.\,e. $ \mathcal K$ is a collection of subsets $I\subset [m]$ such that
for any $I\in\sK$ all subsets of $I$ also belong to~$\sK$. We
always assume that the empty set $\varnothing$ and all one-element
subsets $\{i\}\subset[m]$ belong to~$\sK$; the latter are \emph{vertices} of~$\sK$. We refer to $ I \in
\mathcal K $ as a \emph{simplex} (or a \emph{face}) of~$\sK$.
Every abstract simplicial complex $\sK$ has a
\emph{geometric realisation} $|\sK|$, which is a polyhedron in a
Euclidean space (a union of convex geometric simplices).

A \emph{non-face} of $\sK$ is a subset $I\subset [m]$ such that $I\notin
\sK$. A \emph{missing face} (a \emph{minimal non-face}) of
$\sK$ is an inclusion-minimal non-face of $\sK$, that is, a subset
$I\subset[m]$ such that $I$ is not a simplex of~$\sK$, but every
proper subset of $I$ is a simplex of~$\sK$.

A simplicial complex $\sK$ is called a \emph{flag complex} if each
of its missing faces consists of two vertices. Equivalently, $\sK$
is flag if any set of vertices of $\sK$ which are pairwise
connected by edges spans a simplex. Every flag complex $\sK$ is
determined by its 1-skeleton $\sK^1$, and is obtained from the
graph $\sK^1$ by filling in all complete subgraphs by simplices.

Let $P$ be a simple $n$-polytope with $m$ facets $F_1,\ldots,F_m$.
Then
\[
  \sK_P=\bigl\{I=\{i_1,\ldots,i_k\}\in[m]\colon
  F_{i_1}\cap\cdots\cap F_{i_k}\ne\varnothing\bigr\}
\]
is a simplicial complex on $[m]$, called the \emph{dual complex}
of~$P$. The vertices of $\sK_P$ correspond to the facets of~$P$,
and the empty simplex $\varnothing$ corresponds to~$P$ itself.
Geometrically, $|\sK_P|$ is an $(n-1)$-dimensional sphere
simplicially subdivided as the boundary of the \emph{dual
polytope} of~$P$.
%

The definitions of flag polytopes and complexes agree: $P$ is a
flag polytope if and only if $\sK_P$ is a flag complex. A $k$-belt
in $P$ with $k\ge4$ corresponds to a \emph{chordless $k$-cycle} in
the graph~$\sK_P^1$.

The barycentric subdivision of any simplicial complex $\mathcal K$ on $[m]$ has a structure of a cubical complex $\mathrm{cub}(\sK)$, which embeds canonically into the cubical complex of faces of an $m$-dimensional cube~$I^m$~\cite[\S1.5]{bu-pa00}. The cubical complex $\mathrm{cub}(\sK)$ has a piecewise Euclidean structure in which each cubical face is a Euclidean cube.  It was shown
in~\cite[\S4]{grom87} that the corresponding piecewise Euclidean metric has non-positive curvature (in the sense of the comparison  $\mathrm{CAT}(0)$-inequality of Alexandrov and Toponogov) if all links satisfy the no-$\triangle$-condition (which is equivalent to the flagness of~$\sK$), whereas the no-$\square$-condition implies that the curvature is strictly negative. Hyperbolic manifolds associated with $3$-polytopes from the Pogorelov class (see Subsection~\ref{sechyp}) satisfy a much stronger condition: they carry a genuine Riemannian metric of constant negative curvature.

We fix a commutative ring $\k$ with unit.

The \emph{face ring} of $\sK$ (also known as the
\emph{Stanley--Reisner ring}) is defined as the quotient of the
polynomial ring $\k[v_1,\ldots,v_m]$ by the square-free
monomial ideal generated by non-simplices of~$\sK$:
\[
  \k[\sK]=\k[v_1,\ldots,v_m]\big/\bigl(v_{i_1}\cdots v_{i_k}\
  \colon
  \{i_1,\ldots,i_k\}\notin\sK\bigr).
\]
As $\k[\sK]$ is the quotient of the polynomial ring by a monomial
ideal, it has a grading or even a multigrading (a $\Z^m$-grading).
We use an even grading: $\deg v_i=2$ and $\mdeg v_i=2\mb e_i$,
where $\mb e_i\in\Z^m$ is the $i$th standard basis vector.

Note that when $\sK=\sK_P$ for a simple polytope $P$, the ring
$\Z[P]$ coincides with the quotient of $\Z[v_1,\ldots,v_m]$ by the
relations~(a) in Theorem~\ref{cohomtoric} or in
Theorem~\ref{cohomqtoric}.

A simplicial complex $\sK$ is flag if and only if
$\k[\sK]$ is a \emph{quadratic algebra}, i.\,e. the quotient of
$\k[v_1,\ldots,v_m]$ by an ideal generated by quadratic monomials
(which have degree~4 in our grading).

\subsection{Moment-angle complexes and manifolds}\label{mamanifolds}
Let $\sK$ be a simplicial complex on the set $[m]$,
and let $(D^2,S^1)$ denote the pair of a disc and its boundary
circle. For each simplex $I=\{i_1,\ldots,i_k\}\in\sK$, set
\[
  (D^2,S^1)^I=\{(x_1,\ldots,x_m)\in(D^2)^m\colon x_i\in S^1\text{ when }i\notin I\}.
\]
The \emph{moment-angle complex} is defined as
\begin{equation}\label{defzk}
  \zk=(D^2,S^1)^\sK=\bigcup_{I\in\sK}(D^2,S^1)^I\subset(D^2)^m.
\end{equation}
If $|\sK|$ is homeomorphic to a sphere $S^{n-1}$, then $\zk$ is a
topological manifold. If $|\sK|$ is the boundary of a convex
polytope or is a starshaped sphere (the underlying complex of a complete
simplicial fan), then $\zk$ has a smooth structure~\cite{pano13}.

In the polytopal case there is an alternative way to define $\zk$
in terms of the dual simple polytope~$P$. Namely, assume given a
presentation of a convex $n$-dimensional polytope $P$ by
inequalities~\eqref{ptope}. Define the map
\[
  i_P\colon\R^n\to\R^m,\quad\mb x\mapsto
  \bigl(\langle\mb a_1,\mb x\rangle+b_1,\ldots,\langle\mb a_m,\mb x\rangle+b_m\bigr),
\]
so $i_P(P)\subset\R^m_\ge=\{(y_1,\ldots,y_m)\in\R^m\colon
y_i\ge0\}$. Also, define the map
\begin{equation}\label{mumap}
  \mu\colon\C^m\to\R^m_\ge,\quad(z_1,\ldots,z_m)\mapsto(|z_1|^2,\ldots,|z_m|^2).
\end{equation}
Then define the space $\zp$ by the pullback diagram
\begin{equation}\label{defzp}
\begin{CD}
  \zp @>>> \C^m\\
  @VVV @VV{\mu}V\\
  P @>{i_p}>> \R^m_\ge
\end{CD}
\end{equation}
The space $\zp=\mu^{-1}(i_P(P))$ can be written as an intersection of $(m-n)$
Hermitian quadrics in $\C^m$, and this intersection is
nondegenerate precisely when the polytope $P$ is simple. In the
latter case, $\zp$ is a smooth $(m+n)$-dimensional manifold.
Furthermore, the manifold $\zp$ is diffeomorphic to the
moment-angle complex~$\mathcal Z_{\sK_P}$. In particular, the
diffeomorphism type of $\zp$ depends only on the combinatorial
type of~$P$. We shall therefore not distinguish between $\zp$ and
$\mathcal Z_{\sK_P}$ and refer to it as the \emph{moment-angle
manifold} corresponding to a simple polytope~$P$. The details of
these constructions can be found in~\cite{pano13} or in~\cite[Chapter~6]{bu-pa15}.

The standard coordinatewise action of the $m$-torus $T^m$ on
$(D^2)^m$ or $\C^m$ induces the \emph{canonical} $T^m$-action on
$\zk$ or~$\zp$.

There is a ``real'' version of these constructions with the pair $(D^2,S^1)$ replaced by $(D^1,S^0)$ and the map~\eqref{mumap} replaced by
\[
  \mu_\R\colon\R^m\to\R^m_\ge,\quad(y_1,\ldots,y_m)\mapsto(y_1^2,\ldots,y_m^2).
\]
The resulting \emph{real moment-angle manifold} $\mathcal R_P=\mu_\R^{-1}(i_P(P))$ has dimension $n$ and is given as an intersection of $(m-n)$ quadrics in~$\R^m$. It features in the constructions of Hamiltonian-minimal Lagrangian submanifolds of~\cite{miro04},~\cite{mi-pa13},~\cite{pano13}.

\subsection{Cohomology of moment-angle complexes}
We consider
(co)homology with coefficients in~$\k$. Denote by $\Lambda[u_1,\ldots,u_m]$ the
exterior algebra on $m$ generators over~$\k$ which satisfy the relations
$u_i^2=0$ and $u_iu_j=-u_ju_i$.

The \emph{Koszul complex} (or the \emph{Koszul algebra}) of the
face ring $\k[\sK]$ is defined as the differential
$\Z\oplus\Z^m$-graded algebra
$(\Lambda[u_1,\ldots,u_m]\otimes\k[\sK],d)$, where
\begin{equation}\label{bigrading}
  \mdeg u_i=(-1,2\mb e_i),\quad\mdeg v_i=(0,2\mb e_i),\qquad
  du_i=v_i,\quad dv_i=0.
\end{equation}
Cohomology of $(\Lambda[u_1,\ldots,u_m]\otimes\k[\sK],d)$ is
the \emph{$\Tor$-algebra} $\Tor_{\k[v_1,\ldots,v_m]}(\k[\sK],\k)$.
It also inherits a $\Z\oplus\Z^m$-grading.

\begin{theorem}[{\cite{b-b-p04}, \cite[Theorem~4.5.5]{bu-pa15}}]\label{zkcoh}There
are isomorphisms of (multi)graded commutative algebras
\begin{align*}
  H^*(\zk)&\cong\Tor_{\k[v_1,\ldots,v_m]}\bigl(\k[\sK],\k\bigr)\\
  &\cong H\bigl(\Lambda[u_1,\ldots,u_m]\otimes\k[\sK],d\bigr).
\end{align*}
\end{theorem}

The cohomology of $\zk$ therefore acquires a multigrading, with
the multigraded and ordinary graded components of~$H^*(\zk)$ given
by
\[
  H^{-i,2J}(\zk)=\Tor^{-i,2J}_{\k[v_1,\ldots,v_m]}\bigl(\k[\sK],\k\bigr),
  \quad H^\ell(\zk)=\bigoplus_{-i+2|J|=\ell}H^{-i,2J}(\zk),
\]
where $J=(j_1,\ldots,j_m)\in\Z^m$ and $|J|=j_1+\cdots+j_m$.

The Koszul algebra
$(\Lambda[u_1,\ldots,u_m]\otimes\k[\sK],d\bigr)$ is infinitely
generated as a $\k$-module. We define its quotient algebra
\[
  R^*(\sK)=\Lambda[u_1,\ldots,u_m]\otimes\k[\sK]\bigr/(v_i^2=u_iv_i=0,\;
  1\le i\le m)
\]
with the induced multigrading and differential~\eqref{bigrading}.
Note that $R^*(\sK)$ has a finite $\k$-basis. Passing to
$R^*(\sK)$ does not change the cohomology. This can be proved
either algebraically~\cite[Lemma~4.4]{pano08} or using the
following topological interpretation:

\begin{lemma}[{\cite[Lemma~4.5.3]{bu-pa15}}]\label{cellappr}
The algebra $R^*(\sK)$ coincides with the cellular cochains of
$\zk$ for the appropriate cell structure. In particular, there is
an isomorphism of cohomology algebras
\[
  H(R^*(\mathcal K))\cong H^*(\zk).
\]
\end{lemma}

The multigraded component $R^{-i,2J}(\sK)$ is zero unless all
coordinates of the vector $J\in\Z^m$ are $0$ or~$1$,
and the same is true for the multigraded cohomology
$H^{-i,2J}(\zk)$.

We can identify subsets $J\subset[m]$ with vectors $\sum_{j\in
J}\mb e_j\in\Z^m$. Given $J=\{j_1,\ldots,j_k\}\subset[m]$, we
denote by $v_J$ the monomial $v_{j_1}\cdots
v_{j_k}\in\k[v_1,\ldots,v_m]$, and similarly consider exterior
monomials $u_J=u_{j_1}\cdots
u_{j_k}\in\Lambda[u_1,\ldots,u_m]$. We also use the notation
$u_Jv_I$ for the monomial $u_J\otimes v_I$ in the Koszul algebra
$\Lambda[u_1,\ldots,u_m]\otimes\k[\sK]$. Then $R^*(\sK)$ has a
finite $\k$-basis consisting of monomials $u_Jv_I$ where
$J\subset[m]$, $I\in\sK$ and $J\cap I=\varnothing$.


Given $J\subset[m]$, define the corresponding \emph{full
subcomplex} of~$\sK$ 
as
\[
  \sK_J=\{I\in\sK\colon I\subset J\}.
\]
Consider simplicial cochains $C^*(\sK_J)$  with coefficients
in~$\k$. Let $\alpha_L\in C^{p-1}(\sK_J)$ be the basis cochain
corresponding to an oriented simplex $L=(l_1,\ldots,l_p)\in\sK_J$;
it takes value~$1$ on $L$ and vanishes on all other simplices.
Define a $\k$-linear map
\begin{equation}\label{fmapr}
\begin{aligned}
  f\colon C^{p-1}(\sK_J)&\longrightarrow R^{p-|J|,2J}(\sK),\\
  \alpha_L&\longmapsto \varepsilon(L,J)\,u_{J\setminus L}v_L,
\end{aligned}
\end{equation}
where $\varepsilon(L,J)$ is the sign given by
$\varepsilon(L,J)=\prod_{j\in L}\varepsilon(j,J)$ and
$\varepsilon(j,J)=(-1)^{r-1}$ if $j$ is the $r$th element of the
set~$J\subset[m]$ written in increasing order.

\begin{theorem}[{\cite[Theorem~3.2.9]{bu-pa15}}]\label{HTor}
The maps~\eqref{fmapr} combine to an isomorphism of cochain
complexes $C^*(\sK_J)\to R^{\,*,2J}(\sK)$ and induce an
isomorphism
\[
  \widetilde{H}^{|J|-i-1}(\mathcal K_J)\cong \Tor_{\k[v_1,\ldots,v_m]}^{-i,2J}
  \bigl(\k[\sK],\k\bigr),
\]
where $\widetilde{H}^{k}(\mathcal K_J)$ denotes the $k$th reduced
simplicial cohomology group of~$\sK_J$.
\end{theorem}

\begin{theorem}[{\cite[Theorem~4.5.8]{bu-pa15}}]\label{prodfsc}
There are isomorphisms of $\k$-modules
\[
  H^{-i,2J}(\zk)\cong
  \widetilde H^{|J|-i-1}(\sK_J),\qquad
  H^\ell(\zk)\cong\bigoplus_{J\subset[m]}
  \widetilde H^{\ell-|J|-1}(\sK_J).
\]
These isomorphisms combine to form a ring isomorphism
\[
  H^*(\zk)\cong\bigoplus_{J\subset[m]} \widetilde H^*(\sK_J),
\]
where the ring structure on the right hand side is given by the
product maps
\[
  H^{k-|I|-1}(\sK_{I})\otimes H^{\ell-|J|-1}(\sK_{J})\to
  H^{k+\ell-|I|-|J|-1}(\sK_{I\cup J})
\]
which are induced by the simplicial inclusions $\sK_{I\cup
J}\to\sK_I\mathbin{*}\sK_J$ for $I\cap J=\varnothing$ and are zero
otherwise.
\end{theorem}

\begin{proposition}\label{3dimcoh}
The $3$-dimensional cohomology $H^3(\zk)$ is freely generated by
the cohomology classes $[u_iv_j]=[u_jv_i]$ corresponding to pairs
of vertices $i,j$ such that $\{i,j\}\notin\sK$. If $\sK=\sK_P$ for
a simple polytope $P$, then these $3$-dimensional cohomology
classes correspond to pairs of non-adjacent facets $F_i,F_j$.
\end{proposition}

\begin{example}\label{exacochain}
Let $\mathcal K=\quad\begin{picture}(18,2) \put(0,1){\circle*{1}}
\put(5,1){\circle*{1}} \put(13,1){\circle*{1}}
\put(18,1){\circle*{1}} \put(0,1){\line(1,0){5}}
\put(13,1){\line(1,0){5}} \put(-2.5,0){\scriptsize 1}
\put(6.5,0){\scriptsize 2} \put(10.5,0){\scriptsize 3}
\put(19.5,0){\scriptsize 4}
\end{picture}\quad$
be the union of two segments. Then
 nontrivial integral cohomology
groups of $\zk$ are given below together with a basis represented by cocycles in the
algebra $R^*(\sK)$:
\begin{align*}
  H^0(\zk)&\cong\widetilde H^{-1}(\varnothing)\cong\Z && 1\\
  H^3(\zk)&\cong\bigoplus_{|J|=2}\widetilde H^{0}(\sK_J)\cong\Z^4
  && u_1v_3,\;u_1v_4,\;u_2v_3,\;u_2v_4\\
  H^4(\zk)&\cong\bigoplus_{|J|=3}\widetilde H^{0}(\sK_J)\cong\Z^4
  && u_1u_2v_3,\;u_1u_2v_4,\;u_3u_4v_1,\;u_3u_4v_2\\
  H^5(\zk)&\cong\widetilde H^0(\sK)\cong\Z && u_1u_2u_4v_3-u_1u_2u_3v_4
\end{align*}
Cochains in $C^0(\sK)$ are functions on the vertices of~$\sK$, and
cocycles are functions which are constant on the connected
components of~$\sK$. In our case, the cocycle
$\alpha_{\{3\}}+\alpha_{\{4\}}$ represents a generator
of~$\widetilde H^0(\sK)$. It is mapped by~\eqref{fmapr} to the
cocycle $u_1u_2u_4v_3-u_1u_2u_3v_4$ representing a generator
of~$H^5(\zk)$.
\end{example}

Moment-angle complexes $\zk$ may have nontrivial triple Massey
products of 3-dimensional cohomology classes. First examples
(found by Baskakov~\cite{bask03}) appear already for moment-angle
manifolds corresponding to 3-polytopes (see
also~\cite[\S4.9]{bu-pa15}). A complete description of the triple
Massey product $H^3(\zk)\otimes H^3(\zk)\otimes H^3(\zk)\to
H^8(\zk)$ is given by the following result of Denham and Suciu:

\begin{theorem}[{\cite[Theorem~6.1.1]{de-su07}}]\label{desu}
The following are equivalent:
\begin{itemize}
\item[(a)]
there exist cohomology classes $\alpha,\beta,\gamma\in H^3(\zk)$
for which the Massey product $\langle\alpha,\beta,\gamma\rangle$
is defined and non-trivial;

\item[(b)]
the graph $\sK^1$ contains an induced subgraph isomorphic to one
of the five graphs in Figure~\ref{5gr}.
\end{itemize}
\end{theorem}

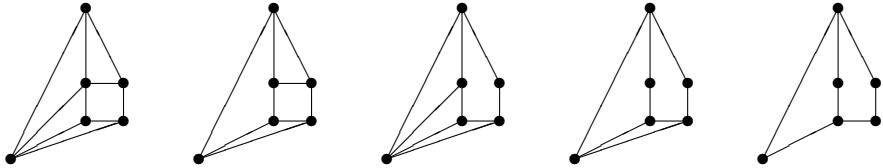
\begin{figure}[h]
\begin{center}
\begin{picture}(120,20)
  \put(0,0){\circle*{1.5}}
  \put(10,5){\circle*{1.5}}
  \put(15,5){\circle*{1.5}}
  \put(10,10){\circle*{1.5}}
  \put(15,10){\circle*{1.5}}
  \put(10,20){\circle*{1.5}}
  \put(0,0){\line(3,1){15}}
  \put(0,0){\line(2,1){10}}
  \put(0,0){\line(1,1){10}}
  \put(0,0){\line(1,2){10}}
  \put(10,5){\line(1,0){5}}
  \put(10,5){\line(0,1){15}}
  \put(10,10){\line(1,0){5}}
  \put(15,10){\line(-1,2){5}}
  \put(15,5){\line(0,1){5}}
  \put(25,0){\circle*{1.5}}
  \put(35,5){\circle*{1.5}}
  \put(40,5){\circle*{1.5}}
  \put(35,10){\circle*{1.5}}
  \put(40,10){\circle*{1.5}}
  \put(35,20){\circle*{1.5}}
  \put(25,0){\line(3,1){15}}
  \put(25,0){\line(2,1){10}}
  \put(25,0){\line(1,2){10}}
  \put(35,5){\line(1,0){5}}
  \put(35,5){\line(0,1){15}}
  \put(35,10){\line(1,0){5}}
  \put(40,10){\line(-1,2){5}}
  \put(40,5){\line(0,1){5}}
  \put(50,0){\circle*{1.5}}
  \put(60,5){\circle*{1.5}}
  \put(65,5){\circle*{1.5}}
  \put(60,10){\circle*{1.5}}
  \put(65,10){\circle*{1.5}}
  \put(60,20){\circle*{1.5}}
  \put(50,0){\line(3,1){15}}
  \put(50,0){\line(2,1){10}}
  \put(50,0){\line(1,1){10}}
  \put(50,0){\line(1,2){10}}
  \put(60,5){\line(1,0){5}}
  \put(60,5){\line(0,1){15}}
  \put(65,10){\line(-1,2){5}}
  \put(65,5){\line(0,1){5}}
  \put(75,0){\circle*{1.5}}
  \put(85,5){\circle*{1.5}}
  \put(90,5){\circle*{1.5}}
  \put(85,10){\circle*{1.5}}
  \put(90,10){\circle*{1.5}}
  \put(85,20){\circle*{1.5}}
  \put(75,0){\line(3,1){15}}
  \put(75,0){\line(2,1){10}}
  \put(75,0){\line(1,2){10}}
  \put(85,5){\line(1,0){5}}
  \put(85,5){\line(0,1){15}}
  \put(90,10){\line(-1,2){5}}
  \put(90,5){\line(0,1){5}}
  \put(100,0){\circle*{1.5}}
  \put(110,5){\circle*{1.5}}
  \put(115,5){\circle*{1.5}}
  \put(110,10){\circle*{1.5}}
  \put(115,10){\circle*{1.5}}
  \put(110,20){\circle*{1.5}}
  \put(100,0){\line(2,1){10}}
  \put(100,0){\line(1,2){10}}
  \put(110,5){\line(1,0){5}}
  \put(110,5){\line(0,1){15}}
  \put(115,10){\line(-1,2){5}}
  \put(115,5){\line(0,1){5}}
\end{picture}
\end{center}
\caption{Five graphs.}\label{5gr}
\end{figure}

\subsection{Moment-angle manifolds, quasitoric manifolds and small covers}\label{zpm}
Let $P$ be a simple $n$-polytope with the dual simplicial
complex~$\sK_P$. The existence of a characteristic matrix~\eqref{Lambdaqtoric} for~$P$
is equivalent to a choice of $n$ linear forms
\begin{equation}\label{linreg}
  t_j=\lambda_{j1}v_1+\cdots+\lambda_{jm}v_m,\quad j=1,\ldots,n
\end{equation}
such that $\Z[\sK_P]$ is a finitely generated free module over
$\Z[t_1,\ldots,t_n]$. Then $t_1,\ldots,t_n$ is a linear
\emph{regular sequence} in~$\Z[\sK_P]$. This implies that
$\k[\sK_P]$ is a \emph{Cohen--Macaulay ring} over any~$\k$, but
the condition of existence of a characteristic matrix is actually
stronger, as it implies the existence of a \emph{linear} regular
sequence over $\Z$ (and hence over any finite field).

Given a characteristic matrix~\eqref{Lambdaqtoric} (or a linear
regular sequence~\eqref{linreg} in $\Z[\sK_P]$), one can define
the corresponding homomorphism of tori $\varLambda_T\colon T^m\to
T^n$. Its kernel $\Ker\varLambda_T$ is an ${(m-n)}$-dimensional
subtorus in $T^m$ that acts \emph{freely} on~$\zp$. The quotient
$\zp/\Ker\varLambda_T$ can be identified with the quasitoric
manifold $M(P,\varLambda)$ from Construction~\ref{canqt}. As $\zp$
is a smooth intersection of quadrics~\eqref{defzp} and the torus
action is smooth, we obtain a canonical smooth structure on
$M(P,\varLambda)$ as in~\cite{b-p-r07}.

We say that $T^n$-manifolds $M$ and $M'$ are \emph{weakly equivariantly diffeomorphic} if there is a diffeomorphism $f\colon M\to M'$ and an automorphism $\theta\colon T^n\to T^n$ such that $f(t\cdot x)=\theta(t)\cdot f(x)$ for any $x\in M$ and $t\in T^n$. The following result is immediate.

\begin{proposition}\label{qteqd}
If characteristic pairs $(P,\varLambda)$ and $(P',\varLambda')$ are equivalent, then the corresponding quasitoric manifolds $M(P,\varLambda)$ and $M(P',\varLambda')$ are weakly equivariantly diffeomorphic.
\end{proposition}

The general homological properties of regular sequences imply yet
another description of the cohomology of~$\zp$:

\begin{theorem}[{\cite[Theorem~4.2.11]{bu-pa00}, \cite[Lemma~A.3.5]{bu-pa15}}]\label{zkcohred}
Let $P$ be a simple $n$-polytope with $m$ facets, and assume there
exists a linear integral regular sequence~\eqref{linreg}. Denote
by $\mathcal J$ the ideal in~$\Z[v_1,\ldots,v_m]$ generated
by~$t_1,\ldots,t_n$. Then there is an isomorphism of cohomology
rings
\[
  H^*(\zp;\Z)\cong\Tor_{\Z[v_1,\ldots,v_m]/\mathcal J}\bigl(\Z[\sK_P]/\mathcal
  J,\Z\bigr).
\]
\end{theorem}

Note that $\Z[\sK_P]/\mathcal J$ is the cohomology ring of the
quasitoric manifold $M(P,\varLambda)$, see
Theorem~\ref{cohomqtoric}. The theorem above implies that the
spectral sequence of the principal $T^{m-n}$-fibration $\zp\to
M(P,\varLambda)$ degenerates at the $E_3$ term.

The complex conjugation $\mb z=(z_1,\ldots,z_m)\mapsto\bar{\mb z}=(\bar z_1,\ldots,\bar z_m)$ defines an involution on $\zp$ whose set of fixed points is the real moment-angle manifold~$\mathcal R_P$. For any element $t$ of the torus $\Ker\varLambda_T\cong T^{m-n}$, this involution satisfies $\overline{t\cdot\mb z}=t^{-1}\cdot\bar{\mb z}$, and therefore it descends to an involution on the quasitoric manifold $M(P,\varLambda)$. The fixed point set of the latter involution is the small cover $N(P,\varLambda)$ corresponding to the $\mod 2$-reduction of the $\Z$-characteristic matrix~$\varLambda$. It is not known whether any small cover over a simple $n$-polytope can be obtained in this way; this question is equivalent to Problem~\ref{mod2prob} (the answer is positive for $3$-polytopes, see Proposition~\ref{lambda2}).

We have a $\Z_2^{m-n}$-covering $\mathcal R_P\to N(P,\varLambda)$ for any small cover of $P$ corresponding to a $\Z_2$-characteristic matrix~$\varLambda$. The fundamental group of $\mathcal R_P$ is the commutator subgroup $G(P)'$ of the (abstract) right-angled Coxeter group~\eqref{racg} corresponding to~$P$. The fundamental group of a small cover $N=N(P,\varLambda)$ is determined by the following exact sequence
\[
  1\longrightarrow G(P)'\longrightarrow \pi_1(N)\longrightarrow\Z_2^{m-n}\longrightarrow1.
\]
The commutator subgroups of right-angled Coxeter groups were
studied in~\cite{pa-ve16}; in particular, a minimal set of
generators for $G(P)'$ was described there. The manifold $\mathcal
R_P$ (and therefore $N$) is aspherical if and only if the
polytope~$P$ is flag. This follows from Davis' construction of a
nonpositively curved piecewise Euclidean metric on
$N(P,\varLambda)$ for flag~$P$, see~\cite[Theorem~2.2.5]{d-j-s98}
and also~\cite[Corollary~3.4]{pa-ve16}.

When $P$ is a right-angled polytope in $\mathbb L^3$ (so that $P\in\mathcal P$; in particular, $P$ is flag), we have a sequence of coverings $\mathbb L^3\to\mathcal R_P\to N(P,\varLambda)$. Here, $G(P)$ is a geometric right-angled Coxeter group generated by reflections in the facets of~$P$, and both $\mathcal R_P$ and $N(P,\varLambda)$ have a genuine Riemannian metric of constant negative curvature:

\begin{proposition}The real moment-angle manifold $\mathcal R_P$ corresponding to a $3$-polytope from the Pogorelov class~$\mathcal P$ has a structure of a hyperbolic $3$-manifold. The fundamental group of $\mathcal R_P$ is isomorphic to the commutator subgroup $G(P)'$ of the corresponding right-angled Coxeter group.
\end{proposition}

\section{Cohomological rigidity} \label{cohrig}
We continue to consider cohomology
with coefficients in a commutative ring with unit~$\k$. When $\k$ is not specified explicitly,
we assume $\k=\Z$.

\begin{definition}\label{defcohorig}
We say that a family of closed manifolds is \emph{cohomologically
rigid} over $\mathbf k$ if manifolds in the family are
distinguished up to diffeomorphism by their cohomology rings with
coefficients in~$\mathbf k$. That is, a family is cohomologically
rigid if a graded ring isomorphism $H^*(M_1;\mathbf k)\cong
H^*(M_2;\mathbf k)$ implies a diffeomorphism $M_1\cong M_2$
whenever $M_1$ and $M_2$ are in the family.

There are homotopical and topological versions of cohomological
rigidity, with diffeomorphisms replaced by homotopy equivalences
and homeomorphisms, respectively.
\end{definition}

In toric topology, cohomological rigidity is studied for
(quasi)toric manifolds and moment-angle manifolds. We refer
to~\cite{ma-su08}, \cite{c-m-s11} and~\cite[\S7.8]{bu-pa15} for a
more detailed survey of related results and problems. The main
question here is as follows.

\begin{problem}\label{crqtm}
Let $M_1$ and $M_2$ be two toric manifolds with isomorphic
cohomology rings. Are they homeomorphic? In other words, is the
family of toric manifolds cohomologically rigid? One can ask the same question
for quasitoric and topological toric manifolds, and with homeomorphisms replaced by diffeomorphisms.
\end{problem}

The problem is solved positively for particular families of
toric and quasitoric manifolds, such as cohomologically trivial
Bott towers~\cite{ma-pa08}, $\Q$-cohomologically trivial Bott
towers~\cite{ch-ma12}, $\Z_2$-cohomologically trivial Bott
towers~\cite{c-m-m15}, Bott towers of real dimension up
to~$8$~\cite{choi}, quasitoric manifolds over a product of two
simplices~\cite{c-p-s12} and over some dual cyclic polytopes~\cite{hasu15}. \emph{Bott towers} (or \emph{Bott manifolds}) are toric manifolds over combinatorial cubes.
The problem is open for general Bott
towers, and for (quasi)toric manifolds of real dimension~$6$, that is,
over $3$-dimensional polytopes. The latter case is the subject of
this paper: we give a solution for a particular class of
$3$-polytopes.

There is also a cohomological rigidity problem for real toric objects, such as real toric manifolds, small covers, and real topological toric manifolds~\cite{i-f-m13}, with $\Z_2$-cohomology rings.
This problem is solved positively for real Bott towers~\cite{c-m-o16}, \cite{ka-ma09}, but negatively in some other cases~\cite{masu09}.

Cohomological rigidity is also open for moment-angle manifolds, in
both graded and multigraded versions:

\begin{problem}\label{crmam}
Let $\mathcal Z_{P_1}$ and $\mathcal Z_{P_2}$ be two moment-angle
manifolds with isomorphic (multigraded) cohomology rings. Are
they diffeomorphic? In other words, is the family of moment-angle
manifolds cohomologically rigid?
\end{problem}

%

A diffeomorphism of two quasitoric manifolds over $P_1$ and $P_2$
or a diffeomorphism of moment-angle manifolds $\mathcal Z_{P_1}$
and $\mathcal Z_{P_2}$ does not imply that the polytopes $P_1$ and
$P_2$ are combinatorially equivalent, as shown by the next
example.

\begin{example}\label{vtpol}
A \emph{vertex truncation}
operation~\cite[Construction~1.1.1]{bu-pa15} can be applied to a
simple polytope $P$ to produce a new simple polytope
$\mathop{\mathrm{vt}}(P)$ with one more facet. If one applies this
operation iteratively, then the combinatorial type of the
resulting polytope depends, in general, on the choice and order of
truncated vertices. For example, by applying this operation three
times to a $3$-simplex one can produce three combinatorially
different polytopes $P_i$, $i=1,2,3$, with $7$ facets each (their
dual simplicial polytopes are known as \emph{stacked}). The
corresponding moment-angle manifolds $\mathcal Z_{P_i}$ are
diffeomorphic, see~\cite[\S4.6]{bu-pa15}. The polytopes $P_i$ have
Delzant realisations such that the correponding toric manifolds
$V_{P_i}$ are obtained from $\C P^3$ by blowing it up three times
in three different ways. Each $V_{P_i}$ is therefore diffeomorphic
to a connected sum of $4$ copies of~$\C P^3$
(see~\cite[Example~4.3]{ma-su08}).
\end{example}

One can look for classes of simple polytopes $P$ whose
combinatorial type is determined by the cohomology ring of any
(quasi)toric manifold over $P$ or by the cohomology ring of the
moment-angle manifold~$\zp$. This leads to the following two
notions of rigidity for simple polytopes, considered in~\cite{ma-su08} and~\cite{buch08} respectively.

\begin{definition}
A simple polytope $P$ is said to be \emph{C-rigid}
if any of the two conditions hold:
\begin{itemize}
\item[(a)] there are no quasitoric manifolds $M$ over $P$ (equivalently,
there are no linear regular sequences~\eqref{linreg}
in~$\Z[\sK_P]$), or

\item[(b)] whenever there exist a quasitoric manifold $M$ over $P$
and a quasitoric manifold $M'$ over another polytope $P'$ with a
cohomology ring isomorphism $H^*(M)\cong H^*(M')$, there is a
combinatorial equivalence $P\simeq P'$.
\end{itemize}

We say that a property of simple polytopes is \emph{C-rigid} if
for any ring isomorphism $H^*(M)\cong H^*(M')$, both $P$ and $P'$
either have or do not have the property.
%
\end{definition}

\begin{definition}
A simple polytope $P$ is said to be \emph{B-rigid} if any
cohomology ring isomorphism $H^*(\zp)\cong H^*(\mathcal Z_{P'})$
of moment-angle manifolds implies a combinatorial equivalence
$P\simeq P'$.

We say that a property of simple polytopes is \emph{B-rigid} if
for any ring isomorphism $H^*(\zp)\cong H^*(\mathcal Z_{P'})$,
both $P$ and $P'$ either have or do not have the property.
%
\end{definition}

According to Example~\ref{vtpol}, a truncated simplex with at
least 3 truncations (the dual to a stacked polytope with at least
3 stacks) is neither C-rigid nor B-rigid. Previously known examples of
C-rigid polytopes include products of simplices and their single
vertex truncations~\cite{c-p-s10},
as well as a product of a simplex and a polygon~\cite{ch-pa16}. Also, C-rigidity was
determined in~\cite{c-p-s10} for all simple $3$-polytopes with
$\le9$ facets. The following relation between the two notions of
rigidity can be extracted from the results of~\cite{c-p-s10}:

\begin{proposition}\label{bcrig}
If a simple polytope $P$ is B-rigid, then it is C-rigid.
\end{proposition}
\begin{proof}
Assume that we have a cohomology ring isomorphism $\varphi\colon
H^*(M)\stackrel\cong\longrightarrow H^*(M')$ for quasitoric
manifolds $M$ over $P$ and $M'$ over~$P'$. We need to show that it
implies a ring isomorphism $\psi\colon
H^*(\zp)\stackrel\cong\longrightarrow H^*(\mathcal Z_{P'})$, as
the latter would give $P\simeq P'$ by B-rigidity. Let $\mathcal J$
and $\mathcal J'$ denote the corresponding ideals in $\Z[\sK_P]$
and $\Z[\sK_{P'}]$, respectively, generated by the linear regular
sequences~\eqref{linreg}. Then we have a ring isomorphism
$\varphi\colon\Z[\sK_P]/\mathcal
J\stackrel\cong\longrightarrow\Z[\sK_{P'}]/\mathcal J'$. We need
to show that this isomorphism gives rise to a ring isomorphism
\begin{equation}\label{quottor}
  \Tor_{\Z[v_1,\ldots,v_m]/\mathcal J}\bigl(\Z[\sK_P]/\mathcal
  J,\Z\bigr)\stackrel\cong\longrightarrow
  \Tor_{\Z[v_1,\ldots,v_{m'}]/\mathcal J'}\bigl(\Z[\sK_{P'}]/\mathcal
  J',\Z\bigr),
\end{equation}
as the latter is nothing but an isomorphism
$H^*(\zp)\stackrel\cong\longrightarrow H^*(\mathcal Z_{P'})$
according to Theorem~\ref{zkcohred}. This is is proved
in~\cite[Lemma~3.7]{c-p-s10}. Namely, the isomorphism
$\varphi\colon\Z[\sK_P]/\mathcal
J\stackrel\cong\longrightarrow\Z[\sK_{P'}]/\mathcal J'$ can be
extended to a commutative diagram
\[
\xymatrix{
  \Z[v_1,\ldots,v_m]/\mathcal J \ar[r]^{\cong} \ar[d] &
  \Z[v_1,\ldots,v_{m'}]/\mathcal J'\ar[d]\\
  \Z[\sK_P]/\mathcal J \ar[r]_\cong^\varphi & \Z[\sK_{P'}]/\mathcal J',
}
\]
implying in particular that $m=m'$. The commutative diagram above
gives rise to an isomorphism~\eqref{quottor} by the standard
properties of~$\Tor$. More specifically, the isomorphism $\varphi$
gives an isomorphism of the Koszul algebras
\begin{equation}\label{isokosz}
  \widetilde\varphi\colon\bigl(\Lambda[u_1,\ldots,u_m]/\mathcal J\otimes\Z[\sK_P]/\mathcal
  J,d\bigr)\stackrel\cong\longrightarrow
  \bigl(\Lambda[u'_1,\ldots,u'_m]/\mathcal J'\otimes\Z[\sK_{P'}]/\mathcal
  J',d\bigr),
\end{equation}
where the ideals in the exterior algebras are defined by the same
linear forms as in the face rings. Then \eqref{quottor} is
obtained by passing to the cohomology.
\end{proof}

\begin{remark}
The argument above is essentially~\cite[Lemma~3.7]{c-p-s10}.
The term ``B-rigidity'' was introduced in the last section of~\cite{c-p-s10}. However, the implication of Proposition~\ref{bcrig} was
erroneously stated there in the opposite direction:
``if $P$ is C-rigid, then it is B-rigid''. This was a confusion. It is not
known whether C-rigidity is equivalent to B-rigidity, and it is
unlikely to be true in general.
\end{remark}

\section{The Pogorelov class: flag $3$-polytopes without $4$-belts}
Recall that the Pogorelov class $\mathcal P$ consists of simple $3$-polytopes $P$ which are
flag and do not have $4$-belts (or, equivalently, simple $3$-polytopes $P\ne\varDelta^3$ without $3$- and $4$-belts). In this section we consider combinatorial properties of polytopes $P\in\mathcal P$ and cohomological properties of the corresponding moment-angle manifolds~$\zp$. The key statements here are Theorem~\ref{flagrig}, Theorem~\ref{4beltrig} and
Lemma~\ref{rig3cl}; they will be used in the proof of the main
results in the next section. More specific properties of Pogorelov polytopes are described in the Appendices.

The first property is straightforward:

\begin{proposition}\label{no34fa}
In a polytope $P\in\mathcal P$, there are no $3$-gonal or
$4$-gonal facets.
\end{proposition}

\begin{lemma}\label{2facets}
For any two
facets $F_i$ and $F_j$ in a polytope $P\in\mathcal P$, there is a vertex $x\notin F_i\cup F_j$.
\end{lemma}
\begin{proof}
Take any facet $F_\ell$ different from $F_i$ and~$F_j$. Then
$F_\ell$ has at most two common vertices with $F_i$ and at most
two common vertices with~$F_j$. On the other hand, $F_\ell$ has at
least $5$ vertices by the Proposition~\ref{no34fa}. Thus, at least
one vertex of $F_\ell$ does not lie in $F_i\cup F_j$.
\end{proof}

\begin{lemma}
In a flag $3$-polytope $P$, for any facet $F_i$ there is a facet
$F_j$ such that $F_i \cap F_j=\varnothing$.
\end{lemma}
\begin{proof}
By Proposition~\ref{34bp}~(a) the facet $F_i$ is surrounded by a $k$-belt~$\mathcal B_k$. Then $\partial P\setminus\mathcal B_k$ consists of two
connected components: one of them is the interior of $F_i$, and
the other contains the interior of a facet $F_j$ that we look for.
\end{proof}

Now we consider cohomology of moment-angle manifolds $\zp$ with coefficients in~$\Z$. We recall from
Proposition~\ref{3dimcoh} that $H^3(\zp)$ has a basis of cohomology
classes $[u_iv_j]=[u_jv_i]$ corresponding to pairs of non-adjacent
facets $F_i,F_j$.

\begin{proposition}\label{prod3pol}
Let $P$ be a simple $3$-polytope with $m$ facets and let
$\sK=\sK_P$ be its dual simplicial complex. In the notation of
Theorem~\ref{prodfsc}, we have
\[
  H^\ell(\zp)=\begin{cases}
  \widetilde H^{-1}(\sK_\varnothing)=\Z&\text{for }\ell =0,\\
  \bigoplus_{|I|=\ell-1}\widetilde H^0(\sK_I)\oplus
  \bigoplus_{|I|=\ell-2}\widetilde H^1(\sK_I)
  &\text{for }3\le\ell\le m,\\[2pt]
  \widetilde H^{2}(\sK)=\Z&\text{for }\ell =m+3,\\
  0&\text{otherwise.}\\
  \end{cases}
\]
In particular, $H^*(\zp)$ does not have torsion. Furthermore, all
nontrivial products in $H^*(\zp)$ are of the form
\[
  \widetilde H^0(\sK_I)\otimes\widetilde H^0(\sK_J)\to
  \widetilde H^1(\sK_{I\cup J}), \quad I\cap J=\varnothing,
\]
or
\[
  \widetilde H^0(\sK_I)\otimes\widetilde H^1(\sK_{[m]\setminus I})\to
  \widetilde H^2(\sK)=\Z.
\]
For the multigraded components of $H^*(\zp)$, these two cases
correspond to
\begin{align*}
  &H^{-(|I|-1),\,2I}(\zp)\otimes H^{-(|J|-1),\,2J}(\zp)\to
  H^{-(|I|+|J|-2),\,2(I\sqcup J)}(\zp),\\
  &H^{-(|I|-1),\,2I}(\zp)\otimes H^{-(m-|I|-2),\,2([m]\setminus I)}(\zp)\to
  H^{-(m-3),\,2[m]}(\zp)=\Z,
\end{align*}
where the latter is the Poincar\'e duality pairing.
\end{proposition}
\begin{proof}
This follows from Theorems~\ref{zkcoh},~\ref{HTor} and~\ref{prodfsc}.
\end{proof}

An element in a graded ring is called \emph{decomposable} if it
can be written as a sum of nontrivial products of elements of nonzero degree.

\begin{lemma}[{\cite[Proposition~6.3]{fa-wa}}]\label{proddeco}
Let $P$ be a flag $3$-polytope and $\sK$ its dual simplicial
complex. Then the ring $H^*(\zp)\cong\bigoplus_{J\subset[m]}
\widetilde H^*(\sK_J)$ is multiplicatively generated by
$\bigoplus_{J\subset[m]} \widetilde H^0(\sK_J)$.
\end{lemma}

To prove this lemma it is enough to show that each nontrivial
cohomology class in $\widetilde H^1(\sK_I)\subset H^*(\zp)$ is
decomposable or, equivalently, the product map
\[
  \bigoplus_{I=I_1\sqcup I_2}\widetilde H^0(\sK_{I_1})\otimes\widetilde H^0(\sK_{I_2})\to
  \widetilde H^1(\sK_{I})
\]
is surjective. This proof is quite technical. We include it in
Appendix~\ref{proofproddeco} for the reader's convenience.

\begin{lemma}\label{proddeco1}
A simple $3$-polytope $P\ne\varDelta^3$ with $m$ facets is flag if
and only if any nontrivial cohomology class in $H^{m-2}(\zp)$ is
decomposable. In particular, if $H^{m-2}(\zp)=0$ then either $P$
is flag or $P=\varDelta^3$.
\end{lemma}
\begin{proof}
Suppose that $P$ is not flag. Since $P\ne\varDelta^3$, it has a
$3$-belt~$\{F_{j_1},F_{j_2},F_{j_3}\}$. Equivalently, the dual
complex $\sK$ has a missing $3$-face $J=\{j_1,j_2,j_3\}$. It gives a
nonzero cohomology class $\alpha\in H^{-1,2J}(\zp)\subset
H^5(\zp)$. Consider the Poincar\'e duality pairing
\[
  H^{m-2}(\zp)\otimes H^{5}(\zp)\to H^{m+3}(\zp)=\Z,
\]
which specifies to
\[
  H^{-(m-4),2([m]\setminus J)}(\zp)\otimes H^{-1,2J}(\zp)\to  H^{-(m-3),\,2[m]}(\zp)=\Z
\]
(see Proposition~\ref{prod3pol}). Take $\beta\in
H^{-(m-4),2([m]\setminus J)}(\zp)\subset H^{m-2}(\zp)$ such that $\alpha\cdot\beta$ is a generator of $H^{-(m-3),\,2[m]}(\zp)=\Z$. By Theorem~\ref{prodfsc}, $H^{-(m-4),2([m]\setminus
J)}(\zp)=\widetilde H^0(\sK_{[m]\setminus J})$, and any element of
$\widetilde H^0(\sK_{[m]\setminus J})$ is indecomposable by
Proposition~\ref{prod3pol}. We have therefore found an indecomposable
element $\beta\in H^{m-2}(\zp)$.

Now suppose that $P$ is flag. By Proposition~\ref{prod3pol},
\[
  H^{m-2}(\zp)= \bigoplus_{|I|=m-3}\widetilde H^0(\sK_I)\oplus
  \bigoplus_{|I|=m-4}\widetilde H^1(\sK_I),
\]
Consider the Poincar\'e duality pairing $\widetilde
H^0(\sK_I)\otimes \widetilde H^1(\sK_{[m]\setminus
  I})\to\Z$.
Since $\sK$ is flag, $\widetilde H^1(\sK_{[m]\setminus I})=0$ for
$|I|=m-3$ (as there are no missing faces with $3$ vertices). Hence,
$\bigoplus_{|I|=m-3}\widetilde H^0(\sK_I)$=0 by Poincar\'e
duality, and $H^{m-2}(\zp)= \bigoplus_{|I|=m-4}\widetilde
H^1(\sK_I)$. Then each nonzero element of $H^{m-2}(\zp)$ is
decomposable by Lemma~\ref{proddeco}.
\end{proof}

\begin{theorem}\label{flagrig}
Let $P$ be a flag $3$-polytope, and assume given a ring
isomorphism $H^*(\zp)\cong H^*(\mathcal Z_{P'})$ for another
simple $3$-polytope~$P'$. Then $P'$ is also flag.

In other words, the property of being a flag $3$-polytope is
$B$-rigid.
\end{theorem}
\begin{proof}
We have $P'\ne\varDelta^3$, as a $3$-simplex is $B$-rigid. Suppose that $P'$ is not flag. By Lemma~\ref{proddeco1}, there is an indecomposable element in $H^{m-2}(\mathcal{Z}_{P'})$. Then the same holds for~$P$, which is a contradiction.
\end{proof}

\begin{remark}
Theorem~\ref{flagrig} also follows from~\cite[Theorem~6.6]{fa-wa}.
\end{remark}

\begin{proposition}\label{Massey}
Let $P$ be a simple $3$-polytope.
\begin{itemize}
\item[(a)]
The product $H^3(\zp)\otimes H^3(\zp)\to H^6(\zp)$ is trivial if
and only if $P$ does not have $4$-belts.
\item[(b)]
The triple Massey product $H^3(\zp)\otimes H^3(\zp)\otimes
H^3(\zp)\to H^8(\zp)$ is trivial if $P$ does not have $4$-belts.
\end{itemize}
\end{proposition}
\begin{proof}
We first prove (a). Suppose $P$ has a $4$-belt
$(F_1,F_2,F_3,F_4)$. It corresponds to a chordless $4$-cycle
$\{1,2,3,4\}$ in $\sK=\sK_P$, i.\,e. a cycle with
$\{1,3\}\notin\sK$ and $\{2,4\}\notin\sK$. Hence, we have a
nontrivial product $\widetilde H^0(\sK_{\{1,3\}})\otimes\widetilde
H^0(\sK_{\{2,4\}})\to\widetilde H^1(\sK_{\{1,2,3,4\}})$, and a
nontrivial product $H^3(\zp)\otimes H^3(\zp)\to H^6(\zp)$.

Now suppose there is a nontrivial product $H^3(\zp)\otimes
H^3(\zp)\to H^6(\zp)$. We have
$H^6(\zp)=\bigoplus_{|I|=5}\widetilde
H^0(\sK_I)\oplus\bigoplus_{|I|=4}\widetilde H^1(\sK_I)$. Elements
of $\widetilde H^0(\sK_I)$ are indecomposable. An element of
$\widetilde H^1(\sK_I)$ with $|I|=4$ can be decomposed into a
product if and only if $I$ can be split into two pairs of
non-adjacent vertices, which means that $I$ is a chordless
$4$-cycle. It corresponds to a $4$-belt in~$P$.

To prove (b), assume that there is a nontrivial Massey product
$\langle\alpha,\beta,\gamma\rangle\in H^8(\zp)$. Then, by
Theorem~\ref{desu}, the graph $\sK_P^1$ contains an induced
subgraph isomorphic to one of the five graphs in Figure~\ref{5gr}.
By inspection, each of these five graphs has a chordless $4$-cycle
(the outer cycle for the first four graphs, and the left cycle for
the last one). Hence, the polytope $P$ has a $4$-belt.
\end{proof}

It is not known whether moment-angle manifolds of polytopes from the Pogorelov class $\mathcal P$ have nontrivial Massey products of cohomology
classes of dimension $>3$ or of order~$>3$, or whether these
moment-angle manifolds are formal in the sense of rational
homotopy theory. For general polytopes~$P$, there are examples of nontrivial
Massey products of any order in~$H^*(\zp)$, see~\cite{limo}.

\begin{theorem}\label{4beltrig}
Let $P$ be a simple $3$-polytope without $4$-belts, and assume
given a ring isomorphism $H^*(\zp)\cong H^*(\mathcal Z_{P'})$ for
another simple $3$-polytope~$P'$. Then $P'$ also does not have
$4$-belts.

It other words, the property of being a simple $3$-polytope without
$4$-belts is $B$-rigid.
\end{theorem}
\begin{proof}
This follows from Proposition~\ref{Massey} (a).
\end{proof}

Recently Fan, Ma and Wang proved that any polytope $P\in\mathcal P$ is $B$-rigid, see~\cite[Theorem 3.1]{f-m-w}. The
proof builds upon the following crucial lemma:

\begin{lemma}[{\cite[Corollary 3.4]{f-m-w}}]\label{rig3cl}
Consider the set of cohomology classes
\[
  \mathcal T(P)=\{\pm[u_iv_j]\in H^3(\zp),\quad F_i\cap
  F_j=\varnothing\}.
\]
If $P\in\mathcal P$, then for any
cohomology ring isomorphism $\psi\colon
H^*(\zp)\stackrel\cong\longrightarrow H^*(\mathcal Z_{P'})$, we
have $\psi(\mathcal T(P))=\mathcal T(P')$.
\end{lemma}

Note that the lemma above does not hold for all simple $3$-polytopes. For
example, if $P$ is a $3$-cube with the pairs of opposite facets
$\{F_1,F_4\}$, $\{F_2,F_5\}$, $\{F_3,F_6\}$, then $\zp\cong
S^3\times S^3\times S^3$ and there is an isomorphism $\psi\colon
H^*(\zp)\stackrel\cong\longrightarrow H^*(\mathcal Z_{P})$ which
maps $[u_1v_4]$ to $[u_1v_4]+[u_2v_5]$.

We include the proof of Lemma~\ref{rig3cl} in
Appendix~\ref{proofrig3cl} for the reader's convenience, and also
because some details were missing in the original argument. Note
that this proof uses Theorem~\ref{flagrig} and
Theorem~\ref{4beltrig}.

\section{Main results}
Here we prove the cohomological rigidity for small covers and quasitoric manifolds
over $3$-polytopes from the Pogorelov class~$\mathcal P$. We start with a crucial
lemma.

\begin{lemma}\label{crid2cl}
In the notation of Theorem~\ref{cohomqtoric}, consider the set of
cohomology classes
\[
  \mathcal D(M)=\{\pm[v_i]\in H^2(M),\quad i=1,\ldots,m\}.
\]
If $P\in\mathcal P$ then, for any cohomology ring isomorphism
$\varphi\colon H^*(M)\stackrel\cong\longrightarrow H^*(M')$ of
quasitoric manifolds over $P$ and $P'$, we have $\varphi(\mathcal
D(M))=\mathcal D(M')$. Furthermore, each of the sets $\mathcal
D(M)$ and $\mathcal D(M')$ consists of $2m$ different elements.
\end{lemma}
\begin{proof}
The idea is to reduce the statement to Lemma~\ref{rig3cl}. The
ring isomorphism $\varphi$ is determined uniquely by the
isomorphism $H^2(M)\stackrel\cong\longrightarrow H^2(M')$ of free
abelian groups. Let $\varphi([v_i])=\sum_{j=1}^m A_{ij}[v'_j]$ for
some $A_{ij}\in\Z$, $1\le i,j\le m$. The elements $A_{ij}$ are not
defined uniquely as there are linear relations between the classes
$[v'_j]$ in $H^2(M')$. To get rid of this indeterminacy, we can
choose a vertex $x=F_{i_1}\cap F_{i_2}\cap F_{i_3}$ of $P$ and a
vertex $x'=F'_{p_1}\cap F'_{p_2}\cap F'_{p_3}$ of $P'$. Then the
complementary cohomology classes $[v_i]$ with
$i\notin\{i_1,i_2,i_3\}$ form a basis in $H^2(M)$ and the
cohomology classes $[v'_p]$ with $p\notin\{p_1,p_2,p_3\}$ form a
basis in $H^2(M')$, so we have
\begin{equation}\label{phivi}
  \varphi([v_i])=\sum_{p\notin\{p_1,p_2,p_3\}} B_{ip}[v'_p], \qquad
  i\in[m]\setminus\{i_1,i_2,i_3\},
\end{equation}
with uniquely defined $B_{ip}\in\Z$ for $i
\in[m]\setminus\{i_1,i_2,i_3\}$, \ $p
\in[m]\setminus\{p_1,p_2,p_3\}$.

As we have seen in the proof of Proposition~\ref{bcrig}, the isomorphism $\varphi$ gives an
isomorphism $\psi\colon H^*(\zp)\stackrel\cong\longrightarrow
H^*(\mathcal Z_{P'})$, which is obtained from~\eqref{isokosz} by
passing to the cohomology. We write~\eqref{isokosz} as
$\widetilde\varphi\colon
C(P,\varLambda)\stackrel\cong\longrightarrow C(P',\varLambda')$.
This isomorphism is defined on the exterior generators $u_i$ and the polynomial
generators $v_i$ of the Koszul algebra $C(P,\varLambda)$ by the
same formulae as~$\varphi$.

Now take a cohomology class $[u_iv_j]\in H^3(\zp)$. By
Lemma~\ref{rig3cl}, it is mapped under $\psi$ to an element
$\varepsilon[u_k'v_l']\in H^3(\mathcal Z_{P'})$, $\varepsilon=\pm
1$. Choose vertices $x=F_{i_1}\cap F_{i_2}\cap F_{i_3}$ of $P$ and
$x'=F_{p_1}'\cap F_{p_2}'\cap F_{p_3}'$ of $P'$ such that $x\notin
F_i\cup F_j$ and $x'\notin F_k'\cup F_l'$ (see
Lemma~\ref{2facets}). We use the vertices $x$ and $x'$ to choose
bases in the groups $H^2(M)$ and $H^2(M')$ as described in the
first paragraph of the proof. Then we have
\[
  \psi[u_iv_j]=\sum_{p,q\notin\{p_1,p_2,p_3\}}B_{ip}B_{jq}[u_p'v_q'].
\]
On the other hand, we have $\psi[u_iv_j]=\varepsilon[u_k'v_l']$ by
Lemma~\ref{rig3cl}. Hence,
\[
  a=\sum_{p,q\notin\{p_1,p_2,p_3\}}B_{ip}B_{jq}u_p'v_q'-\varepsilon
  u_k'v_l'
\]
is a coboundary in $C(P',\varLambda')$, so there exists
\[
  c=\sum_{p,q\notin\{p_1,p_2,p_3\},\;p<q}L_{pq}u'_pu'_q, \qquad dc=a.
\]
We have
\[
  dc=\sum_{p,q\notin\{p_1,p_2,p_3\},\;p<q}L_{pq}(u'_qv'_p-u'_pv'_q).
\]
Comparing this with the expression for $a$ we obtain the following
relations between the coefficients:
\begin{equation}\label{Bijeq}
\begin{aligned}
  &B_{ip}B_{jq}=-B_{iq}B_{jp}=-L_{pq}\quad\text{for }p<q\text{ and }\{p,q\}\ne\{k,l\};\\
  &B_{ik}B_{jl}-\varepsilon=-B_{il}B_{jk}=\begin{cases}-L_{kl}&\text{ if }k<l,\\
                                        L_{lk}&\text{ if }l<k;\end{cases}\\
  &B_{ip}B_{jp}=0.
\end{aligned}
\end{equation}
From the third equation of~\eqref{Bijeq} we have, for any
$p\in[m]\setminus\{p_1,p_2,p_3\}$, either $B_{ip}=0$ or $B_{jp}=0$.
The first equation of~\eqref{Bijeq} implies that for
$\{p,q\}\ne\{k,l\}$ the vectors $\begin{pmatrix}B_{ip}\\
B_{jp}\end{pmatrix}$ and $\begin{pmatrix}B_{iq}\\
-B_{jq}\end{pmatrix}$ are linearly dependent. Hence, for
$\{p,q\}\ne\{k,l\}$, either one of the vectors
$\begin{pmatrix}B_{ip}\\ B_{jp}\end{pmatrix}$ and
$\begin{pmatrix}B_{iq}\\ B_{jq}\end{pmatrix}$ is zero, or both
vectors are nonzero and have a zero entry on the same place. From
the second equation $B_{ik}B_{jl}+B_{il}B_{jk}=\varepsilon$ we see
that both vectors $b_k=\begin{pmatrix}B_{ik}\\
B_{jk}\end{pmatrix}$ and $b_l=\begin{pmatrix}B_{il}\\
B_{jl}\end{pmatrix}$ are nonzero. If
there is a nonzero vector $b_p=\begin{pmatrix}B_{ip}\\
B_{jp}\end{pmatrix}$ for some $p\notin\{k,l\}$, then by
considering the pairs $(b_p,b_k)$ and $(b_p,b_l)$ we see that both
$b_k$ and $b_l$ have zero on the same place, which contradicts the
second equation of~\eqref{Bijeq}. It follows that $B_{ip}=B_{jp}=0$ for any $p\notin\{k,l\}$,
and
\[
  \begin{pmatrix}B_{ik}&B_{il}\\
  B_{jk}&B_{jl}\end{pmatrix}=\begin{pmatrix}B_{ik} & 0\\
  0&\frac{\varepsilon}{B_{ik}}\end{pmatrix}\quad\text{ or }\quad
  \begin{pmatrix}B_{ik}&B_{il}\\B_{jk}&B_{jl}\end{pmatrix}=
  \begin{pmatrix}0&B_{il}\\\frac{\varepsilon}{B_{il}}&0\end{pmatrix}.
\]
Since all entries are integer, we have $B_{ik}=\pm1$ and
$B_{il}=\pm1$. Then~\eqref{phivi} gives
$\varphi([v_i])\in\{\pm[v_k'],\pm [v_l']\}$. It follows that
$\varphi(\mathcal D(M))\subset\mathcal D(M')$.

It remains to show that each of the sets $\mathcal D(M)$,
$\mathcal D(M')$ consists of $2m$ different elements. For this we
note that $[v_i]\ne\pm[v_j]$ in $H^2(M)$ for any $i\ne j$. Indeed,
by Lemma~\ref{2facets} we can choose a vertex $x\notin F_i\cup
F_j$. Then both $[v_i]$ and $[v_j]$ belong to a basis of $H^2(M)$.
Now, since $\varphi$ is an isomorphism, we also have
$\varphi([v_i])\ne\pm\varphi([v_j])$ in $H^2(M')$. Thus, each of
the sets $\mathcal D(M)$ and $\mathcal D(M')$ consists of $2m$
elements.
\end{proof}

It follows from the Steinitz Theorem that any toric manifold of complex dimension $3$
is a quasitoric manifold. Also, the family of quasitoric manifolds agrees with that of topological toric manifolds in real dimension $6$ if we forget the actions.

Now we state the first main result.

\begin{theorem}\label{3cohrig}
Let $M=M(P,\varLambda)$ and $M'=M(P',\varLambda')$ be
quasitoric manifolds over $3$-dimensional simple polytopes $P$ and
$P'$, respectively. Assume that $P$ belongs to the Pogorelov class~$\mathcal P$.
Then the following conditions are equivalent:
\begin{itemize}
\item[(a)] there is a cohomology ring isomorphism $\varphi\colon H^*(M)
\stackrel\cong\longrightarrow H^*(M')$;
\item[(b)] there is a diffeomorphism $M\cong M'$;
\item[(c)] there is an equivalence of characteristic pairs $(P,\varLambda)\sim(P',\varLambda')$.
\end{itemize}
\end{theorem}
\begin{proof}
The implication (b)$\Rightarrow$(a) is obvious. The implication (c)$\Rightarrow$(b) follows from Proposition~\ref{qteqd}. We need to prove (a)$\Rightarrow$(c).

By Lemma~\ref{crid2cl}, $\varphi([v_i])=\pm[v'_{\sigma(i)}]$,
where $\sigma$ is a permutation of the set~$[m]$. Renumbering the facets and multiplying the matrix
$\varLambda$ from the right by a matrix $B$ as in
Definition~\ref{defcharpair}, we may
assume that $\varphi([v_i])=v_i'$; this does not change the equivalence
class of the pair $(P,\varLambda)$. Then
$\varphi[v_iv_j]=[v'_iv'_j]$. By Corollary~\ref{viprod},
$[v_iv_j]=0$ in $H^*(M)$ if and only if $F_i\cap F_j=\varnothing$
and $[v_iv_jv_k]=0$ in $H^*(M)$ if and only if $F_i\cap F_j\cap
F_k=\varnothing$ in $P$, and the same holds for $H^*(M')$ and
$P'$. It follows that $\sK_P$ is isomorphic to~$\sK_{P'}$. Hence,
$P$ and $P'$ are combinatorially equivalent.

Now consider the $(3\times m)$-matrices $\varLambda$ and
$\varLambda'$. First, by changing the order of facets in $P$ and
$P'$ if necessary we may assume that $F_1\cap F_2\cap
F_3\ne\varnothing$ in~$P$ and $F'_1\cap F'_2\cap
F'_3\ne\varnothing$ in~$P'$. Then, by multiplying the matrices
$\varLambda$ and $\varLambda'$ from the left by appropriate
matrices from $\mbox{\textrm{GL}}(3,\Z)$ we may assume that
\[
  \varLambda=\begin{pmatrix}
  1&0&0&\lambda_{14}&\cdots& \lambda_{1m}\\
  0&1&0&\lambda_{24}&\cdots& \lambda_{2m}\\
  0&0&1&\lambda_{34}&\cdots& \lambda_{3m}
  \end{pmatrix},\qquad
  \varLambda'=\begin{pmatrix}
  1&0&0&\lambda'_{14}&\cdots& \lambda'_{1m}\\
  0&1&0&\lambda'_{24}&\cdots& \lambda'_{2m}\\
  0&0&1&\lambda'_{34}&\cdots& \lambda'_{3m}
  \end{pmatrix}.
\]
This does not change the equivalence class of pairs
$(P,\varLambda)$ and $(P',\varLambda')$. Now the entries
$\lambda_{jk}$, $4\le k\le m$, are the coefficients in the
expression of $[v_j]$, $1\le j\le 3$, via the basis
$[v_4],\ldots,[v_m]$ of~$H^2(M)$. The same holds for the
$\lambda'_{jk}$. Since $\varphi([v_i])=v_i'$, it follows that
$\lambda_{jk}=\lambda'_{jk}$. Thus, the pairs $(P,\varLambda)$ and
$(P',\varLambda')$ are equivalent.
\end{proof}

\begin{remark}
Any smooth structure on a quasitoric manifold $M$ over a polytope $P\in\mathcal P$ is equivalent to
the standard one defined on the canonical model $M(P,\varLambda)$ via Proposition~\ref{qteqd}. This follows from the general classification results for $6$-dimensional manifolds, see Corollary~\ref{homeodif}.
\end{remark}

\begin{corollary} \label{coro:5-1}
Toric, quasitoric and topological toric manifolds over polytopes from the Pogorelov class $\mathcal P$
are cohomologically rigid.
\end{corollary}

\begin{remark}
Theorem~\ref{3cohrig} says that a cohomology ring isomorphism of quasitoric manifolds over polytopes $P\in\mathcal P$ implies not only a diffeomorphism of manifolds, but also an equivalence of characteristic pairs. The latter is not true for quasitoric manifolds over arbitrary polytopes. For example, consider the \emph{Hirzebruch surfaces}
$H_k=\mathbb C P(\mathcal O(k)\oplus\underline{\mathbb C})$, where
$\mathcal O(k)$ is the $k$th power of the canonical line bundle
over $\C P^1$, $\underline{\mathbb C}$ is a trivial line bundle,
and $\C P(-)$ denotes the complex projectivisation. Each
$H_k$ is a toric manifold, and it can also be described as the quasitoric manifold over a quadrangle with
characteristic matrix
\[
\begin{pmatrix}
  1&0&-1&k\\
  0&1&0&-1
\end{pmatrix}.
\]
Manifolds $H_k$ with even $k$ are all diffeomorphic to $S^2\times
S^2$, but the characteristic matrices corresponding to different
positive $k$ are not equivalent. Similar examples exist in all
dimensions, see e.g.~\cite{ma-pa08}.
\end{remark}

The family of quasitoric (or topological toric) manifolds over
$3$-polytopes from the Pogorelov class~$\mathcal P$ is large
enough, as there is at least one quasitoric manifold over any such
polytope by Proposition~\ref{4cprop} (recall that this result uses
the Four Colour Theorem). There are fewer toric manifolds in this
family. In fact, there are no \emph{projective} toric manifolds
over polytopes from~$\mathcal P$. The reason is that a Delzant
3-polytope must have at least one triangular or quadrangular face
by the result of C.\,Delaunay~\cite{dela05} (see
also~\cite{ayze}). On the other hand, there are non-projective
toric manifolds in this family, see~\cite{suya15}.

\medskip

Our second main result is about small covers (or hyperbolic $3$-manifolds).

\begin{theorem}\label{3cohrigreal}
Let $N=N(P,\varLambda)$ and $N'=N(P',\varLambda')$ be
small covers of $3$-dimensional simple polytopes $P$ and
$P'$, respectively. Assume that $P$ belongs to the Pogorelov class~$\mathcal P$, so $N$ is
a hyperbolic $3$-manifold of L\"obell type.
Then the following conditions are equivalent:
\begin{itemize}
\item[(a)] there is a cohomology ring isomorphism $\varphi\colon H^*(N;\Z_2)
\stackrel\cong\longrightarrow H^*(N';\Z_2)$;
\item[(b)] there is an isomorphism of fundamental groups $\pi_1(N)\cong\pi_1(N')$;
\item[(c)] there is a diffeomorphism $N\cong N'$;
\item[(d)] there is an equivalence of $\Z_2$-characteristic pairs $(P,\varLambda)\sim(P',\varLambda')$.
\end{itemize}
\end{theorem}
\begin{proof}
The implications (b)$\Rightarrow$(a) and (c)$\Rightarrow$(b) are obvious (in fact, the equivalence (b)$\Leftrightarrow$(c) follows from Mostow's rigidity theorem for hyperbolic manifolds).
The implication (d)$\Rightarrow$(c) follows from the real version of Proposition~\ref{qteqd}.

We need to prove the implication (a)$\Rightarrow$(d). Using Proposition~\ref{lambda2}
we upgrade $(P,\varLambda)$ and $(P',\varLambda')$ to $\Z$-characteristic pairs and consider the corresponding quasitoric manifolds $M=M(P,\varLambda)$ and $M'=M(P',\varLambda')$. Since the cohomology ring $H^*(M;\Z_2)$ is obtained from $H^*(N;\Z_2)$ by doubling the grading (see Theorem~\ref{cohomsc}), we have an isomorphism $H^*(M;\Z_2)
\stackrel\cong\longrightarrow H^*(M';\Z_2)$. Now the equivalence of characteristic pairs follows from Theorem~\ref{3cohrig} (with coefficients in $\Z_2$).
\end{proof}

\begin{example}\label{barrelex}
For $k\ge 5$, let $Q_k$ be a simple $3$-polytope with two ``top''
and ``bottom'' $k$-gonal facets and $2k$ pentagonal facets forming
two $k$-belts around the top and bottom, so that $Q_k$ has $2k+2$
facets in total. Note that $Q_5$ is a combinatorial dodecahedron,
while $Q_6$ is a fullerene, see Figure~\ref{barrel}. It is easy to
see that $Q_k\in\mathcal P$, so it admits a right-angled
realisation in~$\mathbb L^3$. The hyperbolic $3$-manifolds
$N(Q_k,\chi)$ corresponding to regular $4$-colourings $\chi$ of
$Q_k$ (as described by Proposition~\ref{4cprop}) were studied by
Vesnin in~\cite{vesn87}. For example, a dodecahedron $Q_5$ has a
unique regular $4$-colouring up to equivalence, while $Q_6$ has
four non-equivalent regular $4$-colourings  (with $4$-colourings
being equivalent if they differ by a permutation of colours).
\begin{figure}[h]
\begin{center}
\includegraphics[scale=0.4]{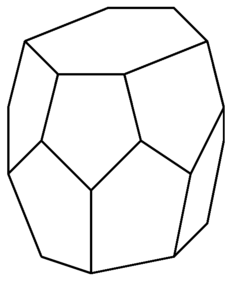}\qquad\qquad
\includegraphics[scale=0.14]{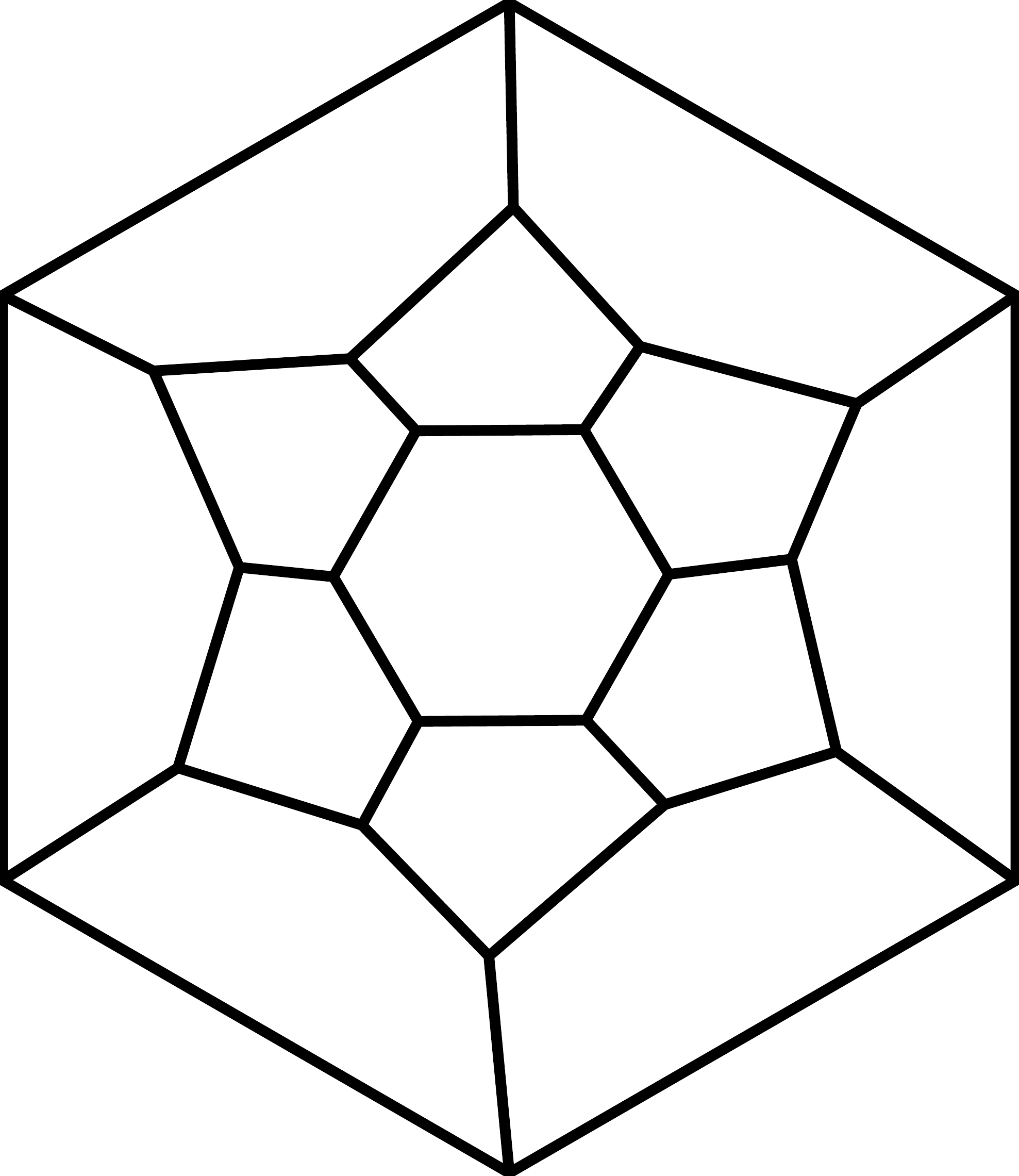}
\end{center}
\caption{The fullerene $Q_6$ and its Schlegel diagram}
\label{barrel}
\end{figure}

Vesnin conjectured the following: the manifolds $N(Q_k,\chi)$ and
$N(Q_k,\chi')$ are isometric if and only if the $4$-colourings
$\chi$ and $\chi'$ are equivalent. In~\cite{vesn91} this
conjecture was proved for those polytopes $Q_k$ whose
corresponding hyperbolic reflection groups $G(Q_k)$ are
\emph{non-arithmetic} (as subgroups of the full isometry group
of~$\mathbb L^3$). The proof used the Margulis
theorem~\cite{marg74} on the discreteness of the commensurator of
a non-arithmetic group. It was eventually proved in~\cite{a-m-r}
that the group $G(Q_k)$ is non-arithmetic for all $k$ except $5$,
$6$ and~$8$. As $Q_5$ has a unique $4$-colouring, Vesnin's
conjecture has remained open only for $k=6,8$.

Theorem~\ref{3cohrigreal} verifies Vesnin's conjecture completely, and this argument does not use previous results on this conjecture:

\begin{corollary}
The hyperbolic manifolds $N(Q_k,\chi)$ and $N(Q_k,\chi')$ defined by regular $4$-colourings of the polytope $Q_k$, $k\ge5$, are isometric if and only if the $4$-colourings $\chi$ and $\chi'$ are equivalent.
\end{corollary}
\begin{proof}
Clearly, if the $4$-colourings $\chi$ and $\chi'$ are equivalent, then the corresponding hyperbolic manifolds are isometric. Conversely, if the manifolds are isometric, then they are diffeomorphic, and Theorem~\ref{3cohrigreal} implies that the corresponding characteristic matrices $\varLambda$ and $\varLambda'$ are equivalent (that is, $\varLambda'=A\varLambda$, where $A\in\mathrm{GL}_3(\Z_2)$). Now, according to a result of~\cite{bu-pa16}, equivalence of characteristic matrices defined by $4$-colourings implies equivalence of $4$-colourings.
\end{proof}
\end{example}

\section{Classification of 6-dimensional manifolds and related
problems}\label{class6}

The classification of smooth simply connected $6$-dimensional
manifolds with torsion-free homology was done in the works of
Wall~\cite{wall66} and Jupp~\cite{jupp73}. They also stated a
classification result in the topological category; the proof was
corrected later in the work of Zhubr~\cite{zhub00}. The latter
work also treated the case of homology with torsion. We only give
the following result, which will be enough for our purposes (the
cohomology is with integer coefficients, unless otherwise
specified).

\begin{theorem}[\cite{wall66}, \cite{jupp73}]\label{class1}
Let $\varphi\colon H^*(N)\stackrel\cong\longrightarrow H^*(N')$ be
an isomorphism of the cohomology rings of closed smooth simply
connected $6$-dimensional manifolds $N$ and $N'$ with
$H^3(N)=H^3(N')=0$. Assume that
\begin{itemize}
\item[(a)] $\varphi(w_2(N))=w_2(N')$, where $w_2(N)\in H^2(N;\Z_2)$
is the second Stiefel--Whitney class;

\item[(b)] $\varphi(p_1(N))=p_1(N')$, where $p_1(N)\in H^4(N)$ is
the first Pontryagin class.
\end{itemize}
Then the manifolds $N$ and $N'$ are diffeomorphic.
\end{theorem}

The following lemma is proved using Steenrod squares:

\begin{lemma}[{\cite[Lemma~8.1]{c-m-s10}}]\label{steen}
Assume that the ring $H^*(N;\Z_2)$ is generated by $H^k(N;\Z_2)$
for some $k>0$. Then any ring isomorphism
$\varphi\colon H^*(N;\Z_2)\stackrel\cong\longrightarrow
H^*(N';\Z_2)$ preserves the total Stiefel--Whitney class, i.\,e.
$\varphi(w(N))=w(N')$.
\end{lemma}

Lemma~\ref{steen} applies to toric or quasitoric manifolds, whose
cohomology is generated in degree two. From Theorem~\ref{class1}
we obtain

\begin{corollary}
Let $\varphi\colon H^2(M)\stackrel\cong\longrightarrow H^2(M')$ be
an isomorphism of second cohomology groups of $6$-dimensional
smooth quasitoric manifolds. Assume that
\begin{itemize}
\item[(a)] $\varphi$ preserves the cubic form
$H^2(M)\otimes H^2(M)\otimes H^2(M)\to\Z=H^6(M)$ given by the
cohomology multiplication;

\item[(b)] $\varphi$ preserves
the first Pontryagin class.
\end{itemize}
Then the manifolds $M$ and $M'$ are diffeomorphic.
\end{corollary}

From the topological invariance of rational Pontryagin classes (proved in general by S.\,P.~Novikov) we obtain

\begin{corollary}\label{homeodif}
Let $M$ and $M'$ be $6$-dimensional smooth quasitoric manifolds. If $M$ and $M'$ are homeomorphic, then they are diffeomorphic.
\end{corollary}

The characteristic classes of quasitoric manifolds are given as
follows:

\begin{proposition}[{\cite[Corollary~6.7]{da-ja91}}]
In the notation of Theorem~\ref{cohomqtoric}, the total
Stiefel--Whitney and Pontryagin classes of a quasitoric manifold
$M$ are given by
\[
  w(M)=\prod_{i=1}^m(1+v_i)\mod 2,\qquad
  p(M)=\prod_{i=1}^m(1+v^2_i).
\]
In particular, $w_2(M)=v_1+\cdots+v_m\mod 2$, and
$p_1(M)=v_1^2+\cdots+v_m^2$.
\end{proposition}

\begin{corollary}
A family of $6$-dimensional quasitoric manifolds is
cohomologically rigid if any cohomology ring isomorphism between
manifolds from the family preserves the first Pontryagin class.
\end{corollary}

This reduces cohomological rigidity for $6$-dimensional quasitoric
manifolds $M$ to a problem of combinatorics and linear algebra, as
both the cohomology ring $H^*(M)$ and the first Pontryagin class
$p_1(M)=v_1^2+\cdots+v_m^2$ are defined entirely in terms of the
characteristic pair~$(P,\varLambda)$.

Our result on cohomological rigidity for quasitoric manifolds over Pogorelov polytopes (Theorem~\ref{3cohrig}) gives a complete classification for this particular class of simply connected $6$-manifolds, and its proof is indepenent of the general classification results of Wall and Jupp.
The invariance of the first Pontryagin class for quasitoric
manifolds over Pogorelov polytopes follows directly
from Lemma~\ref{crid2cl}. It would be interesting to find a direct (combinatorial?) proof
of this fact. Bott towers (of any dimension) form
another family of toric manifolds for which the invariance of
Pontryagin classes under cohomology ring isomorphisms is known, see~\cite{c-m-m15}.

\begin{remark}
In dimension $4$ we have the identity $\langle p_1(M),[M]\rangle=3\mathop{\mathrm{sign}}(M)$, where $[M]\in H_4(M)$ is the fundamental class and $\mathop{\mathrm{sign}}(M)$ is the signature of~$M$. Therefore, $p_1$ is invariant under cohomology ring isomorphisms. When $M$ is a toric manifold, the signature is equal to $4-m$, where $m$ is the number of vertices in the quotient polygon~$P$ (see e.g.~\cite[Example~9.5.3]{bu-pa15}). The identity $\langle p_1(M),[M]\rangle=3\mathop{\mathrm{sign}}(M)$ then becomes
\[
  \langle v_1^2+\cdots+v_m^2,[M]\rangle=12-3m,
\]
which can be seen directly from Theorem~\ref{cohomtoric}.
\end{remark}

\begin{appendix}

\section{Belts in flag 3-polytopes}
Here we give proofs of two combinatorial lemmata on belts in flag
$3$-polytopes, originally due to~\cite{fa-wa} and~\cite{f-m-w}
respectively. These proofs are included mainly for the sake of completeness, but we also fill in some details missing in the original works. Lemma~\ref{Fab} is used in the proof of the product
decomposition lemma in Appendix~\ref{proofproddeco}, while
Lemma~\ref{Fabc} is used in the proof of rigidity of the
set of canonical generators of $H^3(\zp)$ in
Appendix~\ref{proofrig3cl}.

Recall that a belt of facets in a simple polytope $P$ corresponds to a
chordless cycle in the dual simplicial complex $\sK_P$, or to
a full subcomplex $(\sK_P)_I$ isomorphic to the boundary of a polygon.

\begin{lemma}\label{Fab}
Let $P$ be a flag $3$-polytope. Then for every three different facets $F_i$,
$F_{i'}$, $F_k$ with $F_i\cap F_{i'}=\varnothing$, there exists a belt
$\mathcal{B}$ such that $F_i,F_{i'}\in \mathcal{B}$ and $F_k\notin\mathcal B$.
\end{lemma}

We reformulate this lemma in the dual notation; this is how the lemma was stated and proved in~\cite{fa-wa}:

\begin{lemma}[{\cite[Lemma~6.1]{fa-wa}}] \label{Fabd}
Let $\sK$ be a flag triangulation of the disk $D^2$ with $m$ vertices, and let $S$ be the set of vertices of the boundary~$\partial\sK$. Assume that $\sK_{S}=\partial\sK$. Then for every missing edge $\{i,i'\}\notin\sK$ there exists a subset $I\subset[m]$ such that $\{i,i'\}\subset I$ and $\sK_I$ is a chordless cycle (the boundary of a polygon).
\end{lemma}

To obtain Lemma~\ref{Fab} from Lemma~\ref{Fabd} we take as $\sK$ the simplicial complement
to the vertex of $\sK_P$ corresponding to the facet $F_k\subset P$, that is, $\sK:=(\sK_P)_{[m]\setminus\{k\}}$. Then $\sK$ is a flag triangulation of $D^2$ (as a full subcomplex in the flag complex~$\sK_P$), and $\sK_{S}=\partial\sK$ because $\sK_P$ is flag. Lemma~\ref{Fabd} gives a chordless cycle $\sK_I$ in $\sK\subset\sK_P$, which corresponds to the required belt in~$P$.

The \emph{star} and \emph{link} of a vertex $\{i\}\in\sK$ are the subcomplexes
\[
  \st_\sK\{i\}=\{I\in\sK\colon \{i\}\cup I\in\sK\},\quad
   \lk_\sK\{i\}=\{I\in\sK\colon  \{i\}\cup I\in\sK, i\notin I\}.
\]

\begin{proof}[Proof of Lemma~\ref{Fabd}]
We use the induction on $m$, the number of vertices of $\sK$. Since $\sK$ is flag, $|S|\ge4$ and $m\ge5$. If $m=5$, then $|S|=4$ and $\sK$ is the cone over a square, so $\{i,i'\}\in S$ and we can take $I:=S$.

%
Now assume that the statement holds for simplicial complexes with $<m$ vertices.
If both vertices $i$ and $i'$ lie in $\partial\sK$, then $I:=S$ gives the required chordless cycle. Hence, we only need to consider the case $\{i,i'\}\not\subset S$.
Hence, $|S|<m-1$.
For a vertex $j\in S$, denote by $m_j$ the number of vertices in $\st_\sK\{j\}$. Then $m_j\ge 4$ for any $j\in S$, since $\sK_S=\partial\sK$. We consider several cases.

\smallskip

I. Suppose that there is a vertex $j\in S\setminus\{i,i'\}$ such that $m_j=4$. Then the set of vertices of $\st_\sK\{j\}$ is $\{j,j',j'',k\}$, where $j,j',j''\in S$ and $k\notin S$, see Figure~\ref{fig:star-4-vertex}.
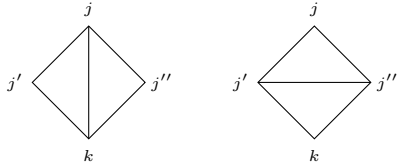
\begin{figure}[h]
        \begin{tikzpicture}[scale=.75]
            \draw (0,0)--(1,-1)--(2,0)--(1,1)--cycle;
            \draw (1,-1)--(1,1);
            \draw (-0.3,0) node{\tiny{$j'$}};
            \draw (1,-1.3) node{\tiny{$k$}};
            \draw (1,1.3) node{\tiny{$j$}};
            \draw (2.3,0) node{\tiny{$j''$}};
            \draw (0+4,0)--(1+4,-1)--(2+4,0)--(1+4,1)--cycle;
            \draw (4,0)--(6,0);
            \draw (-0.3+4,0) node{\tiny{$j'$}};
            \draw (1+4,-1.3) node{\tiny{$k$}};
            \draw (1+4,1.3) node{\tiny{$j$}};
            \draw (2.3+4,0) node{\tiny{$j''$}};
        \end{tikzpicture}
\caption{$\st_\sK\{j\}$ and its bistellar $1$-move}\label{fig:star-4-vertex}
\end{figure}

(i) If there is no vertex $k'\in S\setminus \{j,j',j''\}$ such that $\{k,k'\}$ is an edge of $\sK$, then the simplicial complex $\sK':=\sK_{[m]\setminus\{j\}}$ satisfies the hypothesis of the lemma. By the inductive hypothesis, there is a subset $I'$ of $[m]\setminus \{j\}$ such that $\{i,i'\}\subset I'$ and $\sK'_{I'}$ is a chordless cycle. Then $I:=I'$ is the required set, as $\sK_{I'}=\sK'_{I'}$.

(ii) Now assume that there exists a vertex $k'\in S\setminus \{j,j',j''\}$ such that $\{k,k'\}$ is an edge in~$\sK$. Let $\sK'$ be the simplicial complex obtained from $\sK$ by applying a bistellar $1$-move at $\st_\sK\{j\}$, see Figure~\ref{fig:star-4-vertex}. Then
$\sK'':=\sK'_{[m]\setminus\{j\}}$ satisfies the hypothesis of the lemma. By induction, there is a subset $I''$ of $[m]\setminus\{j\}$ such that $\{i,i'\}\subset I''$ and $\sK''_{I''}$ is a chordless cycle. If $j'$ or $j''$ is not in $I''$, then $I:=I''$ is the required set. If both $j'$ and $j''$ are in $I''$, then
$I:=I''\cup\{j\}$ is the required set.

\smallskip

II. Suppose that $m_j>4$ for every $j\in S\setminus\{i,i'\}$. Let $S=\{j_1,\ldots,j_n\}$, ordered
counterclockwise, and assume that $j_1\notin\{i,i'\}$. Let $\mathcal V_{j_p}$ denote the set of vertices of $\st_\sK(j_p)$, so $|\mathcal V_{j_p}|=m_{j_p}$, for $1\le p\le n$. Note that if $j_p\in S\setminus\{i,i'\}$, then $m_{j_p}>4$ and $|\mathcal{V}_{j_p}\setminus S|>1$.

(i) Assume that, for some $j_p\in S\setminus \{i,i'\}$, there is no edge $\{k,k'\}$ in $\sK$ such that
\[\leqno{(\ast)}\qquad
   k\in \mathcal{V}_{j_p}\setminus S\quad\text{and}\quad
   k'\in S\setminus\{j_{p-1},j_p,j_{p+1}\},\quad\text{where }j_0=j_n.
\]
Then $\sK':=\sK_{[m]\setminus\{j_p\}}$ satisfies the hypothesis of the lemma, so we can find the required subset $I$ of $[m]\setminus\{j_p\}$.

(ii) Assume that, for every $j_p\in S\setminus \{i,i'\}$, there is an edge $\{k_p,j_{q_p}\}$ in $\sK$ satisfying~$(\ast)$ for $k=k_p$ and $k'=j_{q_p}$. We shall lead this case to a contradiction. Set $I_1:=\{j_1,k_1,j_{q_1}\}$. Then $\sK_{I_1}$ divides $\sK$ into two simplicial complexes $\sK_1$ and $\sK_2$,
where $\sK_1$ has boundary vertices $j_1,\ldots,j_{q_1},k_1$, and $\sK_2$ has boundary vertices $j_{q_1},\ldots,j_n,j_1,k_1$, see Figure~\ref{fig:complex from I_1}.
    \begin{figure}[h]
        \begin{tikzpicture}
            \draw (1.75,0.75)--(2,0.5)--(3,0)--(3,1.5)--(3,3)--(2,2.5)--(1.75,2.25);
            \draw[dotted] (1.5,2)--(1.5,1);
            \draw (3,3)--(4,2.5)--(4.25,2.25);
            \draw[dotted] (4.5,2)--(4.5,1);
            \draw (3,0)--(4,0.5)--(4.25,0.75);
            \fill (3,0) circle(2pt);
            \fill (3,3) circle(2pt);
            \fill (3,1.5) circle(2pt);
            \fill (2,0.5) circle(2pt);
            \fill (4,0.5) circle(2pt);
            \fill (2,2.5) circle(2pt);
            \fill (4,2.5) circle(2pt);
            \draw (3,-0.3) node{\tiny{$j_{q_1}$}};
            \draw (3,3.3) node{\tiny{$j_1$}};
            \draw (3.3,1.5) node{\tiny{$k_1$}};
            \draw (2.25,0.75) node{\tiny{$j_{q_1-1}$}};
            \draw (3.75,0.75) node{\tiny{$j_{q_1+1}$}};
            \draw (2,2.25) node{\tiny{$j_2$}};
            \draw (4,2.25) node{\tiny{$j_n$}};
        \end{tikzpicture}
        \caption{$\sK_{I_1}$ divides $\sK$ into two simplicial complexes.}\label{fig:complex from I_1}
    \end{figure}
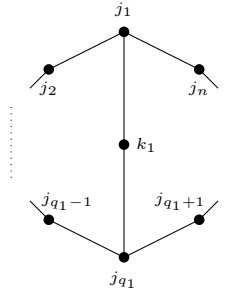

Since $\{i,i'\}\not\subset S$, either $\{i,i'\}\cap\{j_1,\ldots,j_{q_1-1}\}=\varnothing$ or
$\{i,i'\}\cap\{j_{q_1},\ldots,j_{n}\}=\varnothing$. Without loss of generality, assume that
$\{i,i'\}\cap\{j_1,\ldots,j_{q_1-1}\}=\varnothing$. Then $m_{j_p}>4$ for $1\le p\le q_1-1$.
By the flagness of $\sK$ and the condition for the existence of an edge satisfying~$(\ast)$, there is no vertex $k\in[m]\setminus S$ such that $k$ is connected to the vertices $j_p$ and $j_{p+2}$ for $1\le p \le q_1-2$. This implies in particular that $q_1>3$.

Now consider the path from $j_2$ to $k_2$ and to $j_{q_2}$. If $k_2=k_1$, then we may assume that $j_{q_2}=j_{q_1}$. Otherwise, $k_2$ must be contained in the simplicial complex~$\sK_1$. In either case, the the path $j_2-k_2-j_{q_2}$  is contained in the simplicial subcomplex $\sK_1$ with boundary vertices $j_1,\ldots,j_{q_1},k_1$. Proceeding inductively, we obtain that the path
$j_p-k_p-j_{q_p}$ is contained in the simplicial subcomplex whose boundary vertices are $j_{p-1},\ldots,j_{q_{p-1}},k_{p-1}$, see Figure~\ref{fig:path}.
It follows that $p<q_p\le q_{p-1}\le\cdots\le q_1$. Eventually we obtain $p$ such that $q_p=p+2$, so the vertex $k_p$ is connected to the vertices $j_{p}$ and $j_{p+2}$. This is a contradiction.
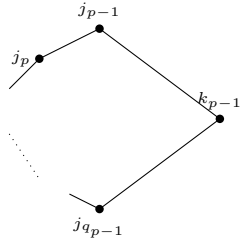
\begin{figure}[h]
        \begin{tikzpicture}[scale=.8]
            \draw (0.5,2)--(1,2.5)--(2,3)--(4,1.5)--(2,0)--(1.5,0.25);
            \draw[dotted] (0.5,1.25)--(1,0.5);
            \fill (1,2.5) circle(2pt);
            \fill (2,3) circle(2pt);
            \fill (4,1.5) circle(2pt);
            \fill (2,0) circle(2pt);
            \draw (0.7,2.5) node{\tiny{$j_p$}};
            \draw (2,3.3) node{\tiny{$j_{p-1}$}};
            \draw (4,1.8) node{\tiny{$k_{p-1}$}};
            \draw (2,-.3) node{\tiny{$j_{q_{p-1}}$}};
        \end{tikzpicture}
        \caption{The path $j_p-k_p-j_{q_p}$ is contained in the above simplicial complex.}\label{fig:path}
\end{figure}

From I and II, the lemma is proved.
\end{proof}

\begin{lemma}[{\cite[Lemma 3.2]{f-m-w}}] \label{Fabc}
Let $P$ be a flag $3$-polytope without $4$-belts. Then for every
three different facets $F_i$, $F_{i'}$, $F_k$ with $F_i\cap
F_{i'}=\varnothing$ there is a belt $\mathcal B$ such that
$F_i,F_{i'}\in\mathcal B$, $F_k\notin\mathcal B$, and $F_k$ does not
intersect at least one of the two connected components of
$\mathcal{B}\setminus\{F_i,F_{i'}\}$.
\end{lemma}
\begin{proof}
We work with the dual simplicial complex $\sK=\sK_P$, which is a triangulated 2-sphere. We need to find a subset $I\subset[m]\setminus\{k\}$ such that $\{i,i'\}\subset I$, $\sK_I$ is a chordless cycle, and $\widetilde H^0(\sK_{(I\setminus\{i,i'\})\cup\{k\}})\ne0$.
By Lemma~\ref{Fabd}, there is a subset $I_0$ of $[m]\setminus\{k\}$ such that $\{i,i'\}\subset I_0$ and $\sK_{I_0}$ is a chordless cycle. We construct the required subset $I$ by modifying $I_0$.

\begin{figure}[h]
\begin{subfigure}[b]{0.3\textwidth}
      \centering
      \begin{tikzpicture}[scale=1]
        \filldraw[fill=gray!10] (0,0) circle(1cm);
        \fill (0.86,0.5) circle(2pt);
        \fill (0.5,0.86) circle(2pt);
        \fill (-0.5,0.86) circle(2pt);
        \fill (-0.86,0.5) circle(2pt);
        \fill (-0.86,-0.5) circle(2pt);
        \fill (-0.5,-0.86) circle(2pt);
        \fill (0.5,-0.86) circle(2pt);
        \fill (0,1) circle(2pt);
        \fill (0,-1) circle(2pt);
        \draw (0,1.2) node{\tiny$i$};
        \draw (0,-1.2) node{\tiny$i'$};
        \fill  (1,0) circle(2pt) (0.86,-0.5) circle(2pt) (-1,0) circle(2pt);
        \draw (-1.5,0.5) node{$\sK_{I_0}$};
      \end{tikzpicture}
      \end{subfigure}
\begin{subfigure}[b]{0.3\textwidth}
      \centering
      \begin{tikzpicture}[scale=1]
        \filldraw[fill=gray!10] (0,0) circle(1cm);
        \fill (0,0) circle(2pt);
        \fill (0.86,0.5) circle(2pt);
        \fill (0.5,0.86) circle(2pt);
        \fill (-0.5,0.86) circle(2pt);
        \fill (-0.86,0.5) circle(2pt);
        \fill (-0.86,-0.5) circle(2pt);
        \fill (-0.5,-0.86) circle(2pt);
        \fill (0.5,-0.86) circle(2pt);
        \fill (0,1) circle(2pt);
        \fill (0,-1) circle(2pt);
        \fill  (1,0) circle(2pt) (0.86,-0.5) circle(2pt) (-1,0) circle(2pt);
        \draw (0,1.2) node{\tiny$i$};
        \draw (0,-1.2) node{\tiny$i'$};
        \draw[line width=7pt,black!50,cap=round,opacity=.5] (-0.5,0.86) arc(120:240:1cm);
        \draw[line width=7pt,black!50,cap=round,opacity=.5] (0.5,-0.86) arc(-60:60:1cm);
        \draw (-1.5,0) node{$\widetilde X$};
        \draw (1.5,0) node{$\widetilde Y$};
        \draw (0,-0.2) node{\tiny$k$};
        \draw (0.5,0.3) node{$\sK_{\rm in}$};
        \draw (1.5,1) node{$\sK_{\rm out}$};
      \end{tikzpicture}
      \end{subfigure}
\begin{subfigure}[b]{0.3\textwidth}
      \centering
      \begin{tikzpicture}[scale=1]
        \filldraw[fill=gray!10] (0,0) circle(1cm);
        \fill (0,0) circle(2pt);
        \fill (0.86,0.5) circle(2pt);
        \fill (0.5,0.86) circle(2pt);
        \fill (-0.5,0.86) circle(2pt);
        \fill (-0.86,0.5) circle(2pt);
        \fill (-0.86,-0.5) circle(2pt);
        \fill (-0.5,-0.86) circle(2pt);
        \fill (0.5,-0.86) circle(2pt);
        \fill (0,1) circle(2pt);
        \fill (0,-1) circle(2pt);
        \fill  (1,0) circle(2pt) (0.86,-0.5) circle(2pt) (-1,0) circle(2pt);
        \draw (0,1.2) node{\tiny$i$};
        \draw (0,-1.2) node{\tiny$i'$};
        \draw (-1.5,0) node{$X$};
        \draw (1.5,0) node{$Y$};
        \fill[black!50,opacity=.5] (-0.86,0.5) circle(5pt);
        \fill[black!50,opacity=.5] (-0.86,-0.5) circle(5pt);
        \fill[black!50,opacity=.5] (0.86,0.5) circle(5pt);
        \fill[black!50,opacity=.5] (1,0) circle(5pt);
        \fill[black!50,opacity=.5] (0.5,-0.86) circle(5pt);
        \draw (0,-0.2) node{\tiny$k$};
        \draw (0,0) edge (0.86,0.5) edge (-0.86,0.5) edge (-0.86,-0.5) edge (0.5,-0.86) edge (1,0);
      \end{tikzpicture}
      \end{subfigure}
      \caption{Complexes $\sK_{I_0}$, $\sK_{\rm in}$ and $\sK_{\rm out}$, and sets $\widetilde X$, $\widetilde Y$, $X$ and $Y$}\label{fig:I0}
\end{figure}
Since $\sK_{I_0}$ is a cycle, it divides $\sK$ into two polygons (triangulated discs) $\sK_{\rm in}$ and $\sK_{\rm out}$ with the common boundary $\sK_{I_0}$. Assume that the vertex $k$ is contained in~$\sK_{\rm in}$. The vertices $i$ and $i'$ divide the cycle $\sK_{I_0}$ into two arcs, and we denote by $\widetilde X$ and $\widetilde Y$ the sets of vertices in $I_0\setminus\{i,i'\}$ contained in these arcs, so $I_0\setminus\{i,i'\}=\widetilde X \sqcup\widetilde Y$. We set $X:=\lk_\sK\{k\}\cap \widetilde X$ and $Y:=\lk_\sK\{k\}\cap\widetilde Y$, see Figure~\ref{fig:I0}. If either $X$ or $Y$ is empty, then $\widetilde{H}^0(\sK_{(I_0\setminus\{i,i'\})\cup\{k\}})\ne 0$, so $I:=I_0$ is the required subset. In what follows we assume that both $X$ and $Y$ are nonempty.

We consider the links of all $x\in X$ in $\sK_{\rm out}$. Since $\sK_{I_0}$ is a chordless cycle, every such link has at least three vertices, that is, there is a vertex in $\lk_{\sK_{\rm out}}\{x\}$ which is not in $I_0$.
To simplify notation, for $X\subset[m]$, we write $\lk_\sK X$ instead of $\bigcup_{x\in X}\lk_\sK \{x\}$. Now define
\[
  \sK_X:=\text{ the full subcomplex of }\sK \text{ induced on the set }
  \widetilde X \cup\{i,i'\}\cup \lk_{\sK_{\rm out}}X.
\]
We take the outermost path $\mathcal P_X$ between $i$ and $i'$ in $\sK_X$ with respect to the vertex~$k$, so that all vertices of $\sK_X$ not in $\mathcal P_X$ are on the side of~$k$, see Figure~\ref{fig:outermost}. Let $I_X$ be the vertex set of $\mathcal P_X$.
\begin{figure}[h]
 \begin{subfigure}[c]{0.3\textwidth}
 \centering
   \begin{tikzpicture}[scale=1]
        \fill[gray!10] (0,0) circle(1cm);
        \path (0,1) coordinate (v) (-0.34,0.93) coordinate (u1) (-0.64,0.76) coordinate(u2) (-0.86,0.5)
        coordinate (u3) (-0.98,0.17) coordinate (u4) (-0.98,-0.17) coordinate (u5) (-0.86,-0.5)
        coordinate (u6) (-0.64,-0.76) coordinate(u7) (-0.34,-0.93) coordinate (u8) (0,-1)
        coordinate (vp) (0,0) coordinate (w) (-0.8,1.2)
        coordinate (a) (-1.5,0.8) coordinate (b) (-1.3,0) coordinate (d) (-1.5,-0.9) coordinate (c) (-0.17,0.98) coordinate (u0);
        \path (0.86,0.5) coordinate (up2) (0.5,0.86) coordinate (up1) (0.5,-0.86) coordinate (up6) (0.86,-0.5) coordinate (up5) (0.98,-0.17) coordinate (up4) (0.98,0.17) coordinate (up3) (0.5,0) coordinate (e) (0.5,-0.5) coordinate (f) (0.75,-0.25) coordinate (g);
        \fill (u1) circle (2pt) (u2) circle (2pt) (u3) circle (2pt) (u4) circle (2pt) (u5) circle (2pt) (u6) circle (2pt) (u7) circle (2pt) (a) circle (2pt) (b) circle (2pt) (c) circle (2pt) (d) circle (2pt) (up6) circle (2pt) (up5) circle (2pt) (up4) circle (2pt) (up3) circle (2pt) (up2) circle (2pt) (up1) circle (2pt) (u0) circle(2pt);
        \fill (u8) circle (2pt) (v) circle (2pt) (vp) circle (2pt);
        \fill (0,0) circle(2pt);
        \draw (0,0) circle (1cm);
        \draw (0,1.2) node{\tiny$i$};
        \draw (0,-1.2) node{\tiny$i'$};
        \draw (0,-0.2) node{\tiny$k$};
        \path (0,0) edge (u2) edge (u7) edge (up2) edge (up6);
        \path (a) edge (u0) edge (u1) edge (u2) edge (b);
        \path (b) edge (u2) edge (u3) edge (d) edge (c);
        \path (c) edge (u8) edge (u7) edge (u6) edge (d);
        \path (d) edge (u3) edge (u4) edge (u5) edge (u6);
        \draw[line width=3pt,black!50,opacity=.5] plot coordinates {(u1) (a) (b) (u3)};
        \draw[line width=3pt,black!50,opacity=.5,join=round] plot coordinates {(u6) (c) (u8)};
  \end{tikzpicture}
  \caption*{$\lk_{\sK_{\rm out}}X$}
 \end{subfigure}
 \begin{subfigure}[c]{0.3\textwidth}
 \centering
  \begin{tikzpicture}[scale=1]
        \fill[gray!10] (0,0) circle(1cm);
        \path (0,1) coordinate (v) (-0.34,0.93) coordinate (u1) (-0.64,0.76) coordinate(u2) (-0.86,0.5)
        coordinate (u3) (-0.98,0.17) coordinate (u4) (-0.98,-0.17) coordinate (u5) (-0.86,-0.5)
        coordinate (u6) (-0.64,-0.76) coordinate(u7) (-0.34,-0.93) coordinate (u8) (0,-1)
        coordinate (vp) (0.5,-0.86) coordinate (up5) (0.86, -0.5) coordinate (up4) (0,1)
        coordinate (up3) (0.86,0.5) coordinate (up2) (0.5,0.86) coordinate (up1) (-0.8,1.2)
        coordinate (a) (-1.5,0.8) coordinate (b) (-1.3,0) coordinate (d) (-1.5,-0.9) coordinate (c) (-0.17,0.98) coordinate (u0);
        \fill (u1) circle (2pt) (u3) circle (2pt) (u4) circle (2pt) (a) circle (2pt) (b) circle (2pt) (c) circle (2pt);
        \fill (u5) circle (2pt) (u6) circle (2pt) (u8) circle (2pt) (u0) circle (2pt)  (u2) circle (2pt) (u7) circle (2pt);
        \fill (0,1) circle(2pt);
        \fill (0,-1) circle(2pt);
        \draw (0,1) arc (90:270:1cm);
        \path (u2) edge (a) edge (b);
        \path (u7) edge (c);
        \draw (0,1.2) node{\tiny$i$};
        \draw (0,-1.2) node{\tiny$i'$};
        \path (a) edge (u0) edge (u1) edge (b);
        \path (b) edge (u3) edge (c);
        \path (c) edge (u8) edge (u6);
  \end{tikzpicture}
  \caption*{$\sK_X$}
 \end{subfigure}
 \begin{subfigure}[c]{0.3\textwidth}
 \centering
  \begin{tikzpicture}[scale=1]
        \fill[gray!10] (0,0) circle(1cm);
        \path (0,1) coordinate (v) (-0.34,0.93) coordinate (u1) (-0.64,0.76) coordinate(u2) (-0.86,0.5)
        coordinate (u3) (-0.98,0.17) coordinate (u4) (-0.98,-0.17) coordinate (u5) (-0.86,-0.5)
        coordinate (u6) (-0.64,-0.76) coordinate(u7) (-0.34,-0.93) coordinate (u8) (0,-1)
        coordinate (vp) (-0.8,1.2)
        coordinate (a) (-1.5,0.8) coordinate (b) (-1.3,0) coordinate (d) (-1.5,-0.9) coordinate (c) (-0.17,0.98) coordinate (u0);
        \fill (u1) circle (2pt) (u3) circle (2pt) (u4) circle (2pt) (a) circle (2pt) (b) circle (2pt) (c) circle (2pt);
        \fill (u5) circle (2pt) (u6) circle (2pt) (u8) circle (2pt);
        \fill (0,-1) circle(2pt) (v) circle (2pt) (0,0) circle (2pt) (u0) circle (2pt);
        \draw[line width=3pt,black!50,opacity=.5] (u8) arc (-110:-90:1cm);
        \fill (u2) circle (2pt) (u7) circle (2pt);
        \draw (0,1) arc (90:270:1cm);
        \path (u2) edge (a) edge (b);
        \path (u7) edge (c);
        \draw (0,1.2) node{\tiny$i$};
        \draw (0,-1.2) node{\tiny$i'$} (0,0)[below] node{\tiny$k$};
        \path (a) edge (u0) edge (u1) edge (b);
        \draw[line width=3pt,black!50,opacity=.5] (v) arc(90:100:1cm);
        \draw (v) arc(90:100:1cm);
        \path[line width=3pt,black!50,opacity=.5] (a) edge (u0) edge (b);
        \path (a) edge (u0) edge (b);
        \path (b) edge (u3) edge (c);
        \path[line width=3pt,black!50,opacity=.5] (b) edge (c);
        \path (b) edge (c);
        \path (c) edge (u8) edge (u6);
        \path[line width=3pt,black!50,opacity=.5] (c) edge (u8);
        \path (c) edge (u8);
        \path (0,0) coordinate (w) (0.86,0.5) coordinate (up2) (0.5,0.86) coordinate (up1) (-0.5,0.86) coordinate (u1) (-0.86,0.5) coordinate (u2) (-1,0) coordinate (u3) (-0.86,-0.5) coordinate (u4) (-0.5,-0.86) coordinate (u5) (0.5,-0.86) coordinate (up6) (0.86,-0.5) coordinate (up5) (0.98,-0.17) coordinate (up4) (0.98,0.17) coordinate (up3) (0,1) coordinate (v) (0,-1) coordinate (vp) (0.5,0) coordinate (a) (0.5,-0.5) coordinate (b) (0.75,-0.25) coordinate (c);
  \end{tikzpicture}
  \caption*{The gray thick path is $\mathcal P_X$.}
 \end{subfigure}
 \caption{Complex $\sK_X$ and path $\mathcal P_X$}\label{fig:outermost}
\end{figure}
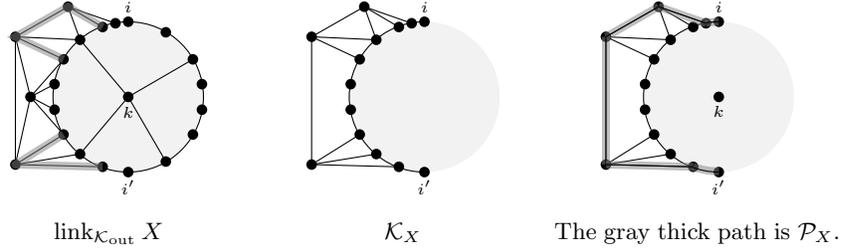

\noindent\textbf{Claim.} The full subcomplex of $\sK$ induced on $I_X$ is the path $\mathcal P_X$, i.e., $\sK_{I_X}=\mathcal P_X$.

\begin{proof}[Proof of Claim]
Suppose to the contrary that there is a subset $\{p,q,r\}$ of $I_1$ such that $\sK_{\{p,q,r\}}$ is a triangle. Consider the intersection $\{p,q,r\}\cap\widetilde X$. Note that $|\{p,q,r\}\cap \widetilde X|<3$ because $\sK_{\widetilde X}$ is a part of a chordless cycle $\sK_{I_0}$.
We have the following cases, shown in Figure~\ref{fig:no-triangle}.
    \begin{figure}[h]
        \begin{subfigure}[b]{.18\textwidth}
        \centering
        \begin{tikzpicture}[scale=1]
            \filldraw[fill=gray!10] (0,0) circle (1cm);
            \path (-0.34,0.93) coordinate (uip) (-1,0) coordinate (y) (-0.86,0.5) coordinate (x) (-0.5,0.86) coordinate
            (u1) (-1.2,1.5) coordinate (z) (-0.9,1) coordinate (c);
            \draw (x) node[below]{\tiny$p$} (y) node[below]{\tiny$q$} (z) node[above]{\tiny$r$} (uip) node[below]{\tiny$x$};
            \path (z) edge (y) edge (x) edge (uip);
            \fill (x) circle(2pt) (y) circle(2pt) (z) circle(2pt) (uip) circle(2pt);
            \fill[blue, opacity=.5] (uip) circle (5pt);
        \end{tikzpicture}
        \caption*{(1)}
        \end{subfigure}
        \begin{subfigure}[b]{.18\textwidth}
        \centering
        \begin{tikzpicture}[scale=1]
            \filldraw[fill=gray!10] (0,0) circle (1cm);
            \path (-1,0) coordinate (uip) (-0.86,0.5) coordinate (x) (-0.5,0.86) coordinate (uiq) (-1.4,0.8) coordinate
            (y) (-0.9,1.2) coordinate (z);
            \draw (x) node[below]{\tiny$p$} (y) node[above]{\tiny$q$} (z) node[above]{\tiny$r$} (uip) node[below]{\tiny$x$} (uiq)
            node[below]{\tiny$x'$};
            \draw (x)--(y)--(z)--cycle;
            \fill (x) circle(2pt) (y) circle(2pt) (z) circle(2pt) (uiq) circle(2pt) (uip) circle(2pt);
            \draw (y)--(uip);
            \draw (z)--(uiq);
            \fill[blue, opacity=.5] (uip) circle (5pt);
            \fill[blue, opacity=.5] (uiq) circle (5pt);
        \end{tikzpicture}
        \caption*{(2)-(a)}
        \end{subfigure}
        \begin{subfigure}[b]{.18\textwidth}
        \centering
        \begin{tikzpicture}[scale=1]
            \filldraw[fill=gray!10] (0,0) circle (1cm);
            \path (-0.34,0.93) coordinate (uiq) (-1,0) coordinate (x) (-0.86,0.5) coordinate (uip) (-0.5,0.86) coordinate
            (u1) (-1.2,1.5) coordinate (y) (-0.9,1) coordinate (z);
            \draw (x) node[below]{\tiny$p$} (y) node[above]{\tiny$q$} (z) node[above]{\tiny$r$} (uiq) node[below]{\tiny$x$};
            \draw (x)--(y)--(z)--cycle;
            \fill (x) circle(2pt) (y) circle(2pt) (z) circle(2pt) (uiq) circle(2pt);
            \draw (y)--(uiq);
            \draw (z)--(uiq);
            \fill[blue, opacity=.5] (uiq) circle (5pt);
        \end{tikzpicture}
        \caption*{(2)-(b)}
        \end{subfigure}
        \begin{subfigure}[b]{.18\textwidth}
        \centering
        \begin{tikzpicture}[scale=1]
            \filldraw[fill=gray!10] (0,0) circle (1cm);
            \path (-0.34,0.93) coordinate (uiq) (-1.5,0.5) coordinate (y) (-0.86,0.5) coordinate (uir) (-1,0) coordinate
            (uip) (-0.9,1.5) coordinate (z) (-1.2,0.8) coordinate (x);
            \draw (x) node[below]{\tiny$p$} (y) node[below]{\tiny$q$} (z) node[above]{\tiny$r$} (uip) node[below]{\tiny$x'$} (uiq)
            node[below]{\tiny$x''$} (uir) node[below]{\tiny$x$};
            \fill (x) circle(2pt) (y) circle(2pt) (z) circle(2pt) (uiq) circle (2pt) (uip) circle(2pt) (uir) circle(2pt);
            \draw (x)--(y)--(z)--cycle;
            \path (z) edge (uiq);
            \path (y) edge (uip);
            \path (x) edge (uir);
            \fill[blue, opacity=.5] (uip) circle (5pt);
            \fill[blue, opacity=.5] (uiq) circle (5pt);
            \fill[blue, opacity=.5] (uir) circle (5pt);
        \end{tikzpicture}
        \caption*{(3)-(a)}
        \end{subfigure}
        \begin{subfigure}[b]{.18\textwidth}
        \centering
        \begin{tikzpicture}[scale=1]
            \filldraw[fill=gray!10] (0,0) circle (1cm);
            \path (-0.34,0.93) coordinate (uiq) (-0.8,1.2) coordinate (y) (-0.86,0.5) coordinate (uir) (-1,0) coordinate
            (uip) (-0.9,1.5) coordinate (z) (-1.2,0.8) coordinate (x);
            \draw (x) node[below]{\tiny$p$} (y) node[below]{\tiny$q$} (z) node[above]{\tiny$r$} (uiq) node[below]{\tiny$x'$} (uir) node[below]{\tiny$x$};
            \fill (x) circle(2pt) (y) circle(2pt) (z) circle(2pt) (uiq) circle (2pt) (uip) (uir) circle(2pt);
            \draw (x)--(y)--(z)--cycle;
            \path (z) edge (uiq);
            \path (y) edge (uiq);
            \path (x) edge (uir);
            \fill[blue, opacity=.5] (uir) circle (5pt);
            \fill[blue, opacity=.5] (uiq) circle (5pt);
        \end{tikzpicture}
        \caption*{(3)-(b)}
        \end{subfigure}
        \caption{$\sK_{X}$ has no triangle.}\label{fig:no-triangle}
    \end{figure}
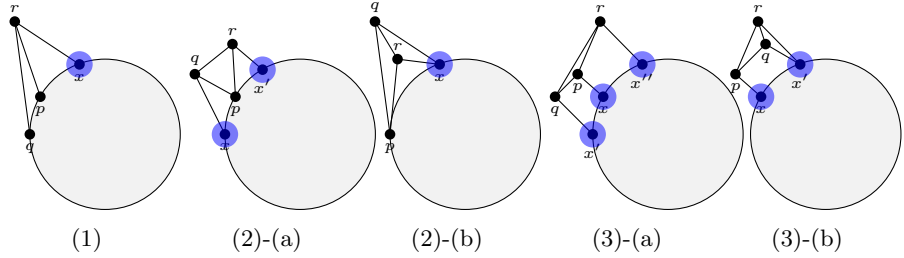
\begin{enumerate}
\item Let $|\{p,q,r\}\cap \widetilde X|=2$, say $\{p,q,r\}\cap \widetilde X=\{p,q\}$. Then $p$ and $q$
  are consecutive vertices in $X$, and $r$ is in $\lk_{\sK_{\rm out}}\{x\}$ for some $x\in X$. Then, $p$ or $q$ is on the side of $k$ in $\sK_{X}$.
  This is a contradiction.

\item Let $|\{p,q,r\}\cap \widetilde X|=1$, say $\{p,q,r\}\cap \widetilde X=\{p\}$. Then $q\in \lk_{\sK_2} x$ and
                $r\in\lk_{\sK_2}\{x'\}$ for some $x,x'\in X$.
    \begin{enumerate}
    \item[(a)] If $x\neq x'$, then $p$ must be on the side of $k$ in $\sK_{X}$, which contradicts
    the assumption that $p\in \mathcal P_X$.
    \item[(b)] If $x=x'$, then either $q$ or $r$ is on the side of $k$ in $\sK_{X}$, and we obtain
    a contradiction again.
    \end{enumerate}

\item Let $|\{p,q,r\}\cap \widetilde X|=0$. Then there are $x,x',x''$ in $X$ such that $p\in \lk_{\sK_{\rm out}}\{x\}$, $q\in \lk_{\sK_{\rm out}}\{x'\}$, and $r\in \lk_{\sK_{\rm out}}\{x''\}$. Since $p,q,r$ are in the outermost path $\mathcal P_X$,  the case $x=x'=x''$ is impossible.  Hence, we may assume that $x\neq x'$ or $x\neq x''$.
\begin{enumerate}
  \item[(a)] If $x,x',x''$ are all distinct, then one of $p$, $q$, and $r$ must be on the side of $k$ in $\sK_{X}$, which contradicts the assumption that $p,q,r$ are on $\mathcal P_X$ and $\mathcal P_X$ is the outermost path with respect to~$k$.
  \item[(b)] If $x'=x''$, then either $q$ or $r$ is on the side of $k$ in $\sK_{X}$.
  This final contradiction finishes the proof of the claim.\qedhere
\end{enumerate}
\end{enumerate}
\end{proof}

We return to the proof of Lemma~\ref{Fabc}.
The endpoints of the path $\mathcal P_X=\sK_{I_X}$ are $i,i'$ and there is no edge connecting $k$ and $I_X$. Therefore, if $\sK_{I_X\cup\widetilde Y}$ is a chordless cycle, then $I_X\cup\widetilde Y$ is the required set~$I$.

Suppose that $\sK_{I_X\cup \widetilde Y}$ has a chord. Then the chord must be an edge in $\sK_{\rm out}$. Note that since $\sK$ has no chordless $4$-cycles, there is no edge connecting $\lk_{\sK_{\rm out}}X$ and~$Y$. We consider the vertices $x_+\in X$ and $x_-\in X$ that are closest to $i$ and $i'$, respectively, on the arc containing~$\widetilde X$. Similarly, consider the vertices $y_+\in Y$ and $y_-\in Y$ that are closest to $i$ and $i'$, respectively, on the arc containing~$\widetilde Y$. Denote by $X_+$ the subset of vertices in $\widetilde X$ lying strictly between $i$ and~$x_+$. Define the subsets $X_-\subset\widetilde X$, $Y_+\subset\widetilde Y$ and $Y_-\subset\widetilde Y$ similarly.
See Figure~\ref{fig:case1}, left.

\begin{figure}[h]
 \begin{subfigure}[c]{0.3\textwidth}
 \centering
   \begin{tikzpicture}[scale=1]
        \fill[gray!10] (0,0) circle(1cm);
        \path (0,1) coordinate (v) (-0.34,0.93) coordinate (u1) (-0.64,0.76) coordinate(u2) (-0.86,0.5)
        coordinate (u3) (-0.98,0.17) coordinate (u4) (-0.98,-0.17) coordinate (u5) (-0.86,-0.5)
        coordinate (u6) (-0.64,-0.76) coordinate(u7) (-0.34,-0.93) coordinate (u8) (0,-1)
        coordinate (vp) (0,0) coordinate (w) (-0.8,1.2)
        coordinate (a) (-1.5,0.8) coordinate (b) (-1.3,0) coordinate (d) (-1.5,-0.9) coordinate (c);
        \path (0.86,0.5) coordinate (up2) (0.5,0.86) coordinate (up1) (0.5,-0.86) coordinate (up6) (0.86,-0.5) coordinate (up5) (0.98,-0.17) coordinate (up4) (0.98,0.17) coordinate (up3) (0.5,0) coordinate (e) (0.5,-0.5) coordinate (f) (0.75,-0.25) coordinate (g) (-0.17,0.98) coordinate (u0);
        \fill (u1) circle (2pt) (u2) circle (2pt) (u3) circle (2pt) (u4) circle (2pt) (u5) circle (2pt) (u6) circle (2pt) (u7) circle (2pt) (a) circle (2pt) (b) circle (2pt) (c) circle (2pt) (d) circle (2pt) (up6) circle (2pt) (up5) circle (2pt) (up4) circle (2pt) (up3) circle (2pt) (up2) circle (2pt) (up1) circle (2pt) (e) circle (2pt) (f) circle (2pt) (g) circle (2pt);
        \fill (u8) circle (2pt) (v) circle (2pt) (vp) circle (2pt) (u0) circle (2pt);
        \fill (0,0) circle(2pt);
        \draw (0,0) circle (1cm);
        \draw (0,1.2) node{\tiny$i$};
        \draw (0,-1.2) node{\tiny$i'$};
        \draw (0,-0.2) node{\tiny$k$};
        \path (0,0) edge (u2) edge (u7) edge (up2) edge (up6);
        \path (a) edge (u0) edge (u1) edge (u2) edge (b);
        \path (b) edge (u2) edge (u3) edge (d) edge (c);
        \path (c) edge (u8) edge (u7) edge (u6) edge (d);
        \path (d) edge (u3) edge (u4) edge (u5) edge (u6);
        \fill (u2) node[left]{\tiny$x_+$} (u7) node[below]{\tiny$x_-$};
        \path (up2) node[right]{\tiny$y_+$} (up6) node[below]{\tiny$y_-$};
        \path (e) edge (w) edge (f) edge (g) edge (up2) edge (up3);
        \path (f) edge (w) edge (up6) edge (up5) edge (g);
        \path (g) edge (up3) edge (up4) edge (up5);
        \draw[line width=3pt,black!50,opacity=.5] plot coordinates {(up2) (e) (g) (up4)};
        \draw[line width=3pt,black!50,opacity=.5] plot coordinates {(up4) (g) (f) (up6)};
        \draw[line width=3pt,black!50,opacity=.5] plot coordinates {(up3) (g) (up5)};
  \end{tikzpicture}
  \caption*{$\lk_{\sK_{\rm in}}(\widetilde Y\setminus(Y\cup Y_-\cup Y_+))$}
 \end{subfigure}
 \begin{subfigure}[c]{0.3\textwidth}
 \centering
  \begin{tikzpicture}[scale=1]
        \fill[gray!10] (0,0) circle(1cm);
        \path (0,1) coordinate (v) (-0.34,0.93) coordinate (u1) (-0.64,0.76) coordinate(u2) (-0.86,0.5)
        coordinate (u3) (-0.98,0.17) coordinate (u4) (-0.98,-0.17) coordinate (u5) (-0.86,-0.5)
        coordinate (u6) (-0.64,-0.76) coordinate(u7) (-0.34,-0.93) coordinate (u8) (0,-1)
        coordinate (vp) (0.5,-0.86) coordinate (up5) (0.86, -0.5) coordinate (up4) (0,1)
        coordinate (up3) (0.86,0.5) coordinate (up2) (0.5,0.86) coordinate (up1) (-0.8,1.2)
        coordinate (a) (-1.5,0.8) coordinate (b) (-1.3,0) coordinate (d) (-1.5,-0.9) coordinate (c) (-0.17,0.98) coordinate (u0);
        \fill (0,1) circle(2pt);
        \fill (0,-1) circle(2pt);
        \draw (0,-1) arc (-90:90:1cm);
        \path (0,0) coordinate (w) (0.86,0.5) coordinate (up2) (0.5,0.86) coordinate (up1) (-0.5,0.86) coordinate (u1) (-0.86,0.5) coordinate (u2) (-1,0) coordinate (u3) (-0.86,-0.5) coordinate (u4) (-0.5,-0.86) coordinate (u5) (0.5,-0.86) coordinate (up6) (0.86,-0.5) coordinate (up5) (0.98,-0.17) coordinate (up4) (0.98,0.17) coordinate (up3) (0,1) coordinate (v) (0,-1) coordinate (vp) (0.5,0) coordinate (a) (0.5,-0.5) coordinate (b) (0.75,-0.25) coordinate (c);
        \fill (up1) circle(2pt) (up2) circle(2pt) (up6) circle (2pt) (a) circle(2pt) (b) circle (2pt) (c) circle (2pt) (up3) circle (2pt) (up4) circle (2pt) (up5) circle (2pt);
        \path (v) node[above]{\tiny$i$} (vp) node[below]{\tiny$i'$} (up2) node[right]{\tiny$y_+$} (up6) node[below]{\tiny$y_-$};
        \path (a) edge (b) edge (c) edge (up2) edge (up3);
        \path (b) edge (c) edge (up6) edge (up5);
        \path (c) edge (up4) edge (up5) edge (up3);
  \end{tikzpicture}
  \caption*{$\sK_Y$}
 \end{subfigure}
 \begin{subfigure}[c]{0.3\textwidth}
 \centering
  \begin{tikzpicture}[scale=1]
        \fill[gray!10] (0,0) circle(1cm);
        \path (0,1) coordinate (v) (-0.34,0.93) coordinate (u1) (-0.64,0.76) coordinate(u2) (-0.86,0.5) coordinate (u3) (-0.98,0.17) coordinate (u4) (-0.98,-0.17) coordinate (u5) (-0.86,-0.5) coordinate (u6) (-0.64,-0.76) coordinate(u7) (-0.34,-0.93) coordinate (u8) (0,-1) coordinate (vp) (-0.8,1.2) coordinate (a) (-1.5,0.8) coordinate (b) (-1.3,0) coordinate (d) (-1.5,-0.9) coordinate (c) (-0.17,0.98) coordinate (u0);
        \fill (a) circle (2pt) (b) circle (2pt) (c) circle (2pt);
        \fill (u8) circle (2pt);
        \fill (0,-1) circle(2pt) (v) circle (2pt) (0,0) circle (2pt) (u0) circle (2pt);
        \draw (u8) arc (-110:-90:1cm);
        \draw (v) arc (90:100:1cm);
        \draw (0,1.2) node{\tiny$i$};
        \draw (0,-1.2) node{\tiny$i'$} (0,0)[below] node{\tiny$k$};
        \path (a) edge (u0) edge (b);
        \path (b) edge (c);
        \path (c) edge (u8);
        \draw (0,-1) arc (-90:90:1cm);
        \draw[line width=3pt,black!50,opacity=.5] (0,-1) arc (-90:-60:1cm);
        \draw[line width=3pt,black!50,opacity=.5] (0.86,0.5) arc (30:90:1cm);
        \path (0,0) coordinate (w) (0.86,0.5) coordinate (up2) (0.5,0.86) coordinate (up1) (-0.5,0.86) coordinate (u1) (-0.86,0.5) coordinate (u2) (-1,0) coordinate (u3) (-0.86,-0.5) coordinate (u4) (-0.5,-0.86) coordinate (u5) (0.5,-0.86) coordinate (up6) (0.86,-0.5) coordinate (up5) (0.98,-0.17) coordinate (up4) (0.98,0.17) coordinate (up3) (0,1) coordinate (v) (0,-1) coordinate (vp) (0.5,0) coordinate (a) (0.5,-0.5) coordinate (b) (0.75,-0.25) coordinate (c);
        \fill (up1) circle(2pt) (up2) circle(2pt) (up6) circle (2pt) (a) circle(2pt) (b) circle (2pt) (c) circle (2pt) (up3) circle(2pt) (up4) circle(2pt) (up5) circle (2pt);
        \path (v) node[above]{\tiny$i$} (vp) node[below]{\tiny$i'$} (up2) node[right]{\tiny$y_+$} (up6) node[below]{\tiny$y_-$};
        \path (a) edge (b) edge (c) edge (up2) edge (up3);
        \path (b) edge (c) edge (up6) edge (up5);
        \path (c) edge (up4) edge (up5) edge (up3);
        \path (0,0) edge (up2) edge (a) edge (b) edge (up6);
        \draw[line width=3pt,black!50,opacity=.5] plot coordinates {(up2) (a) (b) (up6)};
  \end{tikzpicture}
  \caption*{The gray thick path is $\mathcal P_Y$.}
 \end{subfigure}
 \caption{Example of Case 1}\label{fig:case1}
\end{figure}
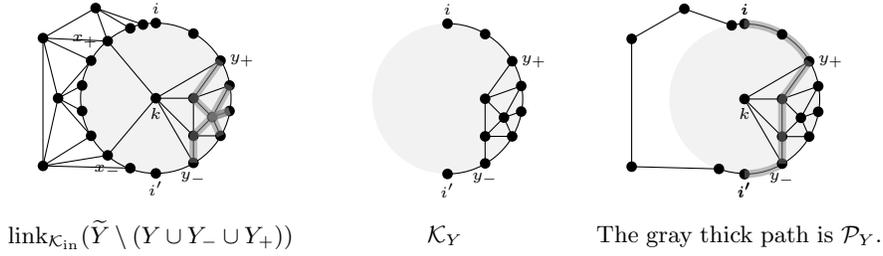

We consider two cases.

\smallskip

\noindent\textbf{Case 1. There is no edge connecting $I_X$ and $Y_-\cup Y_+$ in $\sK_{\rm out}$.}

We define
$$
  \sK_Y:=\text{ the full subcomplex of }\sK\text{ induced on }\widetilde Y\cup\{i,i'\}\cup
  \lk_{\sK_{\rm   in}}(\widetilde Y\setminus (Y\cup Y_-\cup Y_+)).
$$
We take the innermost path $\mathcal P_Y$ connecting $i$ and $i'$ in $\sK_Y$ with respect to $k$, see Figure~\ref{fig:case1}, and let $I_Y$ be the vertex set of $\mathcal P_Y$. Then $\sK_{I_Y}=\mathcal P_Y$ by the same argument as the claim above, and $I_X\cup I_Y$ is the required subset $I$.

\smallskip

\noindent\textbf{Case 2. There is an edge connecting $I_X$ and $Y_+$ or $Y_-$ in $\sK_{\rm out}$.}

Suppose that $I_X$ is connected by an edge in $\sK_{\rm out}$ to only one of $Y_+$ and $Y_-$, say to $Y_+$. We define
$$
  \sK^+_Y:=\text{the full subcomplex of }\sK\text{ induced on }\widetilde Y\cup\{i,i'\}\cup
  \lk_{\sK_{\rm     in}}(\widetilde Y\setminus(Y\cup Y_-)).
$$
We take the innermost path $\mathcal P^+_Y$ connecting $i$ and $i'$ in $\sK_{\rm in}$ with respect to the vertex $k$, and let $I^+_Y$ be the vertex set of $\mathcal P^+_Y$. See Figure~\ref{fig:case-2}, middle. Then $\sK_{I^+_Y}=\mathcal P^+_Y$ by the same reason as the claim above. If $\sK_{I_X\cup I_Y^+}$ is a chordless cycle, then $I_X\cup I_Y^+$ is the required subset $I$.
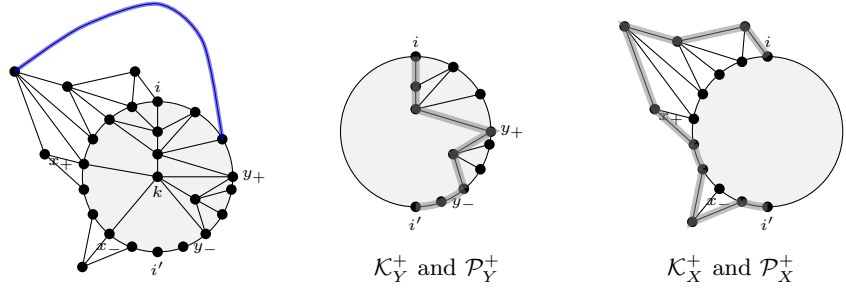
\begin{figure}[h]
    \begin{subfigure}[b]{0.3\textwidth}
    \centering
      \begin{tikzpicture}[scale=1]
        \fill[gray!10] (0,0) circle (1cm);
        \draw (0,0) circle(1cm);
        \path (0,0) coordinate (w) (0,1) coordinate (v) (-0.34,0.93) coordinate (u1) (-0.64,0.76) coordinate(u2)
        (-0.86,0.5) coordinate (u3) (-0.98,0.17) coordinate (u4) (-0.98,-0.17) coordinate (u5) (-0.86,-0.5) coordinate
        (u6) (-0.64,-0.76) coordinate(u7) (-0.34,-0.93) coordinate (u8) (0.34,-0.93) coordinate (up7) (0.64,-0.76) coordinate (up6) (0.86,-0.5) coordinate (up5) (0.98,-0.17) coordinate (up4) (1,0) coordinate (up3) (0.86,0.5) coordinate (up2) (0.5,0.86) coordinate (up1);
        \path (0,0.3) coordinate (a) (0,0.6) coordinate (x) (-1.2,1.2) coordinate (g) (-1.9,1.4) coordinate (h) (-1.5,0.3) coordinate (j) (-0.3,1.4) coordinate (y) (0.5,-0.3) coordinate (b) (-1,-1.2) coordinate (c);
        \fill (u1) circle (2pt) (u2) circle (2pt) (u3) circle (2pt) (u4) circle (2pt) (a) circle (2pt) (b) circle (2pt);
        \fill (u5) circle (2pt) (u6) circle (2pt) (u7) circle (2pt) (u8) circle (2pt) (v) circle (2pt) (vp) circle (2pt)
        (up1) circle (2pt) (up2) circle (2pt) (up3) circle (2pt) (up4) circle (2pt) (up5) circle (2pt) (w) circle (2pt) (up6) circle (2pt) (up7) circle (2pt) (c) circle (2pt);
        \fill (g) circle (2pt) (h) circle (2pt) (j) circle (2pt) (x) circle (2pt) (y) circle (2pt);
        \fill (u4) node[left]{\tiny$x_+$};
        \fill  (v) node[above]{\tiny$i$} (vp) node[below]{\tiny$i'$}  (up3) node[right]{\tiny$y_+$}  (w) node[below]{\tiny$k$} (u7) node[below]{\tiny$x_-$} (up6) node[below]{\tiny$y_-$};
        \path (b) edge (up3) edge (up4) edge (up5) edge (w) edge (up6);
        \path (c) edge (u6) edge (u7) edge (u8);
        \path (w) edge (a) edge (up3) edge (u4) edge (u7) edge (up6);
        \path (a) edge (up3) edge (up2) edge (up1) edge (u2);
        \path (g) edge (u1) edge (u2) edge (u3) edge (h);
        \path (h) edge (u3) edge (u4) edge (j);
        \path (j) edge (u4) edge (u5);
        \path (x) edge (v) edge (u1) edge (u2) edge (up1) edge (a);
        \path (y) edge (v) edge (g) edge (u1);
        \draw plot [smooth] coordinates {(h) (-1,2) (0,2.3) (0.6,1.8) (up2)};
        \draw[blue, ultra thick, opacity=.5] plot [smooth] coordinates {(h) (-1,2) (0,2.3) (0.6,1.8) (up2)};
      \end{tikzpicture}
    \end{subfigure}
    \begin{subfigure}[b]{0.3\textwidth}
    \centering
      \begin{tikzpicture}[scale=1]
        \fill[gray!10] (0,0) circle (1cm);
        \draw (0,0) circle(1cm);
        \path (0,0) coordinate (w) (0,1) coordinate (v) (0.34,-0.93) coordinate (up7) (0.64,-0.76) coordinate (up6) (0.86,-0.5) coordinate (up5) (0.98,-0.17) coordinate (up4) (1,0) coordinate (up3) (0.86,0.5) coordinate (up2) (0.5,0.86) coordinate (up1);
        \path (0,0.3) coordinate (a) (0,0.6) coordinate (x) (0.5,-0.3) coordinate (b);
        \fill (a) circle (2pt) (b) circle (2pt);
        \fill (v) circle (2pt) (vp) circle (2pt)
        (up1) circle (2pt) (up2) circle (2pt) (up3) circle (2pt) (up4) circle (2pt) (up5) circle (2pt) (up6) circle (2pt) (up7) circle (2pt);
        \fill (x) circle (2pt);
        \fill  (v) node[above]{\tiny$i$} (vp) node[below]{\tiny$i'$}  (up3) node[right]{\tiny$y_+$}  (up6) node[below]{\tiny$y_-$};
        \path (b) edge (up3) edge (up4) edge (up5) edge (up6);
        \path (a) edge (up3) edge (up2) edge (up1);
        \path (x) edge (v) edge (up1) edge (a);
        \draw[black!50,line width=3pt,opacity=.5] plot coordinates {(v) (x) (a) (up3) (b) (up6)};
        \draw[black!50,line width=3pt,opacity=.5] (vp) arc (-90:-50:1cm);
      \end{tikzpicture}
      \caption*{$\sK^+_Y$ and $\mathcal P^+_Y$}
    \end{subfigure}
    \begin{subfigure}[b]{0.3\textwidth}
    \centering
      \begin{tikzpicture}[scale=1]
        \fill[gray!10] (0,0) circle (1cm);
        \draw (0,0) circle(1cm);
        \path (0,0) coordinate (w) (0,1) coordinate (v) (-0.34,0.93) coordinate (u1) (-0.64,0.76) coordinate(u2)
        (-0.86,0.5) coordinate (u3) (-0.98,0.17) coordinate (u4) (-0.98,-0.17) coordinate (u5) (-0.86,-0.5) coordinate
        (u6) (-0.64,-0.76) coordinate(u7) (-0.34,-0.93) coordinate (u8);
        \path (-1.2,1.2) coordinate (g) (-1.9,1.4) coordinate (h) (-1.5,0.3) coordinate (j) (-0.3,1.4) coordinate (y) (-1,-1.2) coordinate (c);
        \fill (u1) circle (2pt) (u2) circle (2pt) (u3) circle (2pt) (u4) circle (2pt);
        \fill (u5) circle (2pt) (u6) circle (2pt) (u7) circle (2pt) (u8) circle (2pt) (v) circle (2pt) (vp) circle (2pt) (c) circle (2pt);
        \fill (g) circle (2pt) (h) circle (2pt) (j) circle (2pt) (y) circle (2pt);
        \fill (u4) node[left]{\tiny$x_+$};
        \fill  (v) node[above]{\tiny$i$} (vp) node[below]{\tiny$i'$} (u7) node[below]{\tiny$x_-$};
        \path (c) edge (u6) edge (u7) edge (u8);
        \path (g) edge (u1) edge (u2) edge (u3) edge (h);
        \path (h) edge (u3) edge (u4) edge (j);
        \path (j) edge (u4) edge (u5);
        \path (y) edge (v) edge (g) edge (u1);
        \draw[black!50,line width=3pt,opacity=.5] plot coordinates {(v) (y) (g) (h) (j) (u5)};
        \draw[black!50,line width=3pt,opacity=.5] plot coordinates {(u6) (c) (u8)};
        \draw[black!50,line width=3pt,opacity=.5] (u5) arc (190:210:1cm);
        \draw[black!50,line width=3pt,opacity=.5] (vp) arc (-90:-110:1cm);
      \end{tikzpicture}
      \caption*{$\sK^+_X$ and $\mathcal P^+_X$}
    \end{subfigure}
      \caption{Example of Case 2}\label{fig:case-2}
\end{figure}

If $\sK_{I_X\cup I^+_Y}$ has a chord, then it must be an edge in $\sK_{\rm in}$ connecting $\lk_{\sK_{\rm in}} Y_+$ and $X_+\cap I_X$. In this case we modify $I_X$ as follows. We define
$$
  \sK^+_X:=\text{ the full subcomplex of }\sK\text{ induced on }\widetilde X\cup\{i,i'\}
  \cup \lk_{\sK_{\rm out}}(X\cup X_+).
$$
We take the outermost path $\mathcal P^+_X$ connecting $i$ and $i'$ in $\sK^+_X$ with respect to the vertex~$k$, see Figure~\ref{fig:case-2}, right. Let $I^+_X$ be the vertex set of $\mathcal P^+_X$. Then $\sK_{I^+_X} =\mathcal P^+_X$ by the same argument as the claim above, and we can see that $I^+_X\cup I_Y$ is the required subset~$I$. Indeed, we only need to check that there is no edge connecting $\lk_{\sK_{\rm out}}X_+$ and $Y$ in $\sK_{\rm out}$. This is because there is an edge connecting $I_X$ and $Y_+$.

\smallskip

It remains to consider the case when $I_X$ is connected to both $Y_+$ and $Y_-$ by edges in $\sK_{\rm out}$. Here the same argument as above works if we consider
\begin{gather*}
  \sK^\pm_Y:=\text{the full subcomplex of }\sK\text{ induced on }\widetilde Y
  \cup     \{i,i'\}\cup\lk_{\sK_{\rm in}}(\widetilde Y\setminus Y),\\
  \sK^\pm_X:=\text{the full subcomplex of }\sK\text{ induced on }\widetilde X
  \cup \{i,i'\}\cup\lk_{\sK_{\rm out}}(X\cup X_+\cup X_-)
\end{gather*}
instead of $\sK^+_Y$ and $\sK^+_X$, respectively.
\end{proof}

\section{Combinatorics and constructions of Pogorelov polytopes}
We recall (see Subsection~\ref{sechyp}) that a Pogorelov polytope is a simple $3$-polytope $P\ne\varDelta^3$ without $3$-belts (that is, flag) and without $4$-belts. The class of Pogorelov polytopes is denoted by~$\mathcal{P}$.

We shall use the following reformulation of the Steinitz Theorem:

\begin{theorem}[see~\cite{bu-erS}]\label{S2theorem}
A simple graph on a $2$-dimensional sphere is the graph of a
convex $3$-polytope if and only if the following two conditions
are satisfied:
\begin{itemize}
\item[(a)]
each connected component of the complement to the graph in the
sphere is bounded by a simple edge cycle, and

\item[(b)]
the intersection of the closures of any two different connected
components of the complement is either a single edge, a single
vertex, or empty.
\end{itemize}
\end{theorem}

The following proposition gives a characterisation of flag $3$-polytopes and Pogorelov polytopes in terms of $k$-belts.

\begin{proposition}\label{34bp}\hphantom{sdf}\
\begin{enumerate}
\item[(a)] A simple $3$-polytope $P$ is flag if and only if each of its facets is surrounded by a
$k$-belt, where $k$ is the number of edges in the facet. Furthermore, for a flag polytope we have $k\geqslant 4$.

\item[(b)] A simple $3$-polytope $P$ is a Pogorelov polytope if and only if each pair of its adjacent facets is surrounded by a $k$-belt; if the facets have $k_1$ and $k_2$ edges, then $k=k_1+k_2-4$. Furthermore, $k_1,k_2\geqslant 5$, hence, $k\geqslant 6$.
\end{enumerate}
\end{proposition}
\begin{proof}
(a) Assume that $P$ is flag. Let $\mathcal{B}=(F_{i_1},\ldots,F_{i_k})$ be the sequence of facets
adjacent to a facet~$F$, written in a cyclic order. If $k=3$ and $F_{i_1}\cap F_{i_2}\cap F_{i_3}$ is a vertex, then $P\simeq \varDelta^3$. A contradiction. Let $k\geqslant 4$. If two different facets $F_{i_p}$ and $F_{i_q}$ with $|p-q|\ne 1\mod k$ have nonempty intersection, then $(F,F_{i_p},F_{i_q})$ is a
$3$-belt. A contradiction. Therefore, in either case the sequence $\mathcal{B}$ is a $k$-belt. Since a flag polytope does not have $3$-belts, we have $k\geqslant 4$ for any of its facets.

A simplex $\varDelta^3$ is not flag, and none of its facets is surrounded by a belt. If $P\ne\varDelta^3$ is not a flag polytope, then it has a $3$-belt $(F, F_i,F_j)$. Then the facets $F_i$ and $F_j$ have nonempty intersection, are adjacent to $F$, and are not consecutive in the sequence of facets around~$F$. Therefore, the facet $F$ is not surrounded by a belt.

\smallskip

(b) Assume that $P$ is a Pogorelov polytope. A pair of adjacent
facets $(F_i,F_j)$ is bounded by a simple edge cycle. Let
$\mathcal{L}=(F_{i_1}, \dots, F_{i_k})$ be the sequence of facets
around $F_i\cup F_j$, written in a cyclic order. If
$F_{i_a}=F_{i_b}$ for some $a\ne b$, then $(F_i,F_j,F_{i_a})$ is a
$3$-belt. A contradiction. If $\mathcal{L}$  is not a $k$-belt,
then $F_{i_a}\cap F_{i_b}\ne\varnothing$ for some $a,b$,
$|a-b|\ne0,1\mod k$. Since $P$ is flag, statement~(a) implies that
neither of $F_i$ and $F_j$ can be adjacent to both  $F_{i_a}$ and
$F_{i_b}$. Let $F_{i_a}$ be  adjacent to $F_i$, and $F_{i_b}$
adjacent to~$F_j$. Then $(F_{i_a}, F_i,F_j,F_{i_b})$ is a
$4$-belt. A contradiction. Therefore, $\mathcal{L}$ is a $k$-belt.
A simple calculation shows that $k=k_1+k_2-4$.

Assume now that each pair of adjacent facets in a simple polytope $P$ is surrounded by a belt. Then $P\not\simeq\varDelta^3$. If $(F_i,F_j,F_k)$ is a $3$-belt, then the facet $F_k$ appears twice in the cyclic sequence of facets around the pair of adjacent facets $F_i,F_j$. A contradiction. If $(F_i,F_j,F_k,F_l)$ is a $4$-belt, then the facets $F_k$ and $F_l$ belong to the cyclic sequence of facets around the pair of adjacent facets $F_i,F_j$. Since $F_i\cap F_k=\varnothing=F_j\cap F_l$, the facets $F_k$ and $F_l$ are not consecutive in this cyclic sequence. Therefore, the cyclic sequence is not a
$k$-belt. A contradiction. Thus, $P$ is a Pogorelov polytope.
\end{proof}

To each belt $\mathcal{B}$ on a simple $3$-polytope $P$ we assign a simple closed broken line $\gamma(\mathcal{B})$ in the following way: each segment of $\gamma(\mathcal{B})$ joins the midpoints of the edges obtained as the intersection of a facet from the belt with the preceding and subsequent facets. Theorem~\ref{S2theorem} implies the following result.

\begin{proposition}
Let $P$ and $Q$ be simple $3$-polytopes with chosen $k$-gonal facets $F\subset P$ and $G\subset Q$. Assume that each of $F$ and $G$ is surrounded by a $k$-belt. Then there exists a simple
$3$-polytope $R$ with a $k$-belt $\mathcal{B}$ such that the surfaces of the polytopes $P$ and $Q$ are  obtained by cutting the surface of $R$ along the broken line $\gamma(\mathcal{B})$ and gluing a pair of $k$-gons along this line. Furthermore, every polytope $R$ with a $k$-belt $\mathcal{B}$ is obtained from some polytopes $P$ and $Q$ by reversing this procedure.
\end{proposition}

We refer to the polytope $R$ as the \emph{connected sum of simple polytopes $P$ and $Q$ at the facets $F$ and~$G$}. The result depends on the ordering of facets around $F$ and~$G$. Truncating a simple $3$-polytope at a vertex gives a new triangular facet surrounded by a $3$-belts. In this way the \emph{vertex connected sum of two simple polytopes} is defined, see
\cite[Construction~1.1.13]{bu-pa15}). Truncating a simple $3$-polytope at an edge gives a quadrilateral facet, which is surrounded by a $4$-belt whenever the two facets having a common vertex with the cut edge are not adjacent. If the chosen edges satisfy this property, the \emph{edge connected sum of two simple polytopes} is defined. For flag polytopes, the edge connected sum is defined at any edges.

\begin{example}\label{csexa}\hphantom{sdf}\

\noindent 1. A vertex connected sum of two dodecahedra gives a simple polytope with $18$ pentagonal and $3$ octagonal facets. This $3$-polytope is not flag, as it has a $3$-belt.

\smallskip

\noindent 2. An edge connected sum of two dodecahedra gives a simple polytope with $16$ pentagonal facets and either $4$ heptagonal facets, or $2$ hexagonal and $2$ octagonal facets, depending on the ordering of quadruples of facets around the chosen edges.  This $3$-polytope has a $4$-belt.

\smallskip

\noindent 3. A connected sum of a dodecahedron with two other dodecahedra, one at a pair of vertices and the other at a pair of edges, gives a simple $3$-polytope without triangular and quadrangular facets, but having both $3$- and $4$-belts.

\smallskip

These examples show that the absence of triangular and quadrangular facets does not guarantee that a $3$-polytope belongs to the Pogorelov class~$\mathcal P$.
\end{example}

The above operations of vertex and edge connected sum are used in the following structural result on simple $3$-polytopes.

\begin{theorem}\
\begin{enumerate}
\item[(a)] A simple $3$-polytope has a $3$-belt if and only if it can be decomposed into a connected sum of two simple polytopes at vertices.

\item[(b)] Any simple $3$-polytope is a vertex connected sum of simplices and flag polytopes.

\item[(c)] A flag $3$-polytope has a $4$-belt if and only if it either has a quadrangular facet, or is an edge connected sum of two flag polytopes.

\item[(d)] A $3$-polytope $P$ is flag if and only if it can be obtained from a set of Pogorelov polytopes and cubes by the operations of edge connected sum and edge truncation.
\end{enumerate}
\end{theorem}
\begin{proof}
As we have seen above, a simple $3$-polytope $P$ can be cut along a $k$-belt, therefore decomposing it into a connected sum of two polytopes along $k$-gonal facets. By~\cite[Lemma~2.11]{bu-erS}, if $P$ is a flag polytope, then the two resulting polytopes are also flag. Theorem~\ref{S2theorem} implies that for any triangular facet $F$ of a polytope $P\not\simeq\varDelta^3$ there exists a polytope $Q$ such that  $P$ is combinatorially equivalent to a vertex truncation of $Q$, with the new facet corresponding to~$F$. This is equivalent to taking a vertex connected sum of $Q$ with a simplex. Thus, statements (a) and~(b) are proved. It is easy to see that an edge truncation or an edge connected sum of flag polytopes is a flag polytope (see~\cite{bu-er15,bu-erS}). By~\cite[Lemma~2.17]{bu-er15}, for any quadrangular facet $F$ of a flag $3$-polytope $P\not\simeq I^3$, there exists a flag polytope $Q$ such that $P$ is combinatorially equivalent to an edge truncation of~$Q$, with the new facet corresponding to~$F$. This proves (c) and~(d).
\end{proof}

\begin{proposition}
Given $P\in\mathcal{P}$, let $Q$ be the polytope obtained from $P$ by cutting off a sequence of $s\geqslant 2$ adjacent edges lying on a $k$-gonal facet~$F$.  Assume that $k\geqslant s+4$. Then $Q\in\mathcal P$ (see Figure~\ref{sk-trunk}).
\end{proposition}
\begin{figure}[h]
\begin{center}
\includegraphics[width=.9\textwidth] {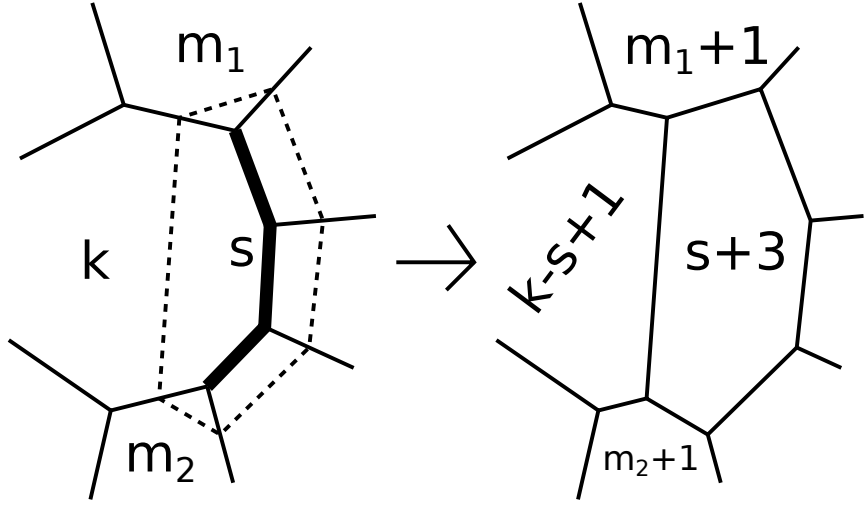}
\end{center}
\caption{A $(s,k)$-truncation}\label{sk-trunk}
\end{figure}
\begin{proof}
Let $G$ be the new facet of $Q$ obtained as the result of truncation. For a facet $F_i$ of $P$, we denote by $\widehat F_i$ the corresponding facet of~$Q$. If $F_{i_1}\cap\dots\cap F_{i_k}=\varnothing$ in $P$, then $\widehat{F}_{i_1}\cap\dots\cap \widehat{F}_{i_k}=\varnothing$ in~$Q$.
If $F_i$ and $F_j$ are adjacent facets different from~$F$, then the corresponding facets $\widehat{F}_i$ and $\widehat{F}_j$ are also adjacent.

If $(\widehat{F}_i,\widehat{F}_j,\widehat{F}_k)$ is a $3$-belt, then $F_i\cap F_j\cap F_k$ is a cut vertex. This vertex is cut together with an incident edge, so two of the facets $\widehat{F}_i$, $\widehat{F}_j$, $\widehat{F}_k$ do not intersect in~$Q$. A contradiction. Therefore, if $Q$ has a $3$-belt, then it has the form $(\widehat{F}_i,\widehat{F}_j,G)$. The facets $F_i$ and $F_j$ are different from~$F$, as otherwise one of these facets has two common edges with~$F$. Furthermore,  $F_i\cap F_j\ne\varnothing$, $F\cap F_i\ne\varnothing$ and $F\cap F_j\ne\varnothing$, because both $F_i$ and $F_j$ intersect with the cut edges. Since $s<k-2$, the edge $F_i\cap F_j$ intersects the set of cut edges, but is not contained in it, so $\widehat{F}_i\cap \widehat{F}_j\cap G\ne\varnothing$. A contradiction.

If $(\widehat{F}_i,\widehat{F}_j,\widehat{F}_k,\widehat{F}_l)$ is a $4$-belt, then $(F_i,F_j,F_k,F_l)$ are the facets around a cut edge $F_i\cap F_k$ or $F_j\cap F_l$. Since $s>1$, one more adjacent edge $F_i\cap F_j$, $F_j\cap F_k$, $F_k\cap F_l$ or $F_l\cap F_i$ is also cut, so the corresponding facets of $Q$ do not intersect. A contradiction. Therefore, if $Q$ has a $4$-belt, then it has the form $(\widehat{F}_i,\widehat{F}_j,\widehat{F}_k,G)$, where $F_i\cap F_j\ne\varnothing$, $F_j\cap F_k\ne\varnothing$, the facets $F_i$, $F_k$ intersect the cut edges of~$P$, and the facet $F_j$ does not intersect the cut edges. Then $F_i\ne F$, as otherwise $(F,F_j,F_k)$ is a $3$-belt of~$P$, because $\widehat{F}\cap \widehat{F}_k=\varnothing$. Similarly, $F_k\ne F$. We also have $F_j\ne F$, because $\widehat{F}\cap G\ne\varnothing$.  In the cyclic sequence $(F,F_i,F_j,F_k)$ consecutive facets have nonempty intersections, so $F\cap F_j\ne\varnothing$ or $F_i\cap F_k\ne\varnothing$, because $P$ does not have $4$-belts. We have $F_i,F_j\ne F$, so  $F_i\cap F_k\ne\varnothing$ implies $\widehat{F}_i\cap\widehat{F}_k\ne\varnothing$. A contradiction. Therefore, $F\cap F_j\ne\varnothing$ and the facets $F_i$ and $F_k$ intersect the edge $F\cap F_j$ at a vertex. Since $s<k-3$, the edge $F\cap F_j$ is being cut. Hence, $G\cap \widehat{F}_j\ne\varnothing$. A contradiction.

Thus, the polytope $Q\not\simeq\varDelta^3$ does not have $3$- and $4$-belts, so it is a Pogorelov polytope.
\end{proof}

\begin{proposition}\label{consumpog}
Let $P,Q\in\mathcal{P}$, and let $F\subset P$, $G\subset Q$ be $k$-gonal facets. Then the connected sum of the polytopes $P$ and $Q$ at the facets $F$ and $G$ is defined, and belongs to the class~$\mathcal{P}$.
\end{proposition}
\begin{proof} Since both $P$ and $Q$ are flag polytopes,
Proposition~\ref{34bp}~(a) implies that the facets $F$ and $G$ are
surrounded by $k$-belts. Therefore, the connected sum at these
facets is defined and gives a simple polytope $R$ with a $k$-belt
$\mathcal{B}$. The combinatorial type of $R$ depends on the order
of facets around $F$ and~$G$. Choose  a pair of adjacent facets
of~$R$. If none of these facets belongs to the belt $\mathcal{B}$,
then we may assume without loss of generality that both chosen
facets belong to $P\setminus\mathcal{B}$. The intersections of
facets around the pair of chosen facets in~$R$ are the same as the
corresponding intersections in~$P$, whence the pair of chosen
facets is surrounded by a belt. If both chosen adjacent facets
belong to the belt, then the facets around them form a cyclic
sequence consisting of two facets from the belt~$\mathcal{B}$ and
two sequences of facets lying in $P\setminus\mathcal{B}$ and
$Q\setminus\mathcal{B}$, respectively. The facets in each sequence
together with the two facets from the belt surround the pairs of
facets in $P$ and~$Q$ corresponding to the chosen adjacent facets,
and the facets from the different sequences do not intersect, so
the whole cyclic sequence is a belt. Finally if one of the chosen
adjacent facets belongs to the belt $\mathcal{B}$, and the other
does not belong to the belt, then we may assume without loss of
generality that the other facet belongs to
$P\setminus\mathcal{B}$. Then the facets around the pair of chosen
facets form a cyclic sequence consisting of two facets from the
belt $\mathcal{B}$ and two sequences of facets lying in
$P\setminus\mathcal{B}$ and $Q\setminus\mathcal{B}$, respectively.
The facets in the first sequence together with the two facets from
the belt surround the pair of facets in $P$ corresponding to the
chosen pair, and the facets in the second sequence together with
the two facets from the belt surround the facet of $P$
corresponding to the facet in the belt. The facets from the
different sequences do not intersect, so the whole cyclic sequence
is a belt. Thus, every pair of adjacent facets in $R$ is
surrounded by a belt, so $R\in\mathcal{P}$ by
Proposition~\ref{34bp}~(b).
\end{proof}

\begin{remark}
Proposition~\ref{consumpog} has a geometric interpretation. By
Theorem~\ref{Pogth} each of the polytopes $P$ and $Q$ has a unique
right-angled realisation in~$\mathbb L^3$. If the corresponding
facets $F$ and $G$ are congruent (for example, if $P\simeq Q$),
then gluing $P$ and $Q$ along $F$ and $G$ gives a right-angled
polytope~$R$. Otherwise the connected sum is a non-local operation
on right-angled polytopes, that is, the shape of $P$ and $Q$
changes globally after realising their connected sum $R$ with
right dihedral angles.
\end{remark}

The following result was obtained by Inoue in~\cite{inou08} (see
also the survey paper~\cite{vesn17}):

\begin{theorem}[\cite{inou08}]
A simple $3$-polytope $P$ belongs to the Pogorelov class $\mathcal
P$ if and only if it can be obtained from a collection of barrels
$Q_r$, $r\ge5$, (see Example~\ref{barrelex}) by a sequence of
connected sums along $p$-gonal facets with $p\ge5$ and
$(s,k)$-truncations, where $k\ge6$ and $2\le s\le k-4$.
\end{theorem}

Inoue's theorem was strengthened in~\cite{bu-er17}. We denote by
$\mathcal P_B$ the set of $r$-barrels $Q_r$ with $r\ge 5$ and
consider the class $\mathcal P_B^\perp=\mathcal P\setminus\mathcal
P_B$.

\begin{theorem}[\cite{bu-er17}]\hphantom{dsf}\
\begin{itemize}
\item[(a)]
An $r$-barrel $Q_r$ with $r\ge 5$ cannot be obtained from another
Pogorelov polytope by $(2,k)$-truncations and connected sums with
a dodecahedron $Q_5$ along a pentagonal facet.

\item[(b)]
A polytope $P$ belongs to the class $\mathcal P_B^\perp$ if and
only if it can be obtained from $Q_5$ or $Q_6$ by a nonempty
sequence of connected sums with $Q_5$ along a pentagonal facet and
$(2,k)$-truncations with $k\ge6$.
\end{itemize}
\end{theorem}

Denote by $p_k$ the number of $k$-gonal facets in a polytope~$P$. The Euler formula implies the following identity for a simple $3$-polytope:
\begin{equation}\label{pkformula}
3p_3+2p_4+p_5=12+\sum\limits_{k\geqslant 7}(k-6)p_k.
\end{equation}

The following result was proved by V.~Eberhard in 1891.

\begin{theorem}
\label{Ebth}
For any sequence of nonnegative integers $p_k$, $k\ge3$, $k\ne6$, satisfying the identity~\eqref{pkformula}, there exists an integer $p_6$ and a simple $3$-polytope $P$ whose number of $k$-gonal facets is~$p_k$.
\end{theorem}

Let $P$ be a simple $3$-polytope given by a system of inequalities~\eqref{ptope}. Each edge $E$ of $P$ is an intersection of two facets, and each facet is defined by setting one of the inequalities $\langle\mb a_i,\mb x\rangle+b_i\ge0$ to equality. Therefore, the edge $E=F_i\cap F_j$ can be specified in $P$ by a single equality $\langle\mb a_i+\mb a_j,\mb x\rangle+(b_i+b_j)=0$.
\begin{figure}[h]
\begin{center}
\includegraphics[scale=0.5]{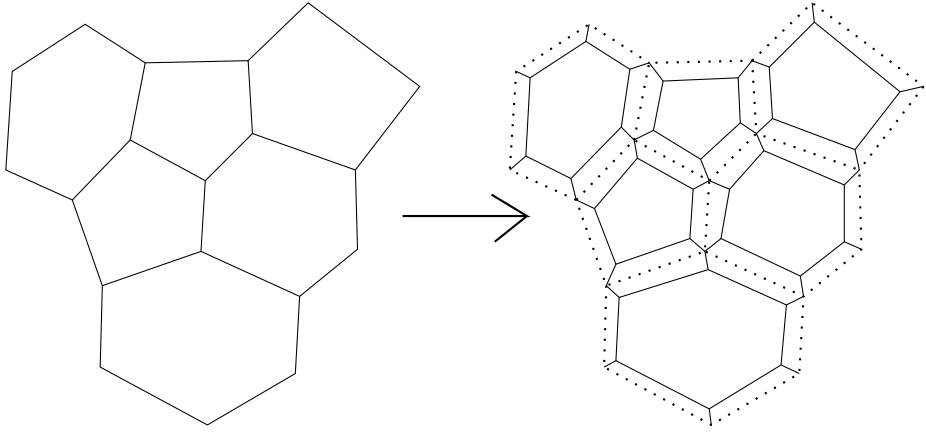}
\end{center}
\caption{Construction of the polytope $P_E$}
\label{ecuts}
\end{figure}

\begin{construction}[see~\cite{l-m-r06, bu-er15}]
Let $P_E$ be the polytope obtained by simultaneous cutting off all edges of a polytope~$P$:
$$
  P_E=P\cap\{\mb x\in\mathbb R^n\colon \langle\mb a_i+
  \mb a_j,\mb x\rangle+(b_i+b_j)\ge\varepsilon\quad
  \text{for all edges } F_i\cap F_j\ne\varnothing\text{ of }P\},
$$
where $\varepsilon>0$ is small enough, see Figure~\ref{ecuts}. Each facet of $P_E$ corresponds either to a facet of $P$ with the same number of edges, or to an edge of~$P$, and in the latter case the facet of $P_E$ is a hexagon. Furthermore,
\begin{itemize}
\item[--] the facets of $P_E$ corresponding to facets of the original polytope~$P$ do not intersect;

\item[--] the facets of $P_E$ corresponding to a facet and an edge of $P$ intersect if and only if the edge is contained in the facet;

\item[--] the facets of $P_E$ corresponding to edges of $P$ intersect if and only if the edges are incident.
\end{itemize}

We therefore obtain
$$
  p_k(P_E)=\begin{cases}p_k(P),& k\ne 6,\\p_6(P)+f_1(P),&k=6,\end{cases}
$$
where $f_1(P)$ is the number of edges of $P$.
\end{construction}

In general, Eberhard's theorem only guarantees the existence of a polytope with some $p_6$ hexagonal facets. The above construction of edge cutting gives infinitely many possible values of~$p_6$. We are interested in the case $p_3=p_4=0$, where the following result of Gr\"unbaum holds:

\begin{theorem}[\cite{grue68}]\label{Gbth}
For any sequence of nonnegative integers $p_k$, $k\ge3$, $k\ne6$, satisfying the conditions \eqref{pkformula}, $p_3=p_4=0$ and $p_6\geqslant 8$, there exists a simple $3$-polytope $P$ whose number of $k$-gonal facets is~$p_k$.
\end{theorem}

\begin{proposition}\label{34cut}
Let $P$ be a simple $3$-polytope with $p_3=p_4=0$. Then $P_E\in\mathcal{P}$.
\end{proposition}
\begin{proof}
We use the criterion of Proposition~\ref{34bp}~(b).

Choose a pair of adjacent facets of $P_E$ and consider the corresponding edges and facets in~$P$.

If the chosen pair corresponds to a facet $F$ and an edge $E\subset F$ of~$P$, then the sequence of facets of $P_E$ around this pair of facets corresponds to edges in $\partial F\setminus E$, the facet $G$ satisfying $F\cap G=E$, and the two edges of $G$ adjacent to~$E$, see Figure~\ref{fbelts}~a). Since $G$ is not a triangle and the pair of adjacent facets $F$ and $G$ is bounded by a simple edge cycle, it follows easily that the cyclic sequence of facets of $P_E$ around the chosen pair of facets is a belt.

If the chosen pair of facets of $P_E$ corresponds to a pair of adjacent edges $E_i$ and $E_j$ of~$P$, then the chosen pair of facets is surrounded by eight facets: the facets corresponding to the facets $F_i$, $F_j$ and $F_k$ of~$P$ meeting at the vertex $E_i\cap E_j\cap E_k$, and the facets corresponding to the edges which are incident to at least one of $E_i$ and $E_j$, see Figure~\ref{fbelts}~b). Each of the facets corresponding to $F_i$, $F_j$ and $F_k$ intersects only two facets out of eight, namely those corresponding to the edges contained in the facet. Since the three facets $F_i$, $F_j$ and $F_k$ are bounded by a simple edge cycle, have a common vertex, and none of the facets is a triangle or quadrangle, it follows easily that the eight facets form a belt.

Thus, in either case the chosen pair of facets is surrounded by a belt.
\begin{figure}[h]
\begin{center}
\includegraphics[scale=0.6]{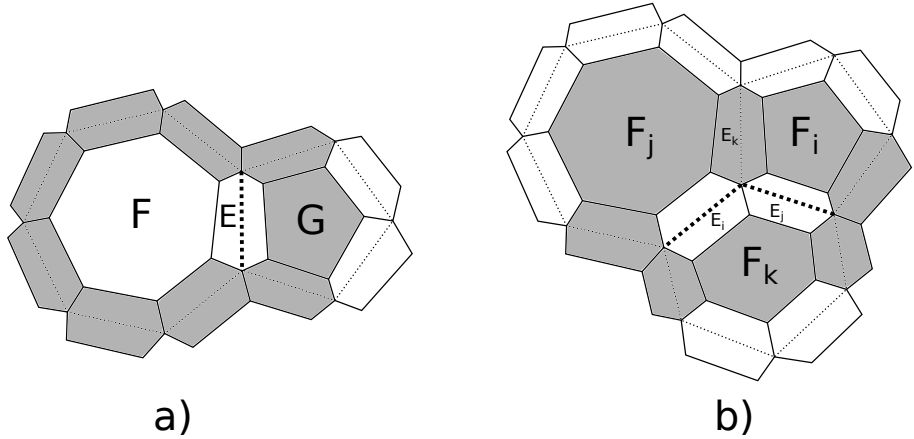}
\end{center}
\caption{Belts around the pairs of adjacent facets of $P_E$}
\label{fbelts}
\end{figure}
\end{proof}

\begin{corollary}\label{Pogpk}
Assume given a sequence of nonnegative integers $p_k$, $k\ge3$, satisfying the following conditions:
\begin{itemize}
\item[(a)] the relation~\eqref{pkformula};
\item[(b)] $p_3=p_4=0$;
\item[(c)] $\frac{p_6-3\left(10+\sum_{k\geqslant 7}(k-5)p_k\right)}{4}$ is an integer~$\ge8$.
\end{itemize}
Then there exists a Pogorelov polytope whose number of $k$-gonal facets is~$p_k$.
\end{corollary}
\begin{proof}
By Theorem \ref{Gbth}, for the given $p_k$, $k\ne 6$, there exists a simple $3$-polytope $P$ whose number of $k$-gonal facets is $p_k$ for $k\ne6$ and whose number of hexagonal facets is $p_6(P)=\frac{p_6-3\left(10+\sum_{k\geqslant 7}(k-5)p_k\right)}{4}\geqslant 8$. Then the polytope $P_E$ has the same numbers $p_k$ for $k\ne 6$, while $p_6(P_E)=p_6(P)+f_1(P)$. Furthermore, $P_E\in\mathcal{P}$ by Proposition~\ref{34cut}. Using the Euler formula and identity~\eqref{pkformula} we calculate
\begin{multline*}
p_6(P_E)=p_6(P)+f_1(P)=p_6(P)+3(f_2(P)-2)=p_6(P)+3\Bigl(\sum\limits_{k\geqslant 5}p_k(P)-2\Bigr)\\
=4p_6(P)+3\Bigl(p_5+\sum\limits_{k\geqslant 7}p_k-2\Bigr)=4p_6(P)+3\Bigl(10+\sum\limits_{k\geqslant 7}(k-5)p_k\Bigr)=p_6.\qedhere
\end{multline*}
\end{proof}

\begin{corollary}\label{Pogpk2}
For any finite sequence of nonnegative integers $p_k$, $k\geqslant 7$, there exists a Pogorelov polytope whose number of $k$-gonal facets is~$p_k$.
\end{corollary}

\section{Proof of Lemma~\ref{proddeco}.}\label{proofproddeco}
Here we give a proof which is different from the original proof of~\cite{fa-wa}. It uses a reformulation of the description of product in the cohomology of a moment-angle complex (Theorem~\ref{prodfsc}) in terms of the polytope~$P$. A detailed description of this approach can be found in~\cite[\S5.8]{bu-erS}.

We need to prove that the product map
\begin{equation}\label{prodK}
  \bigoplus_{I=I_1\sqcup I_2}\widetilde H^0(\sK_{I_1})\otimes\widetilde H^0(\sK_{I_2})\to
  \widetilde H^1(\sK_{I})
\end{equation}
is surjective  for any flag $3$-polytope $P$ and $I\subset[m]$. We
first restate this in terms of the polytope $P$ rather than its
dual simplical complex~$\sK$. The decomposition of $\partial P$
into facets $F_1,\ldots,F_m$ defines a cell decomposition of
$\partial P$ which is Poincar\'e dual to the simplicial
decomposition~$\sK$. The two decompositions have the same
barycentric subdivision, $(\partial P)'\cong\sK'$. We identify the
set of facets $\{F_1,\ldots,F_m\}$ with $[m]$, and for each
$I\subset[m]$ define
\[
  P_I=\bigcup_{i\in I}F_i\subset\partial P.
\]
Note that $P_I$ is the combinatorial neighbourhood of $(\sK_I)'$
in $\sK'$, so there is a deformation retraction
$P_I\stackrel\simeq\longrightarrow\sK_I$. We have Poincar\'e
duality isomorphisms
\begin{equation}\label{PDunred}
  H_{2-i}(P_I,\partial P_I)\cong H^i(\sK_I), \quad i=0,1,2,
\end{equation}
where the boundary $\partial P_I$ consists of points $x\in P_I$
such that $x\in F_j$ for some $j\notin I$. Topologically, $P_I$ is
a disjoint union of several discs with holes, and $\partial P_I$
is a disjoint union of edge cycles.

The cellular homology groups $H_i(P_I,\partial P_I)$ have the
following description. Let $P_I=P_{I^1}\sqcup\cdots\sqcup P_{I^s}$
be the decomposition into connected components. Then
\begin{itemize}
\item[(a)] $H_2(P_I,\partial P_I)$ is a
free abelian group with basis of homology classes
$[P_{I^k}]=\sum_{i\in I^k}[F_i]$, $k=1,\ldots,s$;

\item[(b)] $H_1(P_I,\partial P_I)=\bigoplus\limits_{k=1}^s H_1(P_{I^k},\partial P_{I^k})$,
where $H_1(P_{I^k},\partial P_{I^k})$ is a free abelian group of
rank one less the number of cycles in~$\partial P_{I^k}$. A basis
of $H_1(P_{I^k},\partial P_{I^k})$ is given by any set of edge
paths in $P_{I^k}$ connecting one fixed boundary cycle with the other
boundary cycles.

\item[(c)] $H_0(P_I,\partial P_I)=\Z$ if $I=[m]$, and $0$ otherwise.
\end{itemize}

As the product map~\eqref{prodK} is stated in terms of the reduced
cohomology groups $\widetilde H^i(\sK_I)$, we introduce the
corresponding ``reduced'' homology groups
\[
  \widehat H_i(P_I,\partial P_I)=
  \begin{cases}
    H_i(P_I,\partial P_I), & i=0,1;\\
    H_2(P_I,\partial P_I)\big/\bigl(\sum_{i\in I}[F_i]\bigr), &
    i=2.
  \end{cases}
\]
Then we can rewrite~\eqref{PDunred} as
\begin{equation}\label{PDred}
  \widehat H_{2-i}(P_I,\partial P_I)\cong\widetilde H^i(\sK_I), \quad
  i=0,1,2.
\end{equation}
With this interpretation in mind, we can rewrite the product
map~\eqref{prodK} as the ``intersection pairing''
\begin{equation}\label{prodP}
\begin{aligned}
  \bigoplus_{I=I_1\sqcup I_2}\widehat H_2(P_{I_1},\partial P_{I_1})\otimes
  \widehat H_2(P_{I_2},\partial P_{I_2})&\to
  \widehat H_1(P_I,\partial P_I),\\
  [P_{I_1^p}]\otimes[P_{I_2^q}]
  &\mapsto[P_{I_1^p}\cap
  P_{I_2^q}]=[\gamma_1]+\cdots+[\gamma_r],
\end{aligned}
\end{equation}
where $P_{I_1^p}$ is a connected component of $P_{I_1}$,
$P_{I_2^q}$ is a connected component of $P_{I_2}$, and
$\gamma_1,\ldots,\gamma_r$ are edge paths in $P$ which form the
connected components of the intersection $P_{I_1^p}\cap
P_{I_2^q}$. (There is a sign involved in the transition
from~\eqref{prodK} to~\eqref{prodP}, but it does not affect our
subsequent considerations.)

\begin{proof}[Proof of Lemma~\ref{proddeco}]
To see that \eqref{prodP} is surjective for a flag $3$-polytope $P$, we recall
that $\widehat H_1(P_I,\partial P_I)=\bigoplus\limits_{k=1}^s
\widehat H_1(P_{I^k},\partial P_{I^k})$  and consider for each
connected component $P_{I^k}$ of $P_I$ the decomposition $\partial
P_{I^k}=\eta_1\sqcup\dots\sqcup \eta_{t_k}$ into boundary cycles.
We may assume that $t_k\ge2$, as otherwise $P_{I^k}$ is a disc and
$\widehat H_1(P_{I^k},\partial P_{I^k})=0$. For each pair of
boundary cycles $\eta_p$ and $\eta_q$ among
$\eta_1,\ldots,\eta_{t_k}$, we shall decompose the generator
$g_{pq}$ of $\widehat H_1(P_{I^k},\partial P_{I^k})$ corresponding
to an edge path from $\eta_p$ to $\eta_q$ into a product of
elements of $\widehat H_2(P_{I_1},\partial P_{I_1})$ and $\widehat
H_2(P_{I_2},\partial P_{I_2})$, $I_1\sqcup I_2=I$. This will prove
the surjectivity of~\eqref{prodP}.

We choose facets $F_p$ and $F_q$ in $\partial P\setminus P_{I^k}$
adjacent to $\eta_p$ and $\eta_q$ respectively, see
Figure~\ref{a-b-belt}.
\begin{figure}[h]
\begin{center}
\includegraphics[width=.9\textwidth]{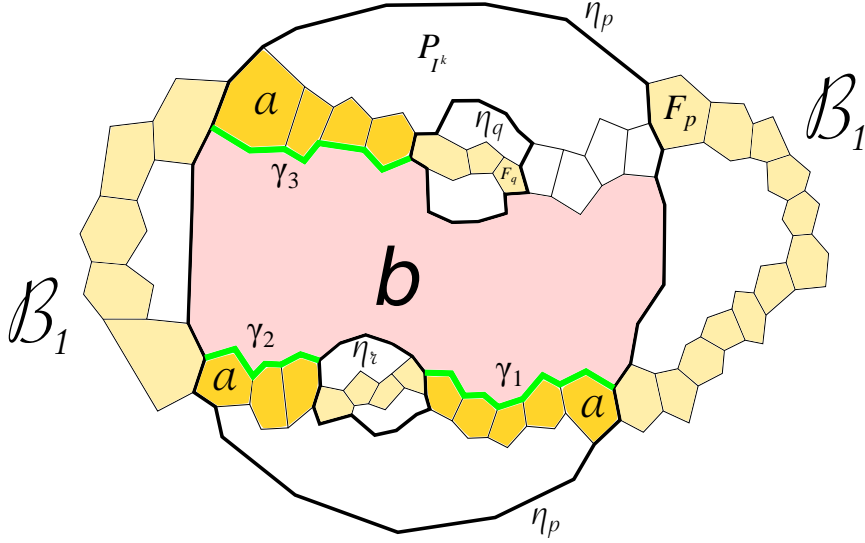}
\end{center}
\caption{A belt crossing a disc with holes.}\label{a-b-belt}
\end{figure}
By Lemma~\ref{Fab}, there is a belt $\mathcal{B}=(F_{j_1},\dots,
F_{j_l})$ with $F_{j_1}=F_p$ and $F_{j_r} = F_q$, where
$3\leqslant r\leqslant l-1$. Let $\mathcal B_1=(F_{j_1},\dots,
F_{j_r})$ be a part of the belt between $F_p$ and $F_q$ (there are
two such parts, and we can take any of them). The complement
$\partial P\setminus\mathcal B$ is a union of two open discs. We
denote the closures of these discs by $\mathcal W_1$ and $\mathcal
W_2$; each of them is a union of facets not in~$\mathcal B$. Now
set
\begin{align*}
  &I_1=\{i\colon F_i\in P_{I^k}\cap\mathcal{B} \},&& I_2=I\setminus I_1,\\
  &a=\sum\limits_{F_i\in P_{I^k}\cap \mathcal B_1}[F_i]\in
  \widehat{H}_2(P_{I_1},\partial P_{I_1}), &&
  b=\sum\limits_{F_j\in P_{I^k}\cap \mathcal{W}_1}[F_j]\in
  \widehat{H}_2(P_{I_2},\partial P_{I_2}).
\end{align*}
Then $a\cdot b=[\gamma_1]+\cdots+[\gamma_s]\in H_1(P_I,\partial
P_I)$, where each $\gamma_i$ is an edge path in $P_{I^k}$ that
begins at some boundary cycle $\eta_{j_{i-1}}$ and ends at
$\eta_{j_i}$. We may assume that $\gamma_1$ begins at~$\eta_p$ and
$\gamma_s$ ends at~$\eta_q$ (where $\eta_p,\eta_q$ is the pair of
boundary cycles chosen above). The homology class
$[\gamma_1]+\cdots+[\gamma_s]\in\widehat H_1(P_I,\partial P_I)$ is
then equal to the chosen generator $g_{pq}$ of $\widehat
H_1(P_{I^k},\partial P_{I^k})$ corresponding to an edge path
from $\eta_p$ to $\eta_q$. We have therefore decomposed $g_{pq}$
into a product $a\cdot b$, as needed.
\end{proof}

\section{Proof of Lemma~\ref{rig3cl}}\label{proofrig3cl}
The proof uses the combinatorial result of Lemma~\ref{Fabc} and an
algebraic ``annihilator lemma'' of Fan, Ma and Wang~\cite{f-m-w}.

Recall that the \emph{annihilator} of an element $r$ in a ring $R$
is defined as
\[
  \Ann_R(r)=\{s\in R\colon rs=0\}.
\]

\begin{lemma}[{\cite[Lemma 3.3]{f-m-w}}]\label{annlemma}
Let $P$ be a $3$-polytope from the Pogorelov class~$\mathcal P$, with the dual
complex $\sK=\sK_P$. Let $R=H^*(\zp;\mathbf k)$, where $\mathbf k$ is a field. In the notation of
Lemma~\ref{rig3cl}, consider a $\mathbf k$-linear combination of elements
of $\mathcal T(P)$,
\[
  \alpha=\sum_{\{i,j\}\notin\sK}r_{ij}[u_iv_j]
\]
with at least two nonzero $r_{ij}\in\mathbf k$. Then, for any $\{k,l\}$
such that $r_{kl}\ne0$,
\[
  \dim\Ann_R[u_kv_l]>\dim\Ann_R\alpha.
\]
\end{lemma}
\begin{proof}
In view of the isomorphisms~\eqref{PDred}, we can rewrite the isomorphism of Theorem~\ref{prodfsc} as
\[
  R=H^*(\zp)\cong\bigoplus_{I\subset[m]}\widehat H_*(P_I,\partial P_I)
\]
(we omit the coefficient field $\mathbf k$ in the notation for homology).

Take a complementary subspace $L_{kl}$ to $\Ann_R[u_kv_l]$ in~$R$,
so that $L_{kl}\oplus\Ann_R[u_kv_l]=R$.  For any $\beta\in
L_{kl}\setminus\{0\}$ we have $\beta\cdot[u_kv_l]\ne0$. Furthermore, we can choose $L_{kl}$ respecting the
multigrading, so that the $I$th multigraded component of $L_{kl}$ is a complementary subspace to
$\Ann_R[u_kv_l]\cap\widehat H_*(P_I,\partial P_I)$ in $\widehat H_*(P_I,\partial P_I)$.
Then we
can write $\beta=\sum_{I\subset[m]\setminus\{k,l\}}\beta_I$, where
$\beta_I$ denotes the $I$th multigraded component of~$\beta\in
L_{kl}\setminus\{0\}$. (Note that $\beta_I=0$ whenever $I\cap\{k,l\}\ne\varnothing$, as such $\beta_I$ would annihilate $[u_kv_l]$.) We
can choose $I\subset[m]\setminus\{k,l\}$ such that
$\beta_I\cdot[u_kv_l]\ne0$. Now consider $\alpha=\sum
r_{ij}[u_iv_j]$. We claim that the $(I\cup\{k,l\})$th multigraded
component of $\beta\cdot\alpha$ consists of $\beta_I\cdot[u_kv_l]$
only. Indeed, for any other component $\beta_{I'}$ of $\beta$ with
$I'\ne I$ and any summand $r_{ij}[u_iv_j]$ of $\alpha$, we have
$I'\cup \{i,j\}\ne I\cup\{k,l\}$, as $I'\in[m]\setminus\{k,l\}$.
Then $(\beta\cdot\alpha)_{I\cup\{k,l\}}=\beta_I\cdot[u_kv_l]\ne0$.
Hence, $L_{kl}\cap\Ann_R\alpha=\{0\}$, which implies that
$\dim\Ann_R[u_kv_l]\ge\dim\Ann_R\alpha$.

In order to show that the strict inequality holds, we shall find
an element $\xi\in \Ann_R[u_kv_l]$ such that
$(L_{kl}\oplus\langle\xi\rangle)\cap\Ann_R\alpha=\{0\}$. Take a
summand $r_{st}[u_sv_t]$ of $\alpha$ different from
$r_{kl}[u_kv_l]$. That is, $\{s,t\}\ne\{k,l\}$ and $r_{st}\ne0$.
We can assume without loss of generality that $l\notin\{s,t\}$. By
Lemma~\ref{Fabc}, there is a belt $\mathcal{B}$ in $P$ such that
$F_s,F_t\in \mathcal{B}$, $F_l\notin\mathcal{B}$, and $F_l$ does
not intersect one of the two connected components $B_1$ and $B_2$
of $\mathcal{B}\setminus\{F_s,F_t\}$, say~$B_1$. In the dual
language, there is a chordless cycle $\mathcal C$ in $\sK_P$ such
that $s,t\in\mathcal C$, $l\notin\mathcal C$, and the vertex $l$
is not joined by an edge to any vertex of the connected component
$L_1$ of $\mathcal C\setminus\{s,t\}$.

Now we observe that $\mathcal C\setminus\{s,t\}$ is a full
subcomplex of $\sK_P$ and take $\xi$ to be the cohomology class in
$R=H^*(\zp)$ given by a generator of $\widetilde H^0(\mathcal
C\setminus\{s,t\})\cong\Z$. Such a generator is represented by the
0-cocycle $\sum_{i\in L_1}\alpha_{\{i\}}$ (see
Example~\ref{exacochain}). We have $\xi\cdot[u_kv_l]=0$ because we
can write $\xi=\sum_{i\in L_1}\pm[u_{J_i}v_i]$ (see
Example~\ref{exacochain}) and $v_iv_l=0$ for any $i\in L_1$ by the
choice of the cycle~$\mathcal C$. On the other hand, the product
$\xi\cdot[u_sv_t]$ corresponds to a generator of $\widetilde
H^1(\mathcal C)\cong\Z$. Therefore, $\xi\in \Ann_R[u_kv_l]$ and
$\xi\cdot\alpha\ne0$ (the latter is because the multigraded
component of $\xi\cdot\alpha$ corresponding to $\mathcal C$ is
$\xi\cdot r_{st}[u_sv_t]\ne0$). Take
$\beta=\sum_{I\subset[m]\setminus\{k,l\}}\beta_I\in
L_{kl}\setminus\{0\}$ and choose $I\subset[m]\setminus\{k,l\}$
such that
$(\beta\cdot\alpha)_{I\cup\{k,l\}}=\beta_I\cdot r_{kl}[u_kv_l]\ne0$, as
in the beginning of the proof. The multigrading of $\xi$ does not
contain $l$, so we have
$(\xi\cdot\alpha)_{I\cup\{k,l\}}=\xi\cdot r_{jl}[u_jv_l]$ for some
$j\in[m]$. Now, $\xi\cdot r_{jl}[u_jv_l]=0$ because $\xi=\sum_{i\in
L_1}\pm[u_{J_i}v_i]$ and $v_iv_l=0$ for any $i\in L_1$, as $i$ and $l$ are not joined by an edge. Hence,
$((\beta+\xi)\cdot\alpha)_{I\cup\{k,l\}}=(\beta\cdot\alpha)_{I\cup\{k,l\}}\ne
0$. Thus, $(\beta+\xi)\cdot \alpha\ne0$ and we have proved that
$(L_{kl}\oplus\langle\xi\rangle)\cap\Ann_R\alpha=\{0\}$. This
implies that $\dim\Ann_R[u_kv_l]>\dim\Ann_R\alpha$.
\end{proof}

\begin{proof}[Proof of Lemma~\ref{rig3cl}]
We are given a $3$-polytope $P$ from the Pogorelov class~$\mathcal P$ and a ring
isomorphism $\psi\colon R=H^*(\zp)\stackrel\cong\longrightarrow
H^*(\mathcal Z_{P'})=R'$. We defined the set
\[
  \mathcal T(P)=\{\pm[u_iv_j]\in H^3(\zp),\quad F_i\cap
  F_j=\varnothing\},
\]
and the corresponding set for $P'$,
\[
  \mathcal T(P')=\{\pm[u'_iv'_j]\in H^3(\mathcal Z_{P'}),\quad F'_i\cap
  F'_j=\varnothing\}.
\]
We need to show that $\psi(\mathcal T(P))=\mathcal T(P')$, in
other words, $\psi([u_pv_q])=\pm[u'_rv'_s]$. We first use
Theorems~\ref{flagrig} and~\ref{4beltrig} to conclude that $P'$ also belongs to the class~$\mathcal P$. Now suppose that
$\psi([u_pv_q])=\alpha'=\sum r_{ij}[u'_iv'_j]$ with at least two
nonzero $r_{ij}$. We are then in the situation of
Lemma~\ref{annlemma}, which we can apply to~$P'$. We obtain that
$\dim\Ann_{R'}\alpha'<\dim\Ann_R[u'_kv'_l]$ for any nonzero summand
$r_{kl}[u'_kv'_l]$ of~$\alpha'$. Considering the inverse
isomorphism $\psi^{-1}\colon R'\to R$, we can choose $[u'_kv'_l]$
such that $\psi^{-1}([u'_kv'_l])=\alpha=\sum r_{ab}[u_av_b]$ where
$[u_pv_q]$ appears in the latter sum. As an isomorphism preserves
the dimension of the annihilator subspace, we obtain
\begin{multline*}
  \dim\Ann_R[u_pv_q]=\dim\Ann_{R'}\alpha'<\dim\Ann_R[u'_kv'_l]=
  \dim\Ann_R\alpha\\<\dim\Ann_R[u_pv_q],
\end{multline*}
which is a contradiction. It follows that $\psi([u_pv_q])$ is a
multiple of a single $[u'_rv'_s]$. Since $\psi$ is an
isomorphism over~$\mathbb Z$, we have $\psi([u_pv_q])=\pm[u'_rv'_s]$.
\end{proof}

\end{appendix}

\end{document}